\documentclass[12pt,a4paper]{article}

\usepackage{mathrsfs}
\usepackage{graphicx}
\usepackage[a4paper, hmargin={2.5cm,2.5cm},vmargin={3.3cm,3.3cm}]{geometry}
\usepackage{amsmath, amssymb, amsfonts, amsbsy, latexsym, color,dsfont}
\usepackage{amscd,amsxtra,array}
\usepackage{amsthm}
\usepackage{mathtools}
\usepackage{booktabs}
\usepackage[normalem]{ulem}
\usepackage[T1]{fontenc}
\usepackage[latin1]{inputenc}
\usepackage[english]{babel}
\usepackage{lmodern}
\usepackage{cite}
\usepackage{url}

\newtheorem{theorem}{Theorem}[section]

\newtheorem{lemma}[theorem]{Lemma}
\newtheorem{proposition}[theorem]{Proposition}
\newtheorem{corollary}[theorem]{Corollary}

\newtheorem{remark}[theorem]{Remark}

\theoremstyle{definition}
\newtheorem{definition}[theorem]{Definition} %[section]
\newtheorem{example}[theorem]{Example}

\usepackage{paralist}
\usepackage[bf,compact,pagestyles,small]{titlesec}%topmarks,calcwidth
\titlespacing*{\section}{0pt}{14pt}{4pt}
\titlespacing*{\subsection}{0pt}{8pt}{3pt}
\usepackage{mathdots}

\usepackage{fancyhdr,lastpage}

\fancypagestyle{plain}{% 
\fancyhf{} % clear all header and footer fields 
\fancyfoot[C]{\small \thepage{} of \pageref{LastPage}}}

\numberwithin{equation}{section}  % OR chapter
\allowdisplaybreaks[4]

\DeclareMathOperator{\supp}{supp} %

\newcommand*{\numbersys}[1]{\ensuremath{\mathbb{#1}}}
\newcommand*{\C}{\numbersys{C}}
\newcommand*{\R}{\numbersys{R}}
\newcommand*{\Q}{\numbersys{Q}}

\newcommand*{\Z}{\numbersys{Z}}

\newcommand*{\N}{\numbersys{N}}
\newcommand*{\T}{\numbersys{T}}

\newcommand*{\cF}{\mathcal{F}}

\usepackage{xspace}
\newcommand{\ie}{i.e., } %{\textit{i.e.,\xspace}}
\newcommand{\eg}{e.g., } %{\textit{e.g.\@\xspace}}

\newlength{\dhatheight}
\newcommand{\doublehat}[1]{%
	\settoheight{\dhatheight}{\ensuremath{\widehat{#1}}}%
	\addtolength{\dhatheight}{-0.35ex}%
    	\widehat{\vphantom{\rule{1pt}{\dhatheight}}%
    	\smash{\hspace{-2pt} \widehat{#1}}}}

\newlength{\dtildeheight}

\newcommand{\ghat}{\widehat{G}}
\newcommand{\ghhat}{\doublehat{G}}

\newcommand{\SO}{\textnormal{\textbf{S}}_{0}}
\newcommand{\SOprime}{\textnormal{\textbf{S}}'_{0}}

{\nopagebreak\par\noindent\hbox to \hsize{\hrulefill}\par\noindent}

\usepackage{hyperref}
\hypersetup{
 pdfview={FitH},
 pdfstartview={FitH},
 pdfauthor = {Mads Sielemann Jakobsen},
 pdftitle = {On a (no longer) new Segal algebra}, 
 pdfsubject = {harmonic analysis},
 pdfkeywords = {Segal, Feichtingers Algebra, Banach spaces},
 pdfcreator = {LaTeX with hyperref package},
 pdfproducer = PDFlatex}

\makeatletter
\def\blfootnote{\xdef\@thefnmark{}\@footnotetext} 
\def\subjclass{\xdef\@thefnmark{}\@footnotetext}
\long\def\symbolfootnote[#1]#2{\begingroup%
\def\thefootnote{\fnsymbol{footnote}}\footnote[#1]{#2}\endgroup} 
\if@titlepage
  \renewenvironment{abstract}{%
      \titlepage
      \null\vfil
      \@beginparpenalty\@lowpenalty
      \begin{center}%
        \bfseries \abstractname
        \@endparpenalty\@M
      \end{center}}%
     {\par\vfil\null\endtitlepage}
\else
  \renewenvironment{abstract}{%
      \if@twocolumn
        \section*{\abstractname}%
      \else
        \small
        \list{}{%
          \settowidth{\labelwidth}{\textbf{\abstractname:}}
          \setlength{\leftmargin}{50pt}
          \setlength{\rightmargin}{50pt}
          \setlength{\itemindent}{\labelwidth}
          \addtolength{\itemindent}{\labelsep}
        }
        \item[\textbf{\abstractname:}]

      \fi}
      {\if@twocolumn\else\endlist\fi}
\fi
\makeatother

\begin{document}

\title{On a (no longer) new Segal algebra \\ --- a review of the Feichtinger algebra
}

\date{\today}

 \author{Mads Sielemann Jakobsen\footnote{NTNU, Department of Mathematical Sciences, Trondheim, Norway, E-mail: \protect\url{mads.jakobsen@ntnu.no}}}

\blfootnote{2010 {\it Mathematics Subject Classification.} Primary 43A15  Secondary 43-02} \blfootnote{{\it Key words.} Feichtinger algebra, test functions, generalized functions,  Schwartz-Bruhat space}

\maketitle

\thispagestyle{plain}

\begin{abstract} 
Since its invention in 1979, the Feichtinger algebra has become a very useful Banach space of functions with applications in time-frequency analysis, the theory of pseudo-differential operators and several other topics. It is easily defined on locally compact abelian groups and, in comparison with the Schwartz(-Bruhat) space, the Feichtinger algebra allows for more general results with easier proofs. 
This review paper gives a linear and comprehensive deduction of the theory of the Feichtinger algebra and its favourable properties. The material gives an entry point into the subject, but it will also give new insight to the expert. 
A main goal of this paper is to show the equivalence of the many different characterizations of the Feichtinger algebra known in the literature. This task naturally guides the paper through basic properties of functions that belong to the space, over operators on it and to aspects of its dual space. Further results include a seemingly forgotten theorem by Reiter on operators which yield Banach space isomorphisms of the Feichtinger algebra; a new identification of the Feichtinger algebra as the unique Banach space in $L^{1}$ with certain properties; and the kernel theorem for the Feichtinger algebra. A historical description of the development of the theory, its applications and related function space constructions is included.
\end{abstract}

\section{Introduction}

In 1979 at the ``International Workshop on Topological Groups and Group Algebras'' at the University of Vienna, H.\ G.\ Feichtinger presented the function space $\SO$ as ``a new Segal algebra with applications in harmonic analysis'' \cite{fe79-5}. Further results on $\SO$ followed in subsequent papers by Feichtinger   
\cite{fe80}, Losert \cite{lo80}, and Poguntke \cite{po80-1}. Together with the significant ``On a new Segal algebra'' paper by Feichtinger \cite{Feichtinger1981} these publications showed many important properties of $\SO$, and gave it a firm footing in the zoo of function spaces. 

In order to define the space, let  
$G$ be a locally compact (Hausdorff) abelian group (\eg the Euclidean space, a discrete abelian group, the torus, the field of $p$-adic numbers or the adele ring of the rational field $\Q$) and let $\ghat$ be its dual group. Furthermore, let $E_{\omega}$ denote multiplication by an element $\omega \in \ghat$, i.e., $E_{\omega}f(x) = \omega(x) f(x)$ for $f\in L^{1}(G)$. The space $\SO(G)$ consists exactly of all integrable functions $f\in L^{1}(G)$ that satisfy 
\begin{equation} \label{eq:0201afly} \int_{\ghat} \Vert E_{\omega} f * f \Vert_{1} \, d\omega < \infty.\end{equation}
Fix any non-zero function $g\in \SO(G)$, then $\Vert f \Vert_{\SO} = \int_{\ghat} \Vert E_{\omega} f * g \Vert_{1} \, d\omega$ is a norm on $\SO(G)$. This defines a Banach space of continuous functions, which nowadays is called the \emph{Feichtinger algebra}.\footnote{With the definition of $\SO(G)$ given here, there are already some immediate questions: Why are functions $f\in L^{1}(G)$ that satisfy \eqref{eq:0201afly} continuous? Why is the set of functions satisfying \eqref{eq:0201afly} a subspace of $L^{1}(G)$? Why is the integral in the norm well-defined?} 
%We will see in a moment that $\SO(G)$ has become a useful and versatile Banach space.

As it turns out, the Feichtinger algebra is a useful replacement for the Schwartz space of smooth functions on the Euclidean space $\R^n$ with rapid decay of all derivatives  
\cite{MR0209834,MR1996773}, and, especially, for its cumbersome generalization to 
locally compact abelian groups, the Schwartz-Bruhat 
space $\mathcal{S}(G)$ \cite{br61,os75}. Indeed, the Feichtinger algebra is, just as the Schwartz-Bruhat space,  invariant under automorphisms, invariant under the Fourier transform, the Poisson formula is valid, the space has the tensor factorization property (cf.\ Theorem~\ref{th:tensor-property-s0}) and it has a kernel theorem (cf.\ Lemma~\ref{le:SO-tensor-operator} and 
Theorem~\ref{th:SO-kernel-theorem}). 
Hence the Feichtinger algebra exhibits many properties which make the Schwartz-Bruhat space convenient to work with. In fact, one can show that the Feichtinger algebra contains the Schwartz-Bruhat space. These properties of $\SO$, together with the fact that it is a Banach space\footnote{If $G$ is an \emph{elementary} locally compact abelian group, i.e., $G\cong \R^{l}\times \Z^{m}\times \T^{n}\times F$ for some finite abelian group $F$ and with $l,m,n\in\N_{0}$, then the Schwartz-Bruhat space $\mathcal{S}(G)$ is a Fr\'echet space. Otherwise, the Schwartz-Bruhat space is a locally convex inductive limit of Fr\'echet spaces, i.e., a so-called LF-space.}, makes it possible, in comparison with the Schwartz-Bruhat space, to prove more general results with less involved proofs.

The original paper \cite{Feichtinger1981} as well as \cite{fezi98} and the relevant book chapters in \cite{MR1843717} and \cite{re89} comprise the core results of the established theory of $\SO$ on $\R^{n}$ and on general locally compact abelian groups. Other literature on $\SO$, sometimes focusing on specific aspects, include\cite{fe79-5,fe80,fe89-2,fe03-1,fegr92-1,feko98,go11,lo80,po80-1} and \cite{ho89,ma87,qu89}. 

The main intention of this paper is to give a comprehensive and linear deduction with proofs of many of the key properties and characterizations of $\SO$ shown over the time. In this way, the paper aims at popularizing the Feichtinger algebra and to make the space $\SO$ and its theory accessible to a large group of readers. Without a doubt, the Feichtinger algebra is mostly used for analysis on $\R^{n}$, in particular in time-frequency analysis. However, beyond the Euclidean space there are other locally compact abelian groups of interest where the Feichtinger algebra hopefully will turn our to be useful. In particular, for analysis on the $p$-adic fields or the adele ring the Feichtinger algebra proposes an, without a doubt, more tractable function space in comparison to the Schwartz-Bruhat space. 

Next to the well-known properties of the Feichtinger algebra that will be presented in this paper, new results include the characterization of $\SO$ given in Theorem \ref{th:s0-ABC}, the inequalities in 
Corollary~\ref{co:f-in-s0-properties}(vi)-(x), the results on operators on $\SO$ in 
Theorem~\ref{th:s0-invariant-under-some-automorphisms} and its application 
in Examples~\ref{ex:unitaries-on-S0}, the results on the short-time Fourier transform as an operator on $\SO(G)$ in Theorem~\ref{th:stft-and-s0}, characterizations of operators defined on $\SO$ or $L^{2}$ that can be extended to the dual space $\SOprime$ in Lemma \ref{le:banach-vs-hilbert-adjoint} and Lemma \ref{le:SOop-to-SOprime}, a new characterization of $\SO$ among Banach spaces in $L^{1}$ in 
Theorem~\ref{th:S0-Banach-Space-Characterization} and the kernel theorem for the Feichtinger algebra in Lemma~\ref{le:SO-tensor-operator} and 
Theorem~\ref{th:SO-kernel-theorem}, which has not been proven in this generality before.

A further objective of this paper is to show that the sets $\mathscr{A}$ to $\mathscr{U}$, below, in Definition~\ref{def:S0-bigdef}, coincide and all characterize $\SO(G)$. We will show that the sets $\mathscr{G}$ to $\mathscr{U}$ induce norms on $\SO(G)$ in a natural way, which in fact are all equivalent. 
Almost all the characterizations in Definition~\ref{def:S0-bigdef} are well known, see, \eg \cite{Feichtinger1981,fezi98,MR1843717}. 
Other characterizations of $\SO$ (which will not be considered here) can be found in
\cite{cogr03-1,ja05,ja06-2}.
%The original definition of $\SO(G)$ used in \cite{fe79,Feichtinger1981} 
%is the one given by the set $\mathscr{L}$, whereas in recent literature the, by far, most used characterization of $\SO(G)$ is the one given by the set $\mathscr{H}$.

An aspect of the theory which is not part of the story told in this paper is that of the weighted Feichtinger algebra, which is commonly studied as a special case of the weighted modulation spaces. For results on this aspect of the theory we refer to, e.g., \cite{fe79,fe03-1,MR1843717,gr07}.  

Next to just mentioned modulation spaces, the Feichtinger algebra is related to a variety of other function spaces. In Section~\ref{sec:history} we give an overview of the history of $\SO(G)$, its relation to other function space constructions and their applications. The remainder of the paper is structured as follows: 

In Section~\ref{sec:prelim} we recall some basic theory on locally compact abelian groups and set the notation and definitions for the rest of the paper.

%The following sections, Sections \ref{sec:basic-prop-of-so}-\ref{sec:minimal}, constitute the main part of this paper and should be read in order. If read, a reader will have a good understanding of the Feichtinger algebra.  

In Section~\ref{sec:basic-prop-of-so} we show properties of functions that belong to $\SO(G)$ and we show that $\SO(G)$ is a Banach algebra with respect to multiplication and convolution. Furthermore, this section establishes that the sets $\mathscr{A}-\mathscr{I}$ from Definition~\ref{def:S0-bigdef} coincide. 

Section~\ref{sec:mappings-on-s0} contains results on operators on $\SO$. In particular, results on the Fourier transform, automorphisms on $G$, metaplectic operators, the short-time Fourier transform, and the restriction and periodization operators with respect to a closed subgroup are considered here. We also prove that the Poisson formula hold for all functions in $\SO$. 

In Section~\ref{sec:s0prime} we consider the dual space of $\SO$. 
We show how the operators on $\SO(G)$ considered in Section~\ref{sec:mappings-on-s0} extend to $\SOprime(G)$ and prove that $\SO(G)$ is weak$^{*}$-dense in $\SOprime$. We also show that $\SO(G)$ can be characterized by the sets $\mathscr{J}$ and $\mathscr{K}$. 

The results from the previous sections allow us to prove key properties of the Feichtinger algebra in Section \ref{sec:minimal}. In particular, this sections contains the  characterization of $\SO$ by series representations given by the set $\mathscr{L}$ and a proof of the minimality of the Feichtinger algebra. Furthermore, this section is concerned with the tensor factorization property of $\SO$, it contains a new characterization of the Feichtinger algebra among all Banach spaces on locally compact abelian groups. which are subspaces of $L^{1}$. Lastly, we also prove a characterization of $\SO$ by use of open subgroups, which is particularly useful for all locally compact abelian groups but the Euclidean space. 

Section \ref{sec:series-expansions} is concerned with the characterization of $\SO(G)$ by the series representation given in the sets $\mathscr{M}$ to $\mathscr{R}$ from Definition \ref{def:S0-bigdef}. 
We also establish the characterization of $\SO(G)$ by bounded uniform partition of unities (BUPUs) given by the sets $\mathscr{T}$ and $\mathscr{U}$. Finally, in Section~\ref{sec:kernel} we show the kernel theorem for the Feichtinger algebra. 

It is recommended to read Sections~\ref{sec:basic-prop-of-so}-\ref{sec:minimal} in order. Sections~\ref{sec:series-expansions} and \ref{sec:kernel} can be read independently of one another.

For later reference we now state the collection of different sets which we will show are all equal and define the Feichtinger algebra.

\begin{definition} \label{def:S0-bigdef} For a locally compact abelian group  $G$ we define the following sets:
\begin{enumerate}[]
\item $ \mathscr{A} = \{ f\in L^{1}(G) \, : \, \exists \, g\in L^{1}(G)\backslash\{0\} \text{ s.t. } \textstyle\int_{\ghat} \Vert E_{\omega} f * g \Vert_{1} \, d\omega < \infty\}$,
\item $  \mathscr{B} =\{ f\in L^{2}(G) \, : \, \exists \, g\in L^{2}(G)\backslash\{0\} \, \ \text{s.t.} \ \int_{G} \int_{\ghat} \vert \mathcal{V}_{g}f(x,\omega) \vert \, d\omega \, dx < \infty\}$,
\item $  \mathscr{C} = \{ f\in A(G) \, : \, \exists \, g\in A(G)\backslash\{0\} \, \ \text{s.t.} \ \int_{G} \Vert T_{x} f \cdot g \Vert_{A(G)} \, dx < \infty \}$,
\item $ \mathscr{D} = \{ f \in L^{1}(G) \, : \, \int_{\ghat} \Vert E_{\omega} f * f \Vert_{1} \, d\omega< \infty\}$,
\item $ \mathscr{E} = \{ f \in L^{2}(G) \, : \, \int_{G} \int_{\ghat} \vert \mathcal{V}_{f}f(x,\omega) \vert \, d\omega \, dx < \infty\}$,
\item $ \mathscr{F} = \{ f \in A(G) \, : \, \int_{G} \Vert T_{x} f \cdot f \Vert_{A(G)} < \infty\}$.
\end{enumerate}
In the next six items, we fix a $g\in\SO(G)\backslash\{0\}$ and define
\begin{enumerate}[]
\item $ \mathscr{G} = \{ f \in L^{1}(G) \, : \, \int_{\ghat} \Vert E_{\omega} f * g \Vert_{1} \, d\omega< \infty\}$,
\item $ \mathscr{H} = \{ f \in L^{2}(G) \, : \, \int_{G} \int_{\ghat} \vert \mathcal{V}_{g}f(x,\omega) \vert \, d\omega \, dx < \infty\}$,
\item $ \mathscr{I} = \{ f \in A(G) \, : \, \int_{G} \Vert T_{x} f \cdot g \Vert_{A(G)} < \infty\}$,
\item $ \mathscr{J} = \{ \sigma \in \SOprime(G) \, : \, \widetilde{\mathcal{V}}_{g} \sigma \in \SO(G\times\ghat)\}$,
\item $ \mathscr{K} = \{ \sigma \in \SOprime(G) \, : \, \widetilde{\mathcal{V}}_{g} \sigma \in L^{1}(G\times\ghat)\}$, 
\item $ \mathscr{L} = \{ f \in L^{1}(G) \, : \, f = \sum_{n\in\N} c_{n} E_{\omega_{n}}T_{x_{n}} g \ \text{with} \ (c_{n})_{n\in \N} \in \ell^{1}(\N), \ (x_{n},\omega_n)_{n\in \N}\subseteq G\times\ghat\}$.
\end{enumerate}
We fix a compact set $K$ in $G$ with non-void interior and define the set
\begin{enumerate}[]
\item $ \mathscr{M} = \{ f \in L^{1}(G) \, : \, f = \sum\limits_{n\in \N} T_{x_{n}} g_{n} , \ (g_{n})_{n\in\N}\subseteq A(G), \ \supp\,g_{n} \subseteq K \ \text{for all} \ n\in\N$ \\
\phantom{m} \hfill with $ (x_{n})_{n\in \N} \subseteq G \text{ and } \sum\limits_{n\in \N} \Vert g_{n} \Vert_{A(G)} < \infty\}$.
\end{enumerate}
We fix a compact set $\tilde{K}$ in $\ghat$ with non-void interior and define the set
\begin{enumerate}[]
\item $ \mathscr{N} = \{ f \in A(G) \, : \, f = \sum\limits_{n\in \N} E_{\omega_{n}} g_{n} , \ (g_{n})_{n\in\N}\subseteq L^{1}(G), \ \supp\,\hat{g}_{n} \subseteq \tilde{K} \ \text{for all} \ n\in\N$ \\
\phantom{m} \hfill with $ (\omega_{n})_{n\in \N} \subseteq \ghat \text{ and } \sum\limits_{n\in \N} \Vert g_{n} \Vert_{1} < \infty\}.$
\end{enumerate}
In the next four items, we fix a function $g\in \SO(G)\backslash\{0\}$ and define
\begin{enumerate}[]
\item $ \mathscr{O} = \{ f \in L^{1}(G) \, : \, f = \sum\limits_{n\in \N} f_{n} * E_{\omega_{n}} g, \ (f_{n})_{n\in\N}\subseteq L^{1}(G), \ (\omega_{n})_{n\in \N} \subseteq \ghat, \sum\limits_{n\in \N} \Vert f_{n} \Vert_{1} < \infty\}$,
\item $ \mathscr{P} = \{ f \in A(G) \, : \, f = \sum\limits_{n\in \N} f_{n} \cdot T_{x_{n}} g, \ (f_{n})_{n\in\N}\subseteq A(G), \ (x_{n})_{n\in \N} \subseteq G, \sum\limits_{n\in \N} \Vert f_{n} \Vert_{A(G)} < \infty\}$,
\item $ \mathscr{Q} = \{ f \in L^{1}(G) \, : \, f = \Big( \sum\limits_{n\in\N} T_{\omega_{n}}\mathcal{V}_{\hat{g}}\hat{f}_{n} \Big) \Big\vert_{\{0\}\times G}, \ (f_{n})_{n\in \N} \subseteq \SO(G), (\omega_{n})_{n\in\N} \subseteq \ghat$ \\
\phantom{m} \hfill with $\sum\limits_{n\in \N} \Vert f_{n} \Vert_{\SO,g} < \infty\} $,
\item $\mathscr{R} = \{ f \in L^{1}(G) \, : \, f = \sum\limits_{n\in\N} f_{n}*g_{n}, \ (f_{n})_{n\in\N},(g_{n})_{n\in\N} \subseteq \SO(G)$ with $\sum\limits_{n\in\N} \Vert f_{n} \Vert_{\SO,g} \, \Vert g_{n} \Vert_{\SO,g} < \infty\} $.
\end{enumerate}
We say a family of functions $(\psi_{i})_{i\in I}\subseteq A(G)$ is a \emph{bounded uniform partition of unity of $G$} if there exists a compact set $W\subseteq G$ and a discrete subset $(x_{i})\subseteq G$ such that
\begin{enumerate}[]
\item [(a.i)] $\sum_{i\in I} \psi_{i} (x) = 1 $ for all $x\in G$,
\item [(a.ii)]$\sup_{i\in I} \Vert \psi_{i} \Vert_{A(G)} < \infty$,
\item [(a.iii)]$\supp \psi_{i} \subseteq x_{i} + W$ for all $i\in I$,
\item [(a.iv)]$\sup_{x\in G} \# \{ i\in I \, : \, (x +K) \cap (x_{i}+ W) \ne \emptyset\} < \infty$ for any compact set $K\subseteq G$.
\end{enumerate} 
If $(\psi_{i})_{i\in I}\subset A(G)$ is a bounded uniform partition of unity of $G$ then we define the set
\begin{enumerate}[]
\item $\mathscr{T} = \{ f \in A(G) \, : \, \sum_{i\in I} \Vert f \psi_{i} \Vert_{A(G)} < \infty\}$.
\end{enumerate}
Similarly, a family of functions $(\varphi_{i})_{i\in I}\subseteq L^{1}(G)$ is a \emph{bounded uniform partition of unity of $\ghat$} if there exists a compact set $V\subseteq \ghat$ and a discrete subset $(\omega_{i})\subseteq \ghat$ such that
\begin{enumerate}[]
\item [(b.i)]$\sum_{i\in I} \hat\varphi_{i} (\omega) = 1 $ for all $\omega\in \ghat$,
\item [(b.ii)]$\sup_{i\in I} \Vert \varphi_{i} \Vert_{1} < \infty$,
\item [(b.iii)]$\supp  \hat\psi_{i} \subseteq \omega_{i} + V$ for all $i\in I$,
\item [(b.iv)]$\sup_{\omega\in \ghat} \# \{ i\in I \, : \, (\omega + K) \cap (\omega_{i}+ V) \ne \emptyset\} < \infty$ for any compact set $K\subseteq \ghat$.
\end{enumerate} 
If $(\varphi_{i})_{i\in I}\subset L^{1}(G)$ is a bounded uniform partition of unity of $\ghat$ we define the set
\begin{enumerate}[]
\item $\mathscr{U} = \{ f \in L^{1}(G) \, : \, \sum_{i\in I} \Vert f *\varphi_{i} \Vert_{1} < \infty\}$.
\end{enumerate}
\end{definition}

\section{Applications and history of the Feichtinger algebra}
\label{sec:history}

This section is concerned with the history of the Feichtinger algebra and its relation to other Banach space constructions. 
Furthermore, we mention some applications of $\SO$.

Let us begin with a central property of $\SO(G)$ that was recognized early on by Feichtinger \cite{fe79-5,Feichtinger1981}. The space $\SO(G)$ is the smallest, so-called, Segal algebra $S(G)$ on a locally compact abelian groups $G$ (see Definition~\ref{def:segal-algebra}) which is invariant under the multiplication of characters such that $\Vert E_{\omega}f \Vert_{S} = \Vert f \Vert_{S}$ for all $f\in S(G)$ and $\omega\in \ghat$. This is the reason for the symbol $\SO$; the $\textbf{S}$ stands for Segal algebra and the index ``$0$'' indicates its minimality.
The minimality of the Feichtinger algebra among a family of function spaces was extended significantly in \cite{fe87-1,fe89-1} and can also be found in \cite{fezi98,MR1843717}. The result can be formulated as follows: Let $B$ be a Banach space with the properties that
\begin{enumerate}[(a)]
\item there exists a non-zero function $g\in \SO(G)$ which also belongs to $B$ such that all time-frequency shifts of $g$, $\{E_{\omega}T_{x}g \, : \, (x,\omega)\in G\times\ghat\}$ belong to $B$ as well (for example $g$ could be a Schwartz-Bruhat function or a compactly supported function with integrable Fourier transform),
\item for $g$ as in (a) there exists a constant $c>0$ for which $\Vert E_{\omega} T_{x} g \Vert_{B} \le c \, \Vert g \Vert_{B}$ for all $(x,\omega)\in G\times\ghat$,
\end{enumerate}
then not only all time-frequency shifts of $g$ but all of $\SO(G)$ is a subset of $B$. That is, $\SO(G)$ is the smallest among all Banach spaces which satisfy (a) and (b). In this paper this result can be found in Theorem~\ref{th:SO-minimality-bochner}. Well-known examples of Banach spaces which satisfy assumptions (a) and (b) are the $L^{p}$-spaces, $C_{0}(G)$ and the Fourier algebra. 

The idea of a minimal algebra of functions in the sense of $\SO$ as above, has been extended to locally compact (non-abelian) groups by Spronk \cite{sp07} and also to homogeneous spaces by Parthasarathy and Shravan Kumar\cite{kupa15}. 
In fact, the construction of $\SO(G)$ has been extended and applied to generalized stochastic processes on hypergroups, for more on this see the book contribution by Heyer \cite{hey14}.
The use of $\SO$ as a setting for generalized stochastic processes was first investigated by H\"ormann in his thesis \cite{ho89}. In particular $\SO(G)$ is used as a replacement for the (non-Fourier invariant) space of smooth functions with compact support as it is used by Gelfand, Vilenkin \cite{MR0173945} and It\^{o}\cite{it54}. This idea has been continued in cooperation with Feichtinger \cite{feho90,feho14} and is also the subject of the thesis by Keville \cite{ke03} and of work by Wahlberg \cite{wa05}.

The Feichtinger algebra finds its most prominent use in the theory of pseudo-differential operators\label{footnotepage}\footnote{see, for example, \cite{ta94,grhe99,cz03,gr06-3,gr06-6,coni10,to10,begrheok05,beok04,cogr03-1,
fehelaleto06,go11,fe02,
grhe03,grst07,towozh07,bo04-2,MR1843717,
cotawa13,coni10,la01-1,to04-2,to04-3}} and in time-frequency analysis, especially in the theory of Gabor frames\footnote{see, for example,  \cite{fe89-1,fegrwa92,fegr97,fezi98,feka04,
feko98,felu06,gr07-2,gr07-3,MR1843717,grle04,
gr14,grorro15,he07,lu09}}. For 
an introduction to the general theory of frames and Gabor analysis 
see the books by Christensen \cite{ch16newbook}, Gr\"ochenig 
\cite{MR1843717} and Heil \cite{he11}. As mentioned in the introduction, the Feichtinger algebra and the Schwartz-Bruhat space have very similar properties. This aspect of $\SO$ is used by Reiter \cite{re89} to     
extend results of Weil \cite{we64} 
on metaplectic operators on the Schwartz-Bruhat space to $\SO(G)$. Similarly, 
Luef \cite{lu07-2} extends results by Rieffel \cite{ri81,ri88} on non-commutative tori from 
the Schwarz-Bruhat space to the Feichtinger algebra. Furthermore,  Feichtinger and Gr\"ochenig \cite{fegr97} extend results from Janssen \cite{ja95} in Gabor frame theory with generators in $\mathcal{S}(\R)$ to generators in $\SO(\R)$. These publications show that $\SO(G)$, compared with the Schwartz-Bruhat space, allows for more general statements with easier proofs. 

Furthermore, the Feichtinger algebra is the setting in the work of Kailath, Pfander, and Walnut \cite{pfwa06,pf13-1,kapfwa15} and in the thesis of Civan \cite{ci15} on operator identification. 
Kaiblinger shows that $\SO$ is a suitable domain for results on interpolation operators \cite{ka05,feka07}.
Further, the Feichtinger algebra is used in the thesis by Querenberger \cite{qu89} on spectral synthesis. For more on spectral synthesis see the book by Benedetto \cite{be75}. 
Feichtinger and Weisz \cite{fewe06,fewe06-1} have shown that all classical summability kernels belong to the Feichtinger algebra $\SO(\R)$. Also, the so-called Wilson bases form an unconditional bases for $\SO(\R^n)$ \cite{fegrwa92}. 

Finally, $\SO(G)$ plays an integral part in the rigged Hilbert space, also known as a Gelfand triple, formed by the  spaces $\SO,L^{2}$ and $\SOprime$. This triple has been strongly advocated by Feichtinger and others over the past years, see \eg \cite{feko98,dofegr06,feluwe07,cofelu08,fe09}. The theory of rigged Hilbert spaces plays a decisive role in the mathematical formulation of quantum mechanics \cite{an98-2}.
 
%All of the above mentioned applications of $\SO(G)$ show, that the Segal algebra $\SO(G)$ has become a classical function space with many different applications. 

In the following we mention some of the function space constructions which yield $\SO$ as a special case. 
 
\begin{itemize}
\item $\SO(G)$ can be realized as the space $\ell^{1}(A)$ by Bertrandias et al.\ in \cite{bedadu78,be82,be84-1}. In fact, this construction predates Feichtinger's discovery of $\SO(G)$. However, the properties of this space were first recognized by Feichtinger.
\item Inspired by the characterization of $\SO$ by bounded uniform partitions of unity in \cite{Feichtinger1981} (see the sets $\mathscr{T}$ and $\mathscr{U}$ in Definition~\ref{def:S0-bigdef}), Feichtinger introduced the Wiener amalgam spaces in \cite{fe81-1,MR751019}. The general setup of Wiener amalgam spaces allows for the construction of a wide variety of Banach spaces and includes, e.g., the Wiener algebra and the usual $L^{p}$-spaces. The Feichtinger algebra is the Wiener amalgam space with local component in the Fourier algebra $A(G)$ and global component in $L^{1}(G)$, denoted by $W(A(G),L^{1})$. For more on Wiener amalgam space see, \eg \cite{fe92-3} and the paper by Heil \cite{he03}.
\item One of the characterizations of $\SO(G)$ in \cite{Feichtinger1981} inspired Feichtinger to define the \emph{modulation spaces} \cite{fe83-4}, see Definition~\ref{def:mod-space}. However, the original manuscript remained unpublished until 2003, where an essentially unchanged version was published \cite{fe03-1}. Modulation spaces have been described in, e.g., \cite{fe83-1,fe89-1,fegr92-1} and in the book by Gr\"ochenig \cite{MR1843717} and de Gosson \cite{go11}. The Feichtinger algebra coincides with the modulation space $M^{1}$. The reader is referred to \cite{fe06} for more information on the history of $\SO(G)$, its role in the construction of the modulation spaces and its relation to them. 
As indicated by the references in the footnotes on page \pageref{footnotepage}, the modulation spaces provide an incredible fruitful environment for results in the theory of pseudo-differential operators
and the theory of Gabor frames. 
 
\item In the mid 80s it became clear 
that $\SO(G)$ was connected to the Schr\"odinger representation of 
the Heisenberg group. This, together with the advent of wavelet theory inspired Feichtinger and Gr\"ochenig to 
establish the \emph{coorbit space theory}
\cite{MR942257,MR1021139,MR1026614}. The theory associates to each 
integrable representation of a locally compact group a family of 
Banach spaces, the so-called coorbit spaces, in which one can achieve 
suitable series representations. In this way the 
theory connects to the field of frame theory. 
In particular the coorbit theory yields $
\SO$ and the aforementioned modulation spaces if one uses the 
Schr\"odinger representation of the Heisenberg group.   
See also the papers \cite{gr91,fegr92-1,ch96-2}.  
\item The space $\SO(G)$ can be obtained from the very general setting of \emph{decomposition spaces}, introduced in 
\cite{fegr85} and \cite{fe87}.
\item $\SO(G)$ coincides with the minimal homogeneous Banach space $(A(G))_{\text{min}}$ introduced in \cite{fe81}, and it is an example of a minimal Banach space as described in \cite{fe87-1}. 
\item As mentioned earlier, the space $\SO(G)$ is also a Segal algebra, i.e., a translation invariant dense subalgebra of $L^{1}(G)$ under convolution, which is continuously embedded into $L^{1}(G)$. For more on Segal algebras see, e.g., \cite{re71} and \cite{MR1802924}.
\end{itemize}

Recent literature on coorbit space theory, 
decomposition spaces and modulation spaces include   
\cite{MR2966135,daforastte08,dastte04-1,fu15,fuvo15,raul11} and \cite{vo15}.

\section{Setup and notation}
\label{sec:prelim}
This section contains definitions, notation and results that will be used throughout this paper. The material presented here can be found in text books on functional analysis (\eg \cite{MR3289046,ru91}) and in books concerning Fourier analysis on locally compact abelian groups (\eg \cite{hani98,hero63,hero70,fo95,ru62,MR1802924}).

Throughout this paper we let $G$ denote a locally compact Hausdorff abelian group. Examples of such are the real line $\R$, the integers $\Z$, the torus $\T\cong[0,1[$, finite abelian groups, for some prime number $p$ the discrete Pr\"ufer $p$-group $\mathbb{Z}(p^{\infty})= \{ z\in \C \, : \, z^n = 1, n= p^{k}, k\in \N \}$, the (non-compact group of) $p$-adic numbers, the (compact group of) $p$-adic integers and finite products thereof. A rather long list of locally compact abelian groups can be found on pages $161$-$163$ in \cite{hani98}. To $G$ we associate
its dual group $\ghat$ which consists of all characters of $G$, i.e., all continuous group homomorphisms from $G$
into the torus $\T = \{z\in\C \, : \, \vert z \vert =1\}$. The dual group $\ghat$ is a locally compact abelian group under pointwise multiplication and the compact-open
topology for continuous functions between topological spaces. Throughout the paper we denote the group operation as addition $+$, hence $-x$ denotes the inverse element of $x\in G$. By the Pontryagin-van Kampen duality theorem of locally compact abelian groups, the dual group
of $\ghat$ is isomorphic to $G$ as a topological group, i.e., $\ghhat \cong G$. 

Every locally compact abelian group $G$ carries a translation invariant measure, the so-called Haar measure $\mu_G$. This measure is unique up to a positive constant. We define the $L^p$-spaces
over the complex field with the measure $\mu_G$ in the usual way. We will usually shorten notation and write $\int_{G} \, dx$ instead of $\int_{G} \, d\mu_{G}(x)$. The space $L^{2}(G)$ is equipped with the usual inner product 
$\langle f, g\rangle = \textstyle \int_{G} f(x) \overline{g(x)} \, dx$, which is linear in the first entry. 
%We shall also use the inner product brackets to denote the action of functions in $L^{q}(G)$ on functions in $L^{p}(G)$, i.e., $\langle f, g\rangle = \int_{G} f(x) \, \overline{g(x)} \, dx$ for $f\in L^{p}(G)$, $g\in L^{q}(G)$, $p,q\in [1,\infty]$, $1/p+1/q=1$. 
%Also, if $B$ is a Banach space, and $B'$ is its dual, we let $\langle f, \sigma\rangle_{B,B'}:= \sigma(f)$ be the action of $\sigma\in B'$ on $f\in B$. 
We let $C_{0}(G)$ and $C_{b}(G)$ denote the space of continuous functions that vanish at infinity and the space of bounded continuous functions on $G$, respectively. 
% and $\, \overline{\phantom{n}}\,$ indicates complex conjugation. 

%For closed subgroups $H$ of $G$ Weil's formula relates integrable functions over $G$ with integrable
%functions on the quotient space $G/H$. 
%In particular if $\pi_H:G\to G/H, \
%  \pi_H(x) = x+H$ is the \textit{canonical map} from $G$ onto $G/H$ then the function $\dot x \mapsto \int_H f(x+h) \, d\mu_{H}(h)$, $\dot x = \pi_H(x)$ defined almost everywhere on $G/H$, is integrable. And, furthermore, if two of the Haar measures on $G, H$ and $G/H$ are given, then the third can be normalized uniquely such that 
%\begin{equation}
%\label{eq:1802a}  \int_G f(x) \, d\mu_{G}(x) = \int_{G/H} \int_H f(x+h) \, d\mu_{H}(h) \, d\mu_{G/H}(\dot x) .
%\end{equation} 
%Remarkeblty, if \eqref{eq:1802a} holds, then the respective dual measures on $\widehat G, H^{\perp}\cong
%  \widehat{G/H}$, $\ghat / H^{\perp} \cong \widehat H$ satisfy
%\begin{equation}
%  \label{eq:1802b} 
%  \int_{\widehat G} \hat{f}(\omega) \, d\mu_{\ghat}(\omega) = \int_{\ghat/H^{\perp}} \int_{H^{\perp}} \hat{f}(\omega \gamma) \, d\mu_{H^{\perp}}(\gamma )\, d\mu_{\ghat/H^{\perp}}(\dot \omega).
%\end{equation}  

We define the Fourier transform of functions $f\in L^1(G)$ by
\[ 
\cF f(\omega) = \hat{f}(\omega) = \int_{G} f(x) \, \overline{\omega(x)} \, d\mu_G(x), \ \
\omega \in \ghat.
\] 

%In the case of $G=\R,\T,\Z/d\Z$, $d\in \N$ this yields the usual Fourier transform on $\R$, the Fourier series for periodic functions and the discrete Fourier transform, respectively. 
%In particular the action of the character $\omega(x)$ is given as $e^{2\pi i x \omega}$, $x\in \R$ and $\omega\in \R$; $e^{2\pi i x \omega}$, $x\in [0,1[\cong \T$ and $\omega\in \Z$; $e^{2\pi i x\omega/d}$, $x\in \Z/dZ$, $\omega\in \Z/dZ$ .

By the Riemann-Lebesgue Lemma the Fourier transform maps $L^{1}(G)$ into $C_{0}(\ghat)$ and $\Vert \hat{f} \Vert_{\infty} \le \Vert f \Vert_{1}$ for all $f\in L^{1}(G)$. If the Haar measure $\mu_{G}$ on $G$ is given, then the measure $\mu_{\ghat}$ on the dual group $\ghat$ can be normalized uniquely such that, for all $f\in L^1(G)$ with $\hat{f} \in L^1(\ghat)$, the function $f$ can be recovered from
$\hat{f}$ by the inverse Fourier transform
\[
f(x) = \cF^{-1}\hat{f}(x) = \int_{\ghat} \hat{f}(\omega) \, \omega(x) \, d\mu_{\ghat}(\omega),
\ \ a.e.\ \ x\in G.
\]
If, in addition, $f$ is continuous, then the inversion of the Fourier transform holds pointwise. If the Fourier inversion formula holds, 
then we refer to $\mu_{G}$ and $\mu_{\ghat}$ as \emph{dual measures}. We always assume that the measures on $G$ and $\ghat$ are related by duality. Under this convention, the Fourier transform $\mathcal{F}$ extends from $L^{1}(G)\cap L^{2}(G)$ to
a unitary operator from $L^2(G)$ onto $L^2(\ghat)$.

With the help of the Fourier transform we define the Fourier algebra
\begin{equation} \label{eq:0211a} A(G) = \{ f \in C_{0}(G) \, : \, \exists \, h \in L^{1}(\ghat) \ \text{s.t.} \ \mathcal{F}_{\ghat}h = f \}.\end{equation}
Here $\mathcal{F}_{\ghat}$ denotes the Fourier transform from $L^{1}(\ghat)$ into $C_{0}(G)$. The Fourier algebra becomes a Banach space under the norm $\Vert f \Vert_{A(G)} = \Vert h \Vert_{1}$, with $h$ as in \eqref{eq:0211a}.
For two functions $f,g\in L^1(G)$ we define their convolution product by
\[ (f* g)(x) = \int_{G} f(s) g(x-s) \, ds \ \ \text{for} \ a.e. \ x\in G.\]
More general, the convolution of two measurable functions is well-defined as a bilinear mapping between suitable $L^{p}$-spaces. Indeed, Young's inequality states that, for $p,q,r\in[1,\infty]$ such that
$ \tfrac{1}{p} + \tfrac{1}{q} = \tfrac{1}{r} + 1$, 
one has the norm estimate
\begin{equation} \label{eq:2910a}
 \Vert f * g \Vert_{r} \le \Vert f \Vert_{p} \Vert g \Vert_{q}.
\end{equation}
For a complex valued function $f: G\to \C$ we define its involution $f^{\dagger}$ and reflection $f^{r}$ by
\[ f^{\dagger}(x) = \overline{f(-x)}, \ \ f^{r}(x) = f(-x).\] 
One shows easily that
$ \mathcal{F}(f^{\dagger}) = \overline{\mathcal{F}f}, \ \mathcal{F}(f^{r}) = (\mathcal{F}f)^{r} = \mathcal{F}^{-1} f, \  \mathcal{F}(\overline{f}) = \overline{\mathcal{F}^{-1} f} = (\mathcal{F}f)^{\dagger}$.

We recall the convolution theorem for the Fourier transform: if $f,g\in L^{1}(G)$, then $\mathcal{F}(f * g) = \mathcal{F}f \cdot \mathcal{F}g$ and if $f,g\in L^{2}$, then $\mathcal{F}(f \cdot g) = \mathcal{F}f * \mathcal{F} g$.

Similar to the Fourier transform, the partial Fourier transform
\[ \mathcal{F}_{1} F(\omega,t) = \int_{G_{1}} F(x,t) \, \overline{\omega(x)} \, dx \ \ \text{for all} \ F\in L^{1}(G_{1}\times G_{2}), \ \omega\in \ghat_{1} \ \text{and}\ a.e.\ t\in G_{2},\]
can be extended to a unitary operator from $L^{2}(G_{1}\times G_{2})$ onto $L^{2}(\ghat_{1}\times G_{2})$.
The index indicates that the Fourier transform is taken with respect to the first argument. In a similar way one defines $\mathcal{F}_{2}$.

Let $X$ be a Banach space which is equipped with a multiplication $\odot: X\times X \to X$ such that for all $x,y,z\in X$, $\alpha\in \C$ it holds that $ (x+y) \odot z = (x\odot z) + (y \odot z)$,  
\[ (x\odot y) \odot z = x\odot (y \odot z)  \ \text{ and } \ \alpha \cdot (x\odot y) = (\alpha \cdot x ) \odot y = x \odot(\alpha \cdot y ).\]
The Banach space $X$ is a \emph{Banach algebra} if $\Vert x \odot y \Vert \le \Vert x \Vert \, \Vert y \Vert$ for all $x,y\in X$.
The Banach spaces $L^{1}(G)$ and $A(G)$ form Banach algebras under convolution and pointwise multiplication, respectively. 

Let $X,Y$ be two normed vector spaces. If $X\subseteq Y$, then we say that $X$ is continuously embedded into $Y$ if there is a constant $c>0$ such that $\Vert x \Vert_{Y} \le c \, \Vert x \Vert_{x}$ for all $x\in X$.

For $x\in G$ and $\omega\in \ghat$ we define the translation operator $T_{x}$ and the modulation operator $E_{\omega}$ by 
\begin{align*}
 T_{x}f(s) = f(s-x), \ \ E_{\omega}f(s) = \omega(s) f(s) .\end{align*}
It is straightforward to verify that the translation and modulation operator are isometries on the Fourier algebra $A(G)$ and any of the $L^{p}$-spaces for $p\in [1,\infty]$. On $L^{2}(G)$ the operators $T_{x}$ and $E_{\omega}$ are unitary. Furthermore, for $f$ in $L^{p}(G)$ (or $A(G)$) the mappings $x\mapsto T_{x}f$ and $\omega \mapsto E_{\omega} f$ are continuous from $G$ into $L^{p}(G)$ (or $A(G)$) and from $\ghat$ into $L^{p}(G)$ (or $A(G)$), respectively. It is useful to note that 
\[ E_{\omega} T_{x} = \omega(x) T_{x} E_{\omega}, \ \ \mathcal{F} T_{x} = E_{-x} \mathcal{F}, \ \ \mathcal{F}E_{\omega} = T_{\omega}\mathcal{F},\]
and furthermore, that $E_{\omega} (f * g) = (E_{\omega}f) * (E_{\omega}g)$ and $T_{x} (f*g) = (T_{x}f)*g = f*(T_{x}g)$, whenever the convolution of $f$ and $g$ is well-defined.
We will often use the time-frequency shift operator $E_{\omega}T_{x}$. For convenience we therefore define
\[ \pi(\chi) = \pi(x,\omega) = E_{\omega}T_{x} \ \ \text{for} \ \ \chi = ( x,\omega) \in G\times \ghat.\] %$\pi(\chi)$ is a linear and bounded operator on $L^{p}(G)$, and in fact isometric. On $L^{2}(G)$ the operator $\pi(\chi)$ is unitary. The commutator relation
%\begin{align}
%T_{x}E_{\omega} & = \overline{\omega(x)}
%E_{\omega}T_{x} \nonumber %\label{eq:cr1}
%\intertext{leads to the following useful identities:}
%\pi(\chi)^* & = \overline{\omega(x)} \, \pi(-\chi),
%\label{eq:cr2} 
%\\  \pi(\chi_1)\pi(\chi_2) & = \overline{\omega_2(x_1)} \, \pi(\chi_1+\chi_2) \label{eq:cr3}
%\\
%\pi(\chi_1)\pi(\chi_2) & = \omega_1(x_2) \overline{\omega_2(x_1)} \, \pi(\chi_2)\pi(\chi_1), \label{eq:cr4}
%\end{align}
%where $\chi_i=(x_i,\omega_i)$, $i=1,2$, and $\pi(\chi)^*$ denotes the Hilbert-space adjoint operator of $\pi(\chi)$ in $L^{2}(G)$. 
For any $p\in [1,\infty]$ we define the \emph{asymmetric coordinate transform}
\begin{equation} \label{eq:assym-cord-trans} 
\tau_{a}: L^{p}(G\times G) \to L^{p}(G\times G), \ \tau_{a}f(x,t) = f(t,t-x).\end{equation}
It is easy to check that $\tau_{a}$ is an isometry and that its inverse is $\tau_{a}^{-1} f(x,t) = f(x-t,x)$. For $p=2$ the operator $\tau_{a}$ is unitary.
For two locally compact abelian groups $G_{1}$ and $G_{2}$ and $p\in [1,\infty]$ we define the \emph{tensor product of functions}
\[ \otimes: L^{p}(G_{1})\times L^{p}(G_{2}) \to L^{p}(G_{1}\times G_{2}), \ ( f_{1} \otimes f_{2} )(x_{1},x_{2}) = f_{1}(x_{1}) \cdot f_{2}(x_{2}).\]

By use of the partial Fourier transform $\mathcal{F}_{2}$, the asymmetric coordinate transform $\tau_{a}$ and the tensor product $\otimes$, we define the \emph{short-time Fourier transform} with respect to a function $g\in L^{2}(G)$: 
\[ \mathcal{V}_{g} : L^{2}(G) \to L^{2}(G\times \ghat), \ \mathcal{V}_{g}f(x,\omega) = \mathcal{F}_{2} \tau_{a} (f \otimes \overline{g})(x,\omega) = \langle f, E_{\omega} T_{x} g\rangle.\]
Since $\mathcal{F}_{2}$ and $\tau_{a}$ are unitary operators on $L^{2}$ it is straightforward to show the bi-orthogonality relations for the short-time Fourier transform: for all $f_{1},f_{2},g_{1},g_{2} \in L^{2}(G)$
\begin{equation} \label{eq:STFT-dual}
 \langle \mathcal{V}_{g_{1}} f_{1}, \mathcal{V}_{g_{2}}f_{2} \rangle = \langle \mathcal{F}_{2} \tau_{a} (f_{1}\otimes \overline{g_{1}}) , \mathcal{F}_{2} \tau_{a} (f_{2}\otimes \overline{g_{2}}) \rangle = \langle f_{1}\otimes \overline{g_{1}} , f_{2}\otimes \overline{g_{2}} \rangle = \langle g_{2}, g_{1}\rangle \langle f_{1}, f_{2}\rangle.  
\end{equation}
From this follows that $\mathcal{V}_{h}^{*} \mathcal{V}_{g} f = \langle h,g\rangle \, f$ for all $f,g,h\in L^{2}(G)$.
%For more on the short-time Fourier transform see, e.g., \cite{MR1843717}.
%The cross Wigner distribution (which is related to the short-time Fourier transform) can be defined in a similar way, 
%see, \eg \cite{ho89}.

\section{Basic properties of $\SO$}
\label{sec:basic-prop-of-so}
We begin our excursion into the Feichtinger algebra with the set $\mathscr{A}$ from 
Definition~\ref{def:S0-bigdef} and define 
\begin{enumerate}[]
\item $\quad \SO(G) = \mathscr{A} = \{ f\in L^{1}(G) \, : \, \exists \, g\in L^{1}(G)\backslash\{0\} \text{ s.t. } \textstyle\int_{\ghat} \Vert E_{\omega} f * g \Vert_{1} \, d\omega < \infty\}$.
\end{enumerate} 

The aim of this section is to show various properties of the functions that belong to $\SO(G)$. An important ingredient of these results are 
Theorem~\ref{th:s0-ABC} and Theorem~\ref{th:SO-six-def} in which we show that $\SO(G)$ is equal to the sets $\mathscr{B}$--$\mathscr{I}$ from Definition~\ref{def:S0-bigdef}. These characterizations allow us to give sufficient conditions for functions to belong to $\SO(G)$ and they allow us to show that $\SO(G)$ forms a Banach algebra with respect to pointwise multiplication and convolution. We will also prove that $\SO(G)\subseteq L^{1}(G) \cap A(G) \subseteq C_{0}(G)$, and that $\SO(G)$ is a Segal algebra.

\subsection{Preliminary observations}
\label{sec:preliminary-results}

In our first result, Theorem~\ref{th:s0-ABC}, we show that $\SO(G)$ coincides with the sets $\mathscr{B}$ and $\mathscr{C}$ from Definition~\ref{def:S0-bigdef}.
\begin{enumerate}[]
\item $\quad \mathscr{B} =\{ f\in L^{2}(G) \, : \, \exists \, g\in L^{2}(G)\backslash\{0\} \, \ \text{s.t.} \ \int_{G\times\ghat} \vert \mathcal{V}_{g}f(\chi) \vert \, d\chi < \infty\}$,
\item $\quad \mathscr{C} = \{ f\in A(G) \, : \, \exists \, g\in A(G)\backslash\{0\} \, \ \text{s.t.} \ \int_{G} \Vert T_{x} f \cdot g \Vert_{A(G)} \, dx < \infty \}$.
\end{enumerate}
The proof of Theorem~\ref{th:s0-ABC} is lengthy, however it consists 
of simple manipulations. As a pay-off we are rewarded with a variety 
of useful statements, which we collect in 
Corollary~\ref{co:f-in-s0-properties} below. The characterization of $\SO(G)$ via the set 
$\mathscr{B}$ appears in \cite{MR1843717}. The sets $\mathscr{A}$ and 
$\mathscr{C}$ have not been used in the literature to describe $\SO(G)$, yet they bear resemblance to characterizations of $\SO(G)$ that appear in \cite{Feichtinger1981}.

\begin{theorem} \label{th:s0-ABC}
For any locally compact abelian group $G$ it holds that $\SO(G) = \mathscr{A} = \mathscr{B} = \mathscr{C}$.

%following statements are equivalent:
%\begin{enumerate}[(i)]
%\item $f\in \SO(G)$.
%\item $f\in \mathscr{A} = \{ f\in L^{1}(G) \, : \, \exists \, g\in L^{1}(G)\backslash\{0\} \text{ s.t. } \int_{\ghat} \Vert E_{\omega} f * g \Vert_{1} \, d\omega < \infty\}$,
%\item $f\in \mathscr{B} =\{ f\in L^{2}(G) \, : \, \exists \, g\in L^{2}(G)\backslash\{0\} \text{ s.t. } \int_{G\times\ghat} \vert \mathcal{V}_{g}f(\chi) \vert \, d\chi < \infty\}$,
%\item $f\in \mathscr{C} = \{ f\in A(G) \, : \, \exists \, g\in A(G)\backslash\{0\} \text{ s.t. } \int_{G} \Vert T_{x} f \cdot g \Vert_{A(G)} \, dx < \infty \}$.
%\end{enumerate}
\end{theorem} 
\begin{proof}
By definition $\SO(G) = \mathscr{A}$.  We now show that (i) $\mathscr{A}\subseteq \mathscr{B}\cap \mathscr{C}$, (ii) $\mathscr{B}\subseteq \mathscr{A}\cap \mathscr{C}$ and (iii) $\mathscr{C}\subseteq \mathscr{A}\cap \mathscr{B}$. \\
(i). Let $f,g\in L^{1}(G)$ and assume that $\int_{\ghat} \Vert E_{\omega} f * g \Vert_{1} \, d\omega < \infty$. By the Riemann-Lebesgue Lemma $\hat
f,\hat{g} \in C_{0}(\ghat)$, and we can thus evaluate $\hat f$ and $\hat{g}$ pointwise. 
For all $\omega,\xi\in \ghat$ we can make the following estimate:
\begin{align}
 \vert \hat{f}(\xi) \, \hat{g}(\omega) \vert & \le \max_{s\in \ghat}\vert \hat{f}(s-\omega+\xi) \,\hat{g}(s) \vert = \max_{s\in \ghat} \vert \big(T_{\omega-\xi}\hat{f}\big)(s) \, \hat{g}(s) \vert \nonumber
\\
& = \max_{s\in \ghat} \vert ( \mathcal{F}E_{\omega-\xi} f)(s) \cdot (\mathcal{F}g)(s) \vert = \Vert \mathcal{F} (E_{\omega-\xi}f * g) \Vert_{\infty} \le \Vert E_{\omega-\xi}f * g \Vert_{1} . \label{eq:integration-step}
\end{align}
Integrating over $\omega \in \ghat$ yields that
\[ \vert \hat{f}(\xi) \vert \int_{\ghat} \vert \hat{g}(\omega)\vert \, d\omega \le \int_{\ghat} \Vert E_{\omega-\xi}f * g \Vert_{1} \, d\omega. \]
By use of the translation invariance of the Haar measure we find that for all $\xi \in \ghat$
\[ \vert \hat{f}(\xi) \vert \Vert \hat{g} \Vert_{1} \le \int_{\ghat} \Vert E_{\omega}f * g \Vert_{1} \, d\omega < \infty. \]
The right side of this inequality does not depend on $\xi$, therefore, taking the maximum over $\xi\in \ghat$,  we arrive at the estimate 
\[ \Vert \hat{f} \Vert_{\infty} \Vert \hat{g} \Vert_{1} \le \int_{\ghat} \Vert E_{\omega}f * g \Vert_{1} \, d\omega < \infty. \]
This shows that $\hat{g}\in L^{1}(\ghat)$, i.e., $g\in A(G)$. Integrating \eqref{eq:integration-step} over $\xi\in \ghat$ and taking the maximum over $\omega\in \ghat$ yields the inequality
\[ \Vert \hat{g} \Vert_{\infty} \Vert \hat{f} \Vert_{1} \le \int_{\ghat} \Vert E_{\omega}f * g \Vert_{1} \, d\omega < \infty. \]
Hence also $\hat{f}\in L^{1}(G)$ and thus $f\in A(G)$. In particular this implies that $\hat{f},\hat{g}\in L^{1}(\ghat)\cap L^{\infty}(\ghat)\subseteq L^{2}(\ghat)$. Since the Fourier transform is a unitary operator from $L^{2}(G)$ onto $L^{2}(\ghat)$ we conclude that $f,g\in L^{2}(G)$. For functions $f,g\in L^{1}(G) \cap L^{2}(G) \cap A(G)$ it is straightforward (essentially a matter of notation and Fubini's theorem) to show that
\begin{align}
 \int_{\ghat} \Vert E_{\omega} f * g \Vert_{1} \, d\omega = \int_{G\times\ghat} \vert \mathcal{V}_{g^{\dagger}}f(\chi)  \vert \, d\chi = \int_{G} \Vert T_{x} f \cdot g^{r} \Vert_{A(G)} \, dx.\label{eq:0703c}
\end{align}
The equalities in \eqref{eq:0703c} together with the assumption $f\in \mathscr{A}$ implies that $f$ belongs to $\mathscr{B}$ and $\mathscr{C}$, that is, we have shown that $\mathscr{A} \subseteq \mathscr{B} \cap \mathscr{C}$. \\
(ii). Assume now that $f,g\in L^{2}(G)$ satisfy
\[ \int_{G} \int_{\ghat} \vert \mathcal{V}_{g}f(\chi) \vert \, d\omega \, dx = \int_{G} \Vert \mathcal{F}(f\cdot T_{x}\overline{g}) \Vert_{1} \, dx < \infty.\]
This implies that the mapping $x \mapsto \Vert \mathcal{F}(f\cdot T_{x}\overline{g}) \Vert_{1}$ belongs to $L^{1}(G)$ and thus $\Vert \mathcal{F}(f\cdot T_{x}\overline{g}) \Vert_{1}$ is finite for almost every $x \in G$. Therefore, for almost every $x \in G$ the function $\omega \mapsto \mathcal{F}(f\cdot T_{x}\overline{g})(\omega)$ belongs to $L^{1}(\ghat)$. By the Fourier inversion formula we have the equality
\[ f \cdot T_{x} \overline{g} = \mathcal{F}^{-1} \mathcal{F}(f\cdot T_{x} \overline{g})\]
for almost every $x\in G$.
In particular, for almost every $x\in G$ we have that $(f \cdot T_{x} \overline{g})\in C_{0}(G)$.
Now, for almost every $x,\xi\in G$, we conclude that
\begin{align}
\vert f(x) \overline{g(\xi)} \vert & = \vert f(x) \overline{g(x-x+\xi)} \vert \le \sup_{s\in G} \vert f(s) \overline{g(s-x+\xi)} \vert \nonumber \\
& = \Vert f \cdot T_{x-\xi} \overline{g} \Vert_{\infty} = \Vert \mathcal{F}^{-1} \mathcal{F} (f \cdot T_{x-\xi} \overline{g}) \Vert_{\infty} \le \Vert \mathcal{F} (f \cdot T_{x-\xi} \overline{g}) \Vert_{1}. \label{eq:0703a}
\end{align}
By integrating the previous inequality over $x\in G$ and using translation invariance of the Haar measure, we find that 
\[ \int_{G} \vert f(x) \vert \, dx \, \vert g(\xi) \vert \le \int_{G} \Vert \mathcal{F} (f \cdot T_{x-\xi} \overline{g}) \Vert_{1} \, dx = \int_{G} \Vert \mathcal{F} (f \cdot T_{x} \overline{g}) \Vert_{1} \, dx = \int_{G\times\ghat} \vert \mathcal{V}_{g}f(\chi) \vert \, d\chi\]
for all $\xi\in G$. 
By taking the supremum over $\xi\in G$ we find that
\[ \Vert f \Vert_{1} \Vert g \Vert_{\infty} \le \int_{G\times\ghat} \vert \mathcal{V}_{g} f(\chi) \vert \, d\chi<\infty.\]
An integration of \eqref{eq:0703a} with respect to $\xi\in G$ and taking the supremum over $x\in G$ yields the inequality 
\[ \Vert f \Vert_{\infty} \Vert g \Vert_{1} \le \int_{G\times \ghat} \vert \mathcal{V}_{g}f(\chi) \vert \, d\chi< \infty.\]
This implies that $f,g\in L^{1}(G)$. 
Since the Fourier transform is a unitary operator from $L^{2}(G)$ onto $L^{2}(\ghat)$ it is a straightforward calculation to show that
$\int_{G\times\ghat} \vert \mathcal{V}_{g}f(\chi) \vert \, d\chi = \int_{\ghat\times G} \vert \mathcal{V}_{\hat{g}} \hat{f}(\tilde\chi) \vert \, d\tilde\chi$. 
Repeating the above argument on the Fourier side, yields 
\begin{equation} \label{eq:2903a} \Vert \hat{f} \Vert_{1} \Vert \hat{g} \Vert_{\infty} \le \int_{G\times\ghat} \vert \langle f, \pi(\chi) g \rangle \vert \, d\chi< \infty \ \ \text{and} \ \ \Vert \hat{g} \Vert_{1} \Vert \hat{f} \Vert_{\infty} \le \int_{G\times\ghat} \vert \langle f, \pi(\chi) g \rangle \vert \, d\chi< \infty. \end{equation}
We conclude that $f,g\in \mathscr{B} \subseteq L^{1}(G) \cap A(G)$. 
Replacing $g$ by $g^{\dagger}$ in \eqref{eq:0703c} yields
\begin{align} \label{eq:2103a}
 \int_{\ghat} \Vert E_{\omega} f * g^{\dagger} \Vert_{1} \, d\omega = \int_{G\times\ghat} \vert \mathcal{V}_{g}f(\chi) \vert \, d\chi = \int_{G} \Vert T_{x} f \cdot \overline{g} \Vert_{A(G)} \, dx.
\end{align}
for $f,g\in L^{1}(G) \cap L^{2}(G) \cap A(G)$. Arguing as in (i) we prove that $\mathscr{B} \subseteq \mathscr{A} \cap \mathscr{C}$. \\
(iii). As in (i) and (ii) one can show that for $f,g\in A(G)$ and all $x,\xi\in G$ that 
\[ \vert f(x) g(\xi) \vert \le \Vert T_{x-\xi} f \cdot g \Vert_{\infty}.\]
For $f,g\in A(G)$ there exists $h_{f},h_{g}\in L^{1}(\ghat)$ such that $\mathcal{F}h_{f} = f, \mathcal{F}h_{g} = g$ and so
$ \mathcal{F}(E_{x-\xi}h_{f}*h_{g}) = T_{x-\xi} f \cdot g$
and $\Vert T_{x-\xi} f \cdot g \Vert_{A(G)} = \Vert E_{x-\xi}h_{f}*h_{g} \Vert_{1}$.
We can thus establish the estimate
\[ \vert f(x) g(\xi) \vert \le \Vert \mathcal{F}(E_{x-\xi}h_{f}*h_{g}) \Vert_{\infty} \le \Vert E_{x-\xi}h_{f}*h_{g} \Vert_{1} = \Vert T_{x-\xi} f \cdot g \Vert_{A(G)}. \]
An integration over $x\in G$ and the translation invariance of the Haar measure yields that for all $f\in A(G)$ and $\xi\in G$
\[ \int_{G} \vert f(x) \vert \, dx \, \vert g(\xi) \vert\le \int_{G} \Vert T_{x}f \cdot g \Vert_{A(G)} \, 
dx .\]
The right hand side does not depend on $\xi\in G$, hence taking the supremum over $\xi\in G$ gives the inequality
\[ \Vert f \Vert_{1} \, \Vert g \Vert_{\infty} \le 
\int_{G} \Vert T_{x}f \cdot g \Vert_{A(G)} \, dx < \infty
\ \ \text{for all} \ \ f\in \mathscr{C}.\]
Interchanging the roles of $f$ and $g$ we also establish the inequality $\Vert f \Vert_{\infty} \Vert g \Vert_{1} \le 
\int_{G} \Vert T_{x}f \cdot g \Vert_{A(G)} < \infty$.
Hence $f,g\in L^{1}(G)\cap L^{\infty}(G)$ and, in particular, $f,g\in L^{2}(G)$. Replacing $g$ by $g^{r}$ in \eqref{eq:0703c} yields that 
\begin{align} \label{eq:0703b}
 \int_{\ghat} \Vert E_{\omega} f * g^{r} \Vert_{1} \, d\omega = \int_{G\times\ghat} \vert \mathcal{V}_{\overline{g}}f(\chi) \vert \, d\chi = \int_{G} \Vert T_{x} f \cdot g \Vert_{A(G)} \, dx
\end{align}
for $f,g\in L^{1}(G)\cap L^{2}(G)\cap A(G)$.
The assumption that $f\in \mathscr{C}$ together with \eqref{eq:0703b} implies that $f\in \mathscr{A}\cap \mathscr{B}$. This shows that $\mathscr{C}\subseteq \mathscr{A}\cap \mathscr{B}$.
\end{proof}

From the proof of Theorem~\ref{th:s0-ABC} we draw several conclusions which we summarize in Corollary~\ref{co:f-in-s0-properties} below. The result in Corollary~\ref{co:f-in-s0-properties}(i)-(iv) are well-known and can be found in \cite{Feichtinger1981}. Statement (v) appears in \cite{MR1843717}. The results in Corollary~\ref{co:f-in-s0-properties}(vi)-(x) are new. The inequalities in Corollary~\ref{co:f-in-s0-properties}(vii) are related to results by Lieb \cite{li90-1}.

\begin{corollary} \label{co:f-in-s0-properties} \label{co:s0-invariant-under-rel-inv-comp-conj} Suppose $f\in \SO(G)$, \ie $f$ belongs to $L^{2}(G)$ and $\int_{G\times\ghat} \vert \mathcal{V}_{g}f(\chi) \vert \, d\chi < \infty$, for some non-zero $g\in L^{2}(G)$. Then the following holds:
\begin{enumerate}[(i)]
\item $\hat{f}\in \SO(\ghat)$, in particular $\Vert \mathcal{V}_{g}f \Vert_{1} = \Vert \mathcal{V}_{\hat{g}}\hat{f} \Vert_{1}$,
\item $E_{\omega}T_{x} f \in \SO(G)$ for all $(x,\omega)\in G\times\ghat$,
\item $f \in L^{1}(G) \cap A(G) \subseteq C_{0}(G)$,
\item $f^{r},\overline{f},f^{\dagger}\in \SO(G)$,
\item If $f\ne 0$, then $g\in \SO(G)$,
\item[(vi.a)] $\Vert f \Vert_{p} \le \Vert g \Vert_{1}^{-1+1/p} \, \Vert g \Vert_{\infty}^{-1/p} \, \Vert \mathcal{V}_{g} f \Vert_{1} 
%\int_{G\times\ghat} \vert \langle f, \pi(\chi) g \rangle \vert \, d\chi 
$ for all $p\in [1,\infty]$,
\item[(vi.b)] $\Vert \hat{f} \Vert_{p} \le \Vert \hat{g} \Vert_{1}^{-1+1/p} \, \Vert \hat{g} \Vert_{\infty}^{-1/p} \, \Vert \mathcal{V}_{g} f \Vert_{1} 
%\int_{G\times\ghat} \vert \langle f, \pi(\chi) g \rangle \vert \, d\chi 
$ for all $p\in [1,\infty]$,
\item[(vii.a)] $\Vert f \Vert_{p} \, \Vert g \Vert_{q} \le \Vert \mathcal{V}_{g} f \Vert_{1} 
%\int_{G\times\ghat} \vert \langle f, \pi(\chi) g \rangle \vert \, d\chi 
$ for all $p,q\in [1,\infty]$ with $1/p+1/q=1$,
\item[(vii.b)] $\Vert \hat{f} \Vert_{p} \, \Vert \hat{g} \Vert_{q} \le \Vert \mathcal{V}_{g} f \Vert_{1} 
%\int_{G\times\ghat} \vert \langle f, \pi(\chi) g \rangle \vert \, d\chi 
$ for all $p,q\in [1,\infty]$ with $1/p+1/q=1$,
%\item $\Vert f \Vert_{1} \, \Vert g \Vert_{\infty} \le 
%\Vert \mathcal{V}_{g} f \Vert_{1}$,
%\item $\Vert g \Vert_{1} \, \Vert f \Vert_{\infty} \le 
%\Vert \mathcal{V}_{g} f \Vert_{1}$,
\setcounter{enumi}{7}
\item $\Vert f \Vert_{p} \, \Vert \hat{g} \Vert_{p} \le \Vert \mathcal{V}_{g}f \Vert_{1}$ for all $p\in [2,\infty]$.
\item $\vert \langle f,g\rangle \vert \le \Vert \mathcal{V}_{g} f \Vert_{1} 
%\int_{G\times\ghat} \vert \langle f, \pi(\chi) g \rangle \vert \, d\chi 
$,
\item $\Vert \mathcal{V}_{g}f \Vert_{p} \le \Vert \mathcal{V}_{g} f \Vert_{1} 
%\int_{G\times\ghat} \vert \langle f, \pi(\chi) g \rangle \vert \, d\chi 
$ for all $p\in [1,\infty]$.
\end{enumerate}
\end{corollary}
\begin{proof} (i). Because $\mathcal{F}$ is a unitary operator from $L^{2}(G)$ onto $L^{2}(\ghat)$ and $E_{x}T_{\omega} \mathcal{F} = \omega(x) \mathcal{F} E_{\omega} T_{-x}$ for all $(x,\omega)\in G\times\ghat$, it is an easy exercise to show the equality
$ \textstyle\int_{\ghat \times G} \vert \mathcal{V}_{\hat{g}} \hat{f}(\tilde\chi) \vert \, d \tilde{\chi} = \int_{G\times\ghat} \vert \mathcal{V}_{g}f(\chi) \vert \, d\chi < \infty$. This equality together with Theorem~\ref{th:s0-ABC} implies that $\hat{f}$ belongs to $\SO(\ghat)$. \\
(ii). For all $\nu\in G\times\ghat$ we find that 
\[ \textstyle \int_{G\times\ghat} \vert \mathcal{V}_{g}\pi(\nu)f (\chi) \vert \, d\chi = \int_{G\times \ghat} \vert \langle f, \pi(\chi-\nu) g \rangle \vert \, d\chi = \int_{G\times\ghat} \vert \mathcal{V}_{g}f(\chi) \vert \, d\chi < \infty.\]
By Theorem~\ref{th:s0-ABC} the function $\pi(\nu) f$ belongs to $\SO(G)$.\\
(iii). The definition of $\mathscr{A}$ and $\mathscr{C}$ together with Theorem \ref{th:s0-ABC} imply the inclusion. \\
(iv). The result follows from Theorem~\ref{th:s0-ABC} together with the equalities
\begin{align*} 
\textstyle \int_{G\times\ghat} \vert \langle f, \pi(\chi) g \rangle \vert \, d\chi & = \int_{G\times\ghat} \vert \langle f^{r}, \pi(\chi) g^{r} \rangle \vert \, d\chi \\
& = \textstyle \int_{G\times\ghat} \vert \langle \overline{f}, \pi(\chi) \overline{g} \rangle \vert \, d\chi = \int_{G\times\ghat} \vert \langle f^{\dagger}, \pi(\chi) g^{\dagger} \rangle \vert \, d\chi.
\end{align*}
(v). This follows from the equality $\int_{G\times \ghat} \vert \langle f, \pi(\chi) g \rangle \vert \, d\chi = \int_{G\times \ghat} \vert \langle g, \pi(\chi) f \rangle \vert \, d\chi$. \\
(vi). For the next result we note that the proof of Theorem \ref{th:s0-ABC} establishes the inequalities
\begin{equation} \label{eq:lieb1infntyF} \Vert \hat{f} \Vert_{1} \Vert \hat{g} \Vert_{\infty} \le \Vert \mathcal{V}_{g}f \Vert_{1} \ \ \text{and} \ \ \Vert \hat{g} \Vert_{1} \Vert \hat{f} \Vert_{\infty} \le \Vert  \mathcal{V}_{g}f \Vert_{1}, \end{equation}
and the inequalities
\begin{equation} \label{eq:lieb1infnty} \Vert f \Vert_{1} \Vert g \Vert_{\infty} \le \Vert \mathcal{V}_{g}f \Vert_{1} \ \ \text{and} \ \ \Vert g \Vert_{1} \Vert f \Vert_{\infty} \le \Vert  \mathcal{V}_{g}f \Vert_{1}. \end{equation}
Furthermore, recall the following implication of H\"older's inequality
\begin{equation}\label{eq:lp-conexity} \Vert f \Vert_{p} \le \Vert f \Vert_{1}^{1/p} \, \Vert f \Vert_{\infty}^{1-1/p} \ \text{for all } f\in L^{1}(G)\cap L^{\infty}(G) \ \ \text{and} \ \ p\in [1,\infty].\end{equation}
Combining \eqref{eq:lieb1infntyF},\eqref{eq:lieb1infnty} and \eqref{eq:lp-conexity} implies that for all $f\in \SO(G)$
\begin{align*}
& \Vert f \Vert_{p} \le \Vert f \Vert_{1}^{1/p} \, \Vert f \Vert_{\infty}^{1-1/p} \le \Vert g \Vert_{\infty}^{-1/p} \, \Vert \mathcal{V}_{g}f \Vert_{1}^{1/p} \, \Vert g \Vert_{1}^{-1+1/p} \, \Vert \mathcal{V}_{g}f \Vert_{1}^{1-1/p}, \\
& \Vert \hat{f} \Vert_{p} \le \Vert \hat{f} \Vert_{1}^{1/p} \, \Vert \hat{f} \Vert_{\infty}^{1-1/p} \le \Vert \hat{g} \Vert_{\infty}^{-1/p} \, \Vert \mathcal{V}_{g}f \Vert_{1}^{1/p} \, \Vert \hat{g} \Vert_{1}^{-1+1/p} \, \Vert \mathcal{V}_{g}f \Vert_{1}^{1-1/p}. 
\end{align*}
(vii). Note that $\Vert g \Vert_{q} \le \Vert g \Vert_{1}^{1-1/p} \, \Vert g \Vert_{\infty}^{1/p}$. This together with (vi) implies the result. \\
(viii) This follows from (vii) together with the Hausdorff-Young inequality $\Vert \hat{g} \Vert_{p} \le \Vert g \Vert_{q}$ for $p\in[2,\infty]$, $1/p+1/q=1$.\\
(ix). Using the inequality in (vii) we find that
\[ \vert \langle f, g\rangle \vert \le \Vert f \Vert_{1} \, \Vert g \Vert_{\infty} \le \Vert \mathcal{V}_{g}f \Vert_{1}.\] 
(x). The inequality in (ix) yields that $\Vert \mathcal
V_{g} f \Vert_{\infty} \le \Vert f \Vert_{1} \, \Vert g \Vert_{\infty} \le \Vert \mathcal
V_{g}f \Vert_{1}$.
H\"older's inequality implies that, for all $p\in[1,\infty]$,
\[ \Vert \mathcal
V_{g} f \Vert_{p} \le \Vert \mathcal
V_{g} f \Vert_{1}^{1/p} \, \Vert \mathcal
V_{g} f \Vert_{\infty}^{1-1/p} \le \Vert \mathcal
V_{g} f \Vert_{1}.\]
 \end{proof}

Corollary~\ref{co:f-in-s0-properties} states several convenient properties for functions in $\SO(G)$. However, from the description of $\SO(G)$ so far, it is not yet clear that any non-trivial function belongs to the Feichtinger algebra.  
Lemma~\ref{le:sufficient-cond-to-be-in-S0} settles this question positively, and it gives sufficient conditions for functions to belong to $\SO(G)$. Furthermore, Lemma~\ref{le:sufficient-cond-to-be-in-S0} shows that $\SO(G)$ contains approximate identities for $L^{1}(G)$ with respect to convolution and approximate identities for $A(G)$ with respect to pointwise multiplication.

\begin{lemma} \label{le:sufficient-cond-to-be-in-S0} \label{le:s0-contains-special-functions} 
\label{le:S0-contains-piecewise-constant-functions} $ \ $ 
\begin{enumerate}[(i)]
\item If $f\in C_{c}(G)$ and $\hat{f} \in L^{1}(\ghat)$, then $f\in \SO(G)$. 
\item If $f\in L^{1}(G)$ and $\hat{f}\in C_{c}(\ghat)$, then $f\in \SO(G)$.
\item If $g,h\in L^{2}(G)$ have compact support, then $g*h \in \SO(G)$.
\item For any compact set $K$ in $G$ there exists a function $f\in \SO(G)$ such that $f\in C_{c}(G)$ and $f\vert_{K} = 1$.
\item There exists a function $f\in \SO(G)$ such that $\Vert f \Vert_{1} = \Vert f \Vert_{\infty} = \Vert \hat{f} \Vert_{1} = \Vert \hat{f} \Vert_{\infty} = 1$ and $f,\hat{f}\ge 0$.
\item For any neighbourhood $U$ of the identity there exists a function $f\in \SO(G)$ such that $\supp\, f \subseteq U$, $\Vert f \Vert_{1} = 1$ and $f,\hat{f}\ge 0$.
\item For any $\epsilon>0$ and any $f\in L^{1}(G)$, there exists a function $g\in \SO(G)$ such that $\Vert g*f - f\Vert_{1} < \epsilon$. Moreover, $g$ can be taken to be compactly supported, non-negative and such that $\Vert g \Vert_{1} = 1$.
\item For any $\epsilon>0$ and any $f\in A(G)$ there exists a function $g\in \SO(G)$ such that $\Vert g\cdot f - f \Vert_{A(G)} < \epsilon$. Moreover, $g$ can be taken such that $\hat{g}$ is compactly supported, non-negative and $\Vert g \Vert_{A(G)} = 1$.
\end{enumerate}
\end{lemma}
\begin{proof} 
(i). By assumption $f\in A(G)$. 
Since $f$ has compact support, there is a compact set $K\subset G$ such that $s\mapsto f(s-x)f(s)$ is equal to zero for $x\in G\backslash K$. We therefore have that
\begin{align*}
& \int_{G} \Vert T_{x} f \cdot f \Vert_{A(G)} \, dx = \int_{K} \Vert T_{x} f \cdot f \Vert_{A(G)} \, dx \\
& \le \int_{K} \Vert T_{x} f \Vert_{A(G)} \cdot \Vert f \Vert_{A(G)} \, dx  = \Vert f \Vert_{A(G)}^{2} \int_{K} \, dx < \infty.
\end{align*}
Hence $f\in\mathscr{C}$ where $g=f$. By Theorem \ref{th:s0-ABC} $f$ is a function in $\SO(G)$. \\
(ii). The combination of (i) and Corollary~\ref{co:f-in-s0-properties}(i) yields the result. \\
(iii). Note that $g*h\in C_{c}(G)$ \cite{ru62}. Since $g,h\in L^{2}(G)$ and since the Fourier transform is a unitary operator on $L^{2}(G)$, we have that $\hat{g},\hat{h}\in L^{2}(G)$. Thus $\mathcal{F}(g*h) = \hat{g}\cdot \hat{h} \in L^{1}(G)$. The result now follows from (i).\\
(iv). Take $g$ to be in $L^{\infty}(G)$ and have compact support such that $\int_{G} g(x) \, dx = 1$. Furthermore, let $h$ be the indicator function on the set $K-\supp \, g$ and define $f = g*h$. Note that $f\in C_{c}(G)$ and by (iii) also $f\in \SO(G)$.
Furthermore, for $x\in K$ we have that
\[ f(x) = \int_{G} g(s) h(x-s) \, ds = \int_{\supp \, g} g(s) h(x-s) \, ds = \int_{\supp\, g} g(s) \, ds = 1.\]
(v). Let $K$ be a compact set such that $\mu_{G}(K)=1$. Let $h=
\mathds{1}_{K}$ and define $f=h * h^{\dagger}$. It follows by (iii) that $f\in \SO(G)$. Note that $ f(x) = \int_{G} \mathds{1}_{K}(s) \mathds{1}_{K}(s-x) \, ds$. It is therefore clear that $\Vert f \Vert_{\infty} = f(0) = 1$. Furthermore, 
\[ \int_{G} \vert f(x) \vert \, dx = \int_{G}  \int_{G} \mathds{1}_{K}(s) \mathds{1}_{K}(s-x) \, ds \, dx = \int_{G} \mathds{1}_{K}(s) \int_{G} \mathds{1}_{K}(s-x) \, dx \, ds = 1.\]
The interchange of the order of the integrals is allowed since the integrand is positive. Concerning the Fourier transform $\hat{f}$ we find that
\[ \hat{f}(\omega) = \mathcal{F}(h*h^{\dagger})(\omega) = \vert \hat{h}(\omega) \vert^2 \ge 0 \ \ \text{for all } \, \omega\in \ghat.\]
Because the Fourier transform is a unitary operator from $L^{2}(G)$ onto $L^{2}(\ghat)$ and $h\in L^{2}(G)$, we easily find that
\[ \int_{\ghat} \vert \hat{f}(\omega) \vert \, d\omega = \int_{\ghat} \vert \hat{h}(\omega) \vert^2 \, d\omega = \Vert \hat{h} \Vert_{2}^{2} = \Vert h \Vert_{2}^{2} = \int_{G} \vert \mathds{1}_{K}(x) \vert^2 \, dx = 1.\]  
This shows that $\Vert \hat{f} \Vert_{1} = 1$. Concerning $\Vert \hat{f} \Vert_{\infty}$ observe that due to the fact that $f\ge 0$  we have the inequality 
\[ \vert \hat{f}(\omega) \vert \le \int_{G} \vert f(x) \, \omega(x) \vert \, dx = \int_{G} f(x) \,dx = \hat{f}(0) = 1 \ \ \text{for all }  \, \omega\in \ghat.\]
This shows that $\Vert \hat{f} \Vert_{\infty} = 1$.\\
(vi). Take $V$ to be a neighbourhood around the identity in $G$ such that $V+(-V) \subseteq U$. Then let $h = \mathds{1}_{V}$ and define $f = c^{-1} \cdot h * h^{\dagger}$, where $c= \Vert h * h^{\dagger} \Vert_{1}$. Note that $\supp \, f \subseteq V + (-V) \subseteq U$. It is clear that $f(x) \ge 0$ for all  $x\in G$ and $\Vert f \Vert_{1} = \int_{G} f(x) \, dx = 1$. Furthermore, $\hat{f}(\omega) = c^{-1} \, \vert \hat{h}(\omega)\vert^{2} \ge 0$ for all $\omega\in \ghat$. Finally, by (iii) we have that $f\in \SO(G)$. \\
(vii). Given $\epsilon>0$ there exists a neighbourhood $U_{\epsilon}$ of the identity in $G$ such that $\Vert T_{s} f - f \Vert_{1} \le \epsilon$ for all $s\in U_{\epsilon}$ \cite{MR1802924}. Take now $g\in \SO(G)$ as in (vi), \ie $\supp\, g\subseteq U_{\epsilon}$ and $g\ge 0$ with $\int_{G} g(x) \, dx = 1$. We then find that
\begin{align*}
& \Vert g * f - f \Vert_{1} = \int_{G} \vert (g*f)(x) - f(x) \vert \, dx = \int_{G} \Big\vert \int_{G} g(s) f(x-s) \, ds - \int_{G} g(s) \, ds \, f(x) \Big\vert\, dx \\
& \le \int_{G} g(s) \int_{G} \vert f(x-s)-f(x) \vert \, dx \, ds = \int_{U_{\epsilon}} g(s) \, \Vert T_{s} f - f \Vert_{1} \, ds \le \epsilon.
\end{align*}
(viii). Let $f\in A(G)$. There exists a function $h_{f} \in L^{1}(\ghat)$ such that $\mathcal{F}h_{f} = f$. Furthermore, by (vii) there exists a function $h_{g} \in \SO(\ghat)$ such that $\Vert h_{g} * h_{f} - h_{f} \Vert_{1} \le \epsilon$. Define $g= \mathcal{F}h_{g}$. By Corollary~\ref{co:f-in-s0-properties} the function $g$ belongs to $\SO(G)$ and, indeed, we find the desired estimate,
\[ \Vert g \cdot f - f \Vert_{A(G)} = \Vert \mathcal{F} (h_{g} * h_{f} - h_{f}) \Vert_{A(G)} \stackrel{ \text{(def)}}{=} \Vert h_{g} * h_{f} - h_{f} \Vert_{1} \le \epsilon.\]
\end{proof}

Next to the sufficient conditions given in Lemma~\ref{le:sufficient-cond-to-be-in-S0}, we mention that Pogunkte \cite{po80-1} has shown that the Schwartz functions are contained in $\SO(\R^n)$. In \cite{Feichtinger1981} it was shown that for any locally compact abelian group $G$ the Schwartz-Bruhat space is contained in $\SO(G)$. A proof of this can also be found in the book by Reiter \cite{re89}. Furthermore, Gr\"ochenig \cite{gr96} has shown that enforcing decay conditions on a function in $L^{2}(\R^n)$ and its Fourier transform also yields sufficient conditions to belong to $\SO(\R^n)$. See also the papers by Okoudjou \cite{ok04} and Hogan and Lakey \cite{hola01}.

\subsection{$\SO$ as a normed vector space}
\label{sec:s0-normed-space}

In this subsection, among other results, we continue with further characterizations of $\SO(G)$.  In Theorem~\ref{th:SO-six-def} below we show that the Feichtinger algebra coincides with the sets $\mathscr{D},\mathscr{E},\mathscr{F},\mathscr{G},\mathscr{H}$ and $\mathscr{I}$ from Definition~\ref{def:S0-bigdef}. These characterizations allow us to define a norm on $\SO(G)$. Furthermore, in Theorem~\ref{th:s0-banach-space} and Corollary~\ref{co:S0-ideal-of-L1-AG}, we show that $\SO(G)$ is a Banach space with respect to this norm and in fact a Banach algebra under pointwise multiplication and convolution. %Finally, in Theorem~\ref{th:S0-is-Segal} we show that $\SO(G)$ is a Segal algebra.

Before we can continue, we need the following lemma concerning the short-time Fourier transform. This lemma is essential and appears in a similar form in, e.g., \cite{fegr92-1,MR1843717,go11}

\begin{lemma} \label{le:stft-S0-inequality} For $f_{1},f_{2},g_{1},g_{2}\in L^2(G)$ and $\nu\in G \times \ghat$
\[ \vert \langle g_{1},g_{2}\rangle\vert \cdot \vert \mathcal{V}_{f_{2}}f_{1}(\nu)  \vert \le \int_{G\times\ghat} \big\vert \mathcal{V}_{g_{1}}f_{1}(\chi) \cdot \mathcal{V}_{g_{2}}f_{2}(\chi-\nu) \big\vert \, d\chi .\]
\end{lemma}
\begin{proof} The orthogonality relation of the short-time Fourier transform \eqref{eq:STFT-dual} easily yields the inequality
\[ \vert \langle g_{1},g_{2}\rangle \langle f_{1},\pi(\nu) f_{2} \rangle \vert \le \int_{G\times\ghat} \big\vert \big( \mathcal{V}_{g_{1}}f_{1}\big)(\chi) \cdot \big(\mathcal{V}_{g_{2}}\pi(\nu) f_{2}\big)(\chi) \big\vert \, d\chi .\]
The observation that $\vert (\mathcal{V}_{g_{2}}\pi(\nu) f_{2})(\chi) \vert =
%= \vert \langle \pi(\nu)f_{2}, \pi(\chi)g_{2} \rangle\vert 
 \vert \mathcal{V}_{g_{2}}f_{2}(\chi-\nu) \vert$
finishes the proof.
\end{proof}

The characterization of $\SO(G)$ by the set $\mathscr{B}$ together with Corollary \ref{co:f-in-s0-properties}(v), implies that for any given $f\in \SO(G)$ there exists \emph{some} non-zero $g\in \SO(G)$ such that $\mathcal{V}_{g}f \in L^{1}(G\times\ghat)$. The inequality in Lemma~\ref{le:stft-S0-inequality} allows us to prove the following. 

\begin{proposition} \label{pr:s0-Vgf-well-defined}
If $f,g\SO(G)$, then $\mathcal{V}_{g}f\in L^{1}(G\times\ghat)$.
\end{proposition}
\begin{proof} Let us first show that if  $f\in \SO(G)$, then $\Vert \mathcal{V}_{f}f\Vert_{1}< \infty$. If $f\in \SO(G)$, then by 
Theorem~\ref{th:s0-ABC} there exists a function $h\in L^{2}(G)\backslash\{0\}$ such that
$ \Vert \mathcal{V}_{h}f \Vert_{1} < \infty$. By use of Lemma~\ref{le:stft-S0-inequality} we find that
\[ \Vert h \Vert_{2}^{2} \, \vert \mathcal{V}_{f} f(\nu) \vert \le \int_{G\times\ghat} \vert \mathcal{V}_{h} f(\chi) \vert \, \vert \mathcal{V}_{h}f(\chi-\nu) \vert \, d\chi \ \ \text{for all} \ \ \nu\in G\times\ghat.\]
Integrating over $\nu\in G\times\ghat$ and using Fubini yields the inequality
$ \Vert \mathcal{V}_{f}f \Vert_{1}\le \Vert h \Vert_{2}^{-2} \Vert \mathcal{V}_{h}f\Vert_{1}^{2}$.
Hence $f\in \SO(G)$ implies that $\mathcal{V}_{f}f\in L^{1}(G\times\ghat)$. Assume now that $f,g\in \SO(G)$. By a similar calculation as above, we then find that
\[ \vert \langle f,g\rangle \vert \, \Vert \mathcal{V}_{g}f \Vert_{1} \le \Vert \mathcal{V}_{f} f \Vert_{1} \, \Vert \mathcal{V}_{g}g \Vert_{1} < \infty.\]
Hence if $\langle f,g\rangle \ne 0$ then $\Vert \mathcal{V}_{g}f \Vert_{1} < \infty$.
If $\langle f,g\rangle = 0$ then for any $h\in \SO(G)$ such that $\langle h, f\rangle \ne 0$ and $\langle h, g\rangle \ne 0$ one has that
\[ \Vert \mathcal{V}_{g}f \Vert_{1} \le \big( \, \Vert h \Vert_{2}^{2} \, \vert \langle h,f\rangle \langle h,g\rangle \vert \, \big)^{-1} \Vert \mathcal{V}_{h}h\Vert_{1}^{2} \, \Vert \mathcal{V}_{f}f\Vert_{1} \, \Vert \mathcal{V}_{g}g\Vert_{1} < \infty.\]\end{proof}

The result of Proposition~\ref{pr:s0-Vgf-well-defined} together with 
Corollary~\ref{co:f-in-s0-properties}(x) allows us to show an uncertainty result concerning the time-frequency concentration of the short-time Fourier transform $\mathcal
V_{g}f$ of functions $f,g\in \SO(G)$. In particular, 
Proposition~\ref{pr:uncertain} shows that there is a lower bound to the area in the time-frequency plane which the short-time Fourier transform can occupy. For $\R^d$ a sharper lower bound is possible, see \cite{li90-1,gr98,gr03-2}. For further results on uncertainty principles see, \eg the book by Hogan and Lakey \cite{hola05}.
\begin{proposition} \label{pr:uncertain} Let $f,g\in \SO(G)\backslash\{0\}$. If a subset $\mathcal{U}\subseteq G\times\ghat$ is given such that
\[ (1-\epsilon) \, \Vert \mathcal{V}_{g}f \Vert_{1} \le \textstyle\int_{\mathcal{U}} \vert \mathcal{V}_{g}f(\chi) \vert \, d\mu_{G\times\ghat}(\chi),\]
for some $\epsilon > 0$, then $\mu_{G\times\ghat} (\mathcal{U}) \ge (1-\epsilon)$.
\end{proposition}
\begin{proof}
By assumption and the H\"older inequality, we can estimate that
\[ (1-\epsilon) \, \Vert \mathcal{V}_{g}f \Vert_{1} \le 
\int_{\mathcal{U}} \vert \mathcal{V}_{g}f(\chi) \vert \, d\mu_{G\times\ghat}(\chi) \le \mu_{G\times\ghat}(\mathcal{U}) \, \Vert \mathcal{V}_{g} f \Vert_{\infty} \stackrel{\textnormal{Cor.\ref{co:f-in-s0-properties}(x)} }{\le} \mu_{G\times\ghat}(\mathcal{U}) \, \Vert \mathcal{V}_{g} f \Vert_{1}. \]
\end{proof}

A combination of Theorem~\ref{th:s0-ABC} and Proposition~\ref{pr:s0-Vgf-well-defined} allows us to show that $\SO(G)$ coincides with the sets $\mathscr{D,E,F,G,H}$ and $\mathscr{I}$ from Definition~\ref{def:S0-bigdef}.
%\begin{enumerate}[]
%\item $\quad\mathscr{D} = \{ f \in L^{1}(G) \, : \, \int_{\ghat} \Vert E_{\omega} f * f \Vert_{1} \, d\omega< \infty\}$,
%\item $\quad\mathscr{E} = \{ f \in L^{2}(G) \, : \, \int_{G\times \ghat} \vert \mathcal{V}_{f} f (\chi) \vert \, d\chi < \infty\}$,
%\item $\quad\mathscr{F} = \{ f \in A(G) \, : \, \int_{G} \Vert T_{x} f \cdot f \Vert_{A(G)} < \infty\}$.
%\end{enumerate}
%For a fixed $g\in\SO(G)\backslash\{0\}$ we define
%\begin{enumerate}[]
%%\setcounter{enumi}{4}
%\item $\quad\mathscr{G} = \{ f \in L^{1}(G) \, : \, \int_{\ghat} \Vert E_{\omega} f * g \Vert_{1} \, d\omega< \infty\}$,
%\item $\quad\mathscr{H} = \{ f \in L^{2}(G) \, : \, \int_{G\times \ghat} \vert \mathcal{V}_{g} f (\chi) \vert \, d\chi < \infty\}$,
%\item $\quad\mathscr{I} = \{ f \in A(G) \, : \, \int_{G} \Vert T_{x} f \cdot g \Vert_{A(G)} < \infty\}$.
%\end{enumerate}
In particular, either of the characterization of $\SO(G)$ via the set $\mathscr{G}$, $\mathscr{H}$ or $\mathscr{I}$ implies that the Feichtinger algebra is a linear vector space. The set $\mathscr{H}$ is often taken as the definition of $\SO(G)$. However, this requires prior knowledge of a function $g\in \SO(G)$. If $G=\R^n$ this initial function $g$ is usually taken to be the Gaussian $g(x) = e^{-\pi x\cdot x}$. 
\begin{theorem} \label{th:SO-six-def} For any locally compact abelian group $G$ it holds that $\SO(G) = \mathscr{D} = \mathscr{E} = \mathscr{F} = \mathscr{G} = \mathscr{H} = \mathscr{I}$.
\end{theorem}
\begin{proof} If $f$ belongs to any of the sets  $\mathscr{D}, \mathscr{E}, \mathscr{F}, \mathscr{G} , \mathscr{H}$ or $\mathscr{I}$, then the characterizations of $\SO(G)$ in Theorem~\ref{th:s0-ABC} imply that $f\in \SO(G)$. Conversely, if $f,g\in \SO(G)$, then Corollary~\ref{co:f-in-s0-properties} implies that $\overline{f},f^{\dagger},\overline{g},g^{\dagger}\in \SO(G)$. Now Proposition~\ref{pr:s0-Vgf-well-defined} implies that the function $\mathcal{V}_{f^{\dagger}}f$ belongs to $L^{1}(G\times \ghat)$. Recall that we in Theorem \ref{th:s0-ABC} showed that
\begin{align}
 \int_{\ghat} \Vert E_{\omega} f * g \Vert_{1} \, d\omega = \int_{G\times\ghat} \vert \mathcal{V}_{g^{\dagger}}f(\chi)  \vert \, d\chi = \int_{G} \Vert T_{x} f \cdot g^{r} \Vert_{A(G)} \, dx\label{eq:0703ccc}
\end{align} for all $f,g\in L^{1}(G)\cap L^{2}(G) \cap A(G)$, thus in particular for $f,g\in\SO(G)$. By equation \eqref{eq:0703ccc} the function $\mathcal{V}_{f^{\dagger}}f$ belongs to $L^{1}(G\times \ghat)$ if and only if $\int_{\ghat} \Vert E_{\omega} f * f \Vert_{1} \, d\omega < \infty$, hence $\SO(G)\subseteq \mathscr{D}$. We have thus shown that $\SO(G) = \mathscr{D}$. In a similar way, Proposition~\ref{pr:s0-Vgf-well-defined} implies that $\mathcal{V}_{f}f $, $\mathcal{V}_{\overline{f}}f$, $\mathcal{V}_{g^{\dagger}}f$, $\mathcal{V}_{g}f$ and $\mathcal{V}_{\overline{g}}f$ belong to $L^{1}(G\times \ghat)$. This together with the relations in \eqref{eq:0703ccc} implies that $\SO(G)$ is contained in $\mathscr{E},\mathscr{F},\mathscr{G},\mathscr{H}$ and $\mathscr{I}$, respectively. Hence $\SO(G)$ coincides with all the sets $\mathscr{D}$-$\mathscr{I}$.\end{proof}

\begin{remark} In the definition of the sets $\mathscr{G},\mathscr{H}$ and $\mathscr{I}$ it is important that $g\in \SO(G)$. In fact, if $g\notin \SO(G)$, then the sets $\mathscr{G},\mathscr{H}$ and $\mathscr{I}$ only contain the zero function on $G$.
\end{remark}

We are now in the position to define a norm on $\SO(G)$.
\begin{definition} \label{def:s0-norm} For a fixed function $g\in \SO(G)\backslash \{0\}$ we define the $\SO$-norm with respect to $g$ by
\[ \Vert \cdot \Vert_{\SO(G),g}: \SO(G) \to \R^{+}_{0} , \ \ \Vert f \Vert_{\SO(G),g} = \Vert \mathcal{V}_{g}f\Vert_{1}.\]
\end{definition}
By Proposition~\ref{pr:s0-Vgf-well-defined} the mapping $\Vert \cdot \Vert_{\SO(G),g}:\SO(G)\to \R_{0}^{+}$ is well-defined. We leave it to the reader to verify that $\Vert \cdot \Vert_{\SO(G),g}$ satisfies the norm axioms.
To ease notation we may write $\Vert \cdot \Vert_{\SO,g}$ or $\Vert \cdot \Vert_{\SO}$ instead of $\Vert \cdot \Vert_{\SO(G),g}$.
From the definition of the norm $\Vert \cdot \Vert_{\SO,g}$ one can easily verify the following relations:
\begin{equation} \label{eq:2711a} \Vert f \Vert_{\SO,g} = \Vert g \Vert_{\SO,f} = \Vert \overline{f} \Vert_{\SO,\overline{g}} = \Vert f^{\dagger} \Vert_{\SO,g^{\dagger}} = \Vert \hat{f} \Vert_{\SO(\ghat),\hat{g}} \ \text{ for all } f\in \SO(G).\end{equation}
Moreover, time-frequency shifts leave the norm invariant, i.e, 
\begin{equation} \label{eq:0812a} \Vert \pi(\nu_{1}) f \Vert_{\SO, \pi(\nu_{2})g} = \Vert f \Vert_{\SO,g}  \  \text{ for all } \nu_{1},\nu_{2}\in G\times\ghat \text{ and } f\in \SO(G).\end{equation}
Also, for all $f\in \SO(G)$ one can show that
\begin{align}
  \Vert f \Vert_{\SO,g} & = \textstyle\int_{\ghat} \Vert f * E_{\omega} g^{\dagger} \Vert_{1} \, d\omega = \textstyle\int_{\ghat} \Vert E_{\omega} f * g^{\dagger} \Vert_{1} \, d\omega = \textstyle\int_{\ghat} \Vert E_{\omega} f^{\dagger} * g \Vert_{1} \, d\omega \nonumber \\
 & = \textstyle\int_{G} \Vert f \cdot T_{x} \overline{g} \Vert_{A(G)} \, dx = \textstyle\int_{G} \Vert T_{x}f \cdot \overline{g} \Vert_{A(G)} = \textstyle\int_{G} \Vert T_{x}\overline{f} \cdot g \Vert_{A(G)}.\label{eq:s0-norm-with-modulation-space-and-wiener-amalgam}
\end{align}

Even though the norm $\Vert \cdot \Vert_{\SO,g}$ depends on the function $g\in \SO(G)\backslash\{0\}$, it turns out that all norms constructed as in Definition~\ref{def:s0-norm} are equivalent. 

\begin{proposition}\label{pr:equivalent-norm-on-s0} Let $g_{1},g_{2}\in \SO(G)\backslash\{0\}$. The norms $\Vert \cdot \Vert_{\SO,g_{1}}$ and $\Vert \cdot \Vert_{\SO,g_{2}}$ on $\SO(G)$ are equivalent.
To be precise, for all $f\in \SO(G)$ one has the inequalities 
\[ c \, \Vert f \Vert_{\SO,g_{2}} \le \Vert f \Vert_{\SO,g_{1}} \le C \, \Vert f \Vert_{\SO,g_{2}}, \]
with $c = \Vert g_{1} \Vert_{2}^{2} \, \Vert g_{2} \Vert_{\SO,g_{1}}^{-1}$ and $ C = \Vert g_{2}\Vert_{2}^{-2} \, \Vert g_{1} \Vert_{\SO,g_{2}}$.
\end{proposition}
\begin{proof} By use of Lemma~\ref{le:stft-S0-inequality} we find that
\[ 
\Vert g_{2} \Vert_{2}^{2} \, \vert \mathcal{V}_{g_{1}}f(\nu)  \vert \le \textstyle\int_{G\times\ghat} \big\vert  \mathcal{V}_{g_{2}}f(\chi) \cdot \mathcal{V}_{g_{2}}g_{1}(\chi-\nu) \big\vert \, d\chi.  \]
An integration over $\nu\in G\times\ghat$ yields the inequality $\Vert g_{2} \Vert_{2}^{2} \, \Vert f \Vert_{\SO,g_{1}} \le \Vert g_{1} \Vert_{\SO,g_{2}} \, \Vert f \Vert_{\SO,g_{2}}$.
%\[ \Vert g_{2} \Vert_{2}^{2} \, \Vert f \Vert_{\SO,g_{1}} \le \Vert g_{1} \Vert_{\SO,g_{2}} \, \Vert f \Vert_{\SO,g_{2}}.  \]
Interchanging the role of $g_{1}$ and $g_{2}$ yields the relation $\Vert g_{1} \Vert_{2}^{2} \, \Vert f \Vert_{\SO,g_{2}} \le \Vert g_{2} \Vert_{\SO,g_{1}} \, \Vert f \Vert_{\SO,g_{1}}$.
%\[ \Vert g_{1} \Vert_{2}^{2} \, \Vert f \Vert_{\SO,g_{2}} \le \Vert g_{2} \Vert_{\SO,g_{1}} \, \Vert f \Vert_{\SO,g_{1}}.  \]
A combination of these two inequalities gives the desired result.
\end{proof}
%The equivalence between the norms induced by all $g\in \SO(G)\backslash\{0\}$ as presented in Proposition~\ref{pr:equivalent-norm-on-s0} is a special case of a result that holds (with a similar proof) for all modulation spaces, see, \eg \cite{MR1843717}. 

By Corollary \ref{co:f-in-s0-properties}(iii) we have the inclusion $\SO(G)\subseteq L^{1}(G)\cap A(G)$. If $G$ is discrete or compact, then it is easy to show that equality holds.
\begin{lemma} \label{le:s0-discrete-compact} $ \ $
\begin{enumerate}
\item[(i)] Let $G$ be a discrete abelian group equipped with the counting measure. If $g=\delta_{0}$, then $\Vert f \Vert_{\SO,g} = \Vert f \Vert_{1}$. Hence $\SO(G) = \ell^{1}(G) \ (=\ell^{1}(G)\cap A(G))$. 
\item[(ii)] Let $G$ be a compact abelian group equipped with its normalized Haar measure. If $g = \mathds{1}_{G}$, then $\Vert f \Vert_{\SO,g} = \Vert f \Vert_{A(G)}$. Hence $\SO(G) = A(G) \ (=L^{1}(G)\cap A(G))$.
\end{enumerate}
\end{lemma}
%In fact, $\SO$ is a proper subspace of $L^{1}(G)\cap A(G)$ exactly if $G$ is neither discrete nor compact, see Remark \ref{rem:L1AG}.

Let us now show that $\SO(G)$ is a Banach space. The proof of 
Theorem~\ref{th:s0-banach-space} is an adaptation of a proof for the same result for Wiener amalgam spaces given by Heil in \cite{he03}.

\begin{theorem} \label{th:s0-banach-space}
For any $g\in \SO(G)\backslash\{0\}$ the vector space $\SO(G)$ is complete with respect to the norm $\Vert \cdot \Vert_{\SO,g}$. 
\end{theorem}
\begin{proof} It suffices to show that every absolute convergent series in $\SO(G)$ has a limit in $\SO(G)$. Let therefore $\{f_{n}\}_{n\in\N}$ be a sequence in $\SO(G)$ such that $\sum_{n\in\N} \Vert f_{n} \Vert_{\SO,g}< \infty$. We now wish to show that there exists a function $\Phi \in \SO(G)$ such that
\begin{equation} \label{eq:1910a} \big\Vert \, \Phi - \textstyle\sum_{n=1}^{N} f_{n} \,  \big\Vert_{\SO,g} \to 0 \text{ as } N\to \infty.\end{equation}
By Corollary \ref{co:f-in-s0-properties}(vi.b) for $p=1$ it follows that
$\sum_{n\in \N} \Vert f_{n} \Vert_{A(G)} \le c \sum_{n\in \N} \Vert f_{n} \Vert_{\SO,g} < \infty$,
with $c = \Vert \hat{g} \Vert_{\infty}^{-1}$. Hence $\{ f_{n} \}_{n\in \N}$ is absolutely convergent in the Banach space $A(G)$ and therefore $\sum_{n\in \N} f_{n}$ is convergent in $A(G)$.
We now want to show that $\Phi = \sum_{n\in \N} f_{n}$ belongs to $\SO(G)$.
By Theorem~\ref{th:SO-six-def} the function $\Phi\in A(G)$ belongs to $\SO(G)$ if, and only if $ \int_{G} \Vert \Phi \cdot T_{x} \overline{g }\Vert_{A(G)} \, dx<\infty$.
We have that
\[
\int_{G} \Vert \Phi \cdot T_{x}\overline{g} \Vert_{A(G)} \, dx \le \sum_{n\in \N} \int_{G} \Vert  f_{n} \cdot T_{x}\overline{g} \Vert_{A(G)} \, dx = \sum_{n\in \N} \Vert f_{n} \Vert_{\SO,g} < \infty.\]
This shows that $\Phi \in \SO(G)$. It is now straightforward to show that \eqref{eq:1910a} holds. 
\end{proof}
A similar proof of Theorem~\ref{th:s0-banach-space} is possible with the use of the equality $\Vert \Phi \Vert_{\SO,g} = \int_{\ghat} \Vert \Phi * E_{\omega} g^{\dagger} \Vert_{1} \, d\omega$. See also \cite{MR1843717} for a proof of Theorem~\ref{th:s0-banach-space} based on properties of the short-time Fourier transform. In the original paper \cite{Feichtinger1981}, the characterization of $\SO(G)$ via the set $\mathscr{M}$ is used to show that the Feichtinger algebra is a Banach space.

The equalities in \eqref{eq:s0-norm-with-modulation-space-and-wiener-amalgam} concerning the $\SO$-norm give us an easy way to show that $\SO(G)$ is an ideal of $L^{1}(G)$ and $A(G)$ under convolution and pointwise multiplication, respectively.
\begin{proposition} \label{pr:s0-l1-ag-ideal} Suppose $g\in \SO(G)\backslash\{0\}$. Then the following holds:
\begin{enumerate}[(i)]
\item $\Vert h * f \Vert_{\SO,g} \le \Vert h \Vert_{1} \Vert f \Vert_{\SO,g}$ for all $h\in L^{1}(G)$ and $f\in \SO(G)$. 
\item $\Vert h \cdot f \Vert_{\SO,g} \le \Vert h \Vert_{A(G)} \Vert f \Vert_{\SO,g}$ for all $h\in A(G)$ and $f\in \SO(G)$. 
\end{enumerate}
\end{proposition}
\begin{proof} 
Recall that $ \Vert h * f \Vert_{1} \le \Vert h \Vert_{1} \, \Vert f \Vert_{1}$ for all $f,h\in L^{1}(G)$. Hence
\[ \Vert h*f \Vert_{\SO,g} \stackrel{\eqref{eq:s0-norm-with-modulation-space-and-wiener-amalgam}}{=} \textstyle\int_{\ghat} \Vert h * f * E_{\omega}g^{\dagger} \Vert_{1} \, d\omega \le \Vert h \Vert_{1} \Vert f \Vert_{\SO,g}. \]
This shows (i). For (ii) we use that $\Vert h \cdot f \Vert_{A(G)} \le \Vert h \Vert_{A(G)} \, \Vert f \Vert_{A(G)}$ for all $f,h\in A(G)$ and establish the inequality 
\begin{align*}
\Vert h\cdot f\Vert_{\SO,g} \stackrel{\eqref{eq:s0-norm-with-modulation-space-and-wiener-amalgam}}{=} \textstyle\int_{G} \Vert h\cdot f \, \cdot \overline{T_{x}g} \Vert_{A(G)} \, dx \le \Vert h \Vert_{A(G)} \, \Vert f \Vert_{\SO,g}.\end{align*}
\end{proof}

The inequalities established in Proposition~\ref{pr:s0-l1-ag-ideal} show that $\SO(G)$ is a $L^{1}(G)$-Banach module with respect to convolution and an $A(G)$-Banach module with respect to pointwise multiplication.
In particular, since $\SO(G)\subseteq L^{1}(G)\cap A(G)$ it follows that $\SO(G)$ is a Banach algebra under pointwise multiplication and convolution. This accounts for the name of $\SO(G)$ as the Feichtinger \emph{algebra}. 
\begin{corollary} \label{co:S0-ideal-of-L1-AG} If $f_{1},f_{2}\in \SO(G)$, then $f_{1}\cdot f_{2},\,f_{1}*f_{2}\in \SO(G)$. In fact, for $g\in \SO(G)\backslash\{0\}$,
\begin{align*}
& \Vert f_{1} * f_{2} \Vert_{\SO,g} \le c \, \Vert f_{1} \Vert_{\SO,g} \, \Vert f_{2} \Vert_{\SO,g} \ \ \text{with} \ c= \Vert g \Vert_{\infty}^{-1}, \\ & \Vert f_{1} \cdot f_{2} \Vert_{\SO,g} \le c \, \Vert f_{1} \Vert_{\SO,g} \, \Vert f_{2} \Vert_{\SO,g} \ \ \text{with} \ c= \Vert \hat{g} \Vert_{\infty}^{-1}.
\end{align*}
If $g$ is chosen with the normalization as in Lemma~\ref{le:S0-contains-piecewise-constant-functions}(v), then in both cases $c=1$.
\end{corollary}
\begin{proof} Apply Corollary \ref{co:f-in-s0-properties}(vi) with $p=1$ to the inequalities in Proposition \ref{pr:s0-l1-ag-ideal}.
\end{proof}

As is the case for, e.g., $L^{1}(G)$, the following lemma shows that the translation and modulation operators are continuous mappings from $G$ and $\ghat$, respectively, into $\SO(G)$.
\begin{lemma} \label{le:time-freq-shift-continuous-on-SO} For any $f\in \SO(G)$ and $\epsilon>0$ the following holds:
\begin{enumerate}[(i)]
\item There exists a neighbourhood $U_{\epsilon} \subseteq G$ of the identity such that 
\[ \Vert T_{x} f - f \Vert_{\SO(G),g} < \epsilon \ \ \textnormal{ for all } x\in U_{\epsilon}.\]
\item There exists a neighbourhood $V_{\epsilon} \subseteq \ghat$ of the identity such that 
\[ \Vert E_{\omega} f - f \Vert_{\SO(G),g} < \epsilon \ \ \textnormal{ for all } \omega\in V_{\epsilon}.\]
\end{enumerate}
\end{lemma}
\begin{proof} We follow the proof in \cite{fezi98}. Note that $\mathcal{V}_{g} T_{x} = E_{(0,-x)}T_{(x,0)} \mathcal{V}_{g}$. This implies that
$\Vert T_{x}f -f \Vert_{\SO(G),g} = \Vert E_{(0,-x)} T_{(x,0)} \mathcal{V}_{g} f - \mathcal{V}_{g} f\Vert_{1}$. However, for $h\in L^{1}(G)$ the mappings $x\mapsto T_{x}h$ and $\omega \mapsto E_{\omega}h$ are continuous. Therefore there exists a neighbourhood $U_{\epsilon}\subseteq G$ around the identity such that
(i) holds. Statement (ii) follows in the same fashion by use of the identity $\mathcal{V}_{g} E_{\omega} = T_{(0,\omega)} \mathcal{V}_{g}$.
\end{proof}

The continuity of the translation operator in Lemma~\ref{le:time-freq-shift-continuous-on-SO} allows us to prove that the space of functions
\[ C_{c}(G)\cap A(G) = \{ f \in \SO(G) \, : \, f\in C_{c}(G) \} \] is dense in $\SO(G)$ and that it contains approximate identities for $\SO(G)$ with respect to both convolution and pointwise multiplication.
\begin{proposition} \label{pr:s0-and-ccfl1} $ \ $
\begin{enumerate}[(i)]
\item For any $f\in \SO(G)$ and $\epsilon>0$ there exists a function $h\in C_{c}(G)\cap A(G)\subseteq \SO(G)$ such that
\[ \Vert h * f - f \Vert_{\SO,g} < \epsilon.\]
\item For any $f\in \SO(G)$ and $\epsilon>0$ there exists a function $h\in C_{c}(G)\cap A(G)\subseteq \SO(G)$ such that
\[ \Vert h \cdot f - f \Vert_{\SO,g} < \epsilon.\]
\item $C_{c}(G)\cap A(G)$ is dense in $\SO(G)$.
\end{enumerate}
The same statements hold for functions $h\in L^{1}(G)$ with $\hat{h}\in C_{c}(G)$.
\end{proposition}
\begin{proof} As shown in Lemma~\ref{le:time-freq-shift-continuous-on-SO} there exists a neighbourhood $U_{\epsilon}$ of $0$ such that $\Vert T_{x} f - f \Vert_{\SO,g} < \epsilon$ for all $x\in U_{\epsilon}$. Take now $u\in C_{c}(G) \cap A(G) \subseteq \SO(G)$ as in Lemma~\ref{le:S0-contains-piecewise-constant-functions}(vi), \ie $u$ is a positive, continuous function with support in $U_{\epsilon}$ and with $\Vert u \Vert_{1} = 1$. Then $(u * f - f)(t) = \int_{G} u(s) (f(t-s)-f(t)) \, ds$.
This allows us to conclude that
\begin{equation}
\label{eq:2801a} \Vert u * f - f \Vert_{\SO,g} \le \int_{G} u(s) \, \Vert T_{s} f - f \Vert_{\SO,g} \, ds < \epsilon.\end{equation}
Indeed, with the equality $\Vert f \Vert_{\SO,g} = \int_{\ghat} \Vert f * E_{\omega} g^{\dagger} \Vert_{1} \, d\omega$ it is straightforward to argue for the reordering of integrals in \eqref{eq:2801a}. This shows (i). Now, by \cite[Proposition~5.4.1]{MR1802924} the set of functions in $L^{1}(G)$ with compactly supported Fourier transform are dense in $L^{1}(G)$, therefore there exists a function $v\in L^{1}(G)$ with $\hat{v} \in C_{c}(\ghat)$ such that $\Vert u - v \Vert_{1} < \epsilon \, \Vert f \Vert_{\SO,g}^{-1}$. These considerations yield the following estimates:
\begin{align*}
\Vert v * f - f \Vert_{\SO,g} & \le \Vert v * f - u*f \Vert_{\SO,g} + \Vert u*f-f \Vert_{\SO,g} \\
& \le \Vert v - u \Vert_{1}  \Vert f \Vert_{\SO,g} + \Vert u * f -f \Vert_{\SO,g} \le 2 \epsilon.
\end{align*}
This proves statement (i) for functions in $L^{1}(G)$ with compactly supported Fourier transform. Using the invariance of $\SO$ and its norm under the Fourier transform \eqref{eq:2711a} yields statement (ii) for functions in $C_{c}(G)\cap A(G)$ and for functions in $L^{1}(G)$ with compactly supported Fourier transform.
Statement (iii) follows from (ii) together with the observation that $h\cdot f \in C_{c}(G)\cap A(G)$. 
\end{proof}

From the inclusion $\SO(G)\subseteq L^{1}(G)\cap A(G)\subseteq C_{0}(G)$ it is clear that $\SO(G)\subseteq L^{p}(G)$ for all $p\in [1,\infty]$. In fact, we have the following result.

\begin{lemma} \label{le:cont-embedded-in-Lp} The Banach space $\SO(G)$ is continuously embedded into and dense in $C_{0}(G)$, the Fourier algebra $A(G)$ and $L^{p}(G)$ for $p\in[1,\infty[$.
Specifically, for $g\in \SO(G)\backslash\{0\}$,
\begin{enumerate}[(i)]
\item $\Vert f \Vert_{p} \le c \, \Vert f \Vert_{\SO,g}$ for all $f\in \SO(G)$, $p\in [1,\infty]$, with $c = \Vert g \Vert_{1}^{-1+1/p} \, \Vert g \Vert_{\infty}^{-1/p}$,
\item $\Vert f \Vert_{A(G)} \le c \, \Vert f \Vert_{\SO,g}$ for all $f\in \SO(G)$, with $c = \Vert \hat{g} \Vert_{\infty}^{-1}$.
\end{enumerate}
If $g$ is chosen with the normalization as in Lemma~\ref{le:S0-contains-piecewise-constant-functions}(v), then in both cases $c=1$.
\end{lemma}
\begin{proof} The inequalities follow from Corollary~\ref{co:f-in-s0-properties}(vi). The results in Corollary~\ref{co:S0-ideal-of-L1-AG} together with 
Lemma~\ref{le:S0-contains-piecewise-constant-functions}(vii) and (viii) imply that $\SO(G)$ is dense in $L^{1}(G)$ and $A(G)$. Since $L^{1}(G)\cap L^{p}(G)$ is dense in $L^{p}(G)$ for all $p\in [1,\infty[$, it follows that $\SO(G)$ is dense in those $L^{p}$-spaces. That $\SO(G)$ is dense in $C_{0}(G)$ follows from the fact that $A(G)$ is dense in $C_{0}(G)$ (see \cite[Proposition 5.4.4]{MR1802924}).
\end{proof}

While the statement of Lemma \ref{le:cont-embedded-in-Lp} is well-known, the specific constants for the continuous embedding are new. 

We end this section by showing that $\SO(G)$ is a \emph{Segal algebra}.

\begin{definition} \label{def:segal-algebra} Let $G$ be a locally compact abelian group. A linear and normed subspace $(S(G), \Vert \cdot \Vert_{S})$ of $L^{1}(G)$ is a \emph{Segal algebra} if it satisfies the following conditions:
\begin{enumerate}[(i)]
\item $S(G)$ is dense in and continuously embedded into $L^{1}(G)$.
\item $(S(G), \Vert \cdot \Vert_{S})$ is a Banach space.
\item If $f\in S(G)$, then $T_{x} f \in S(G)$ for all $x\in G$.
\item $\Vert f \Vert_{S} = \Vert T_{x} f \Vert_{S}$ for all $f\in S(G)$ and $x\in G$.
\item The translation operator is continuous from $G$ into $S(G)$, i.e., for all $f\in S(G)$ and all $\epsilon>0$ there exists a neighbourhood $U_{\epsilon}\subseteq G$ of the identity such that 
\[ \Vert T_{x} f - f \Vert_{S} < \epsilon \ \ \text{for all} \ \ x\in U_{\epsilon}.\]
\end{enumerate}
\end{definition}

Indeed, one can show that the definition implies that a Segal algebra is a Banach algebra and an ideal of $L^{1}(G)$ under convolution. For more on Segal algebras see the books by Reiter \cite{re71,re89} and Reiter and Stegeman \cite{MR1802924}.

\begin{theorem} \label{th:S0-is-Segal} The Feichtinger algebra $\SO(G)$ is a Segal algebra.
\end{theorem}
\begin{proof} In Theorem~\ref{th:s0-banach-space} and Lemma~\ref{le:cont-embedded-in-Lp} we established that $\SO(G)$ is a Banach space which is continuously embedded and dense in $L^{1}(G)$. Corollary~\ref{co:f-in-s0-properties}(ii) shows that $f\in \SO(G)$ implies that also $T_{x}f\in \SO(G)$ for all $x\in G$. The relations in \eqref{eq:0812a} show that the norm of a function $f\in \SO(G)$ is invariant under translation. Finally, Lemma~\ref{le:time-freq-shift-continuous-on-SO} shows that the mapping $x\mapsto T_{x}f$ is continuous. 
\end{proof}

One may ask if there exists a smallest among all Segal algebras. The answer to this question is negative.  
The intersection of all Segal algebras on a locally compact abelian group is exactly the space of functions in $L^{1}(G)$ with compactly supported Fourier transform, which is not a Segal algebra (unless $G$ is discrete, but then $\ell^{1}(G)$ is the only Segal algebra; compare with Lemma \ref{le:s0-discrete-compact}). 
Feichtinger discovered that there is a smallest among all Segal algebras if one adds a condition related to the modulation operator.

\begin{proposition} \label{pr:min-segal} Among all Segal algebras $S(G)$ on a locally compact abelian group $G$ with the additional property that 
\[ \Vert E_{\omega}f \Vert_{S} = \Vert f \Vert_{S}\ \ \text{for all} \ \ f\in S(G),\ \omega\in\ghat,\] the Feichtinger algebra $\SO(G)$ is the smallest. That is, $\SO(G)$ is continuously embedded into any other Segal algebra with this additional property.
\end{proposition}
\begin{proof}
This result follows from the more general result in Theorem \ref{th:SO-minimality-bochner}. We therefore omit the proof. \end{proof}

As is the case for $\SO(G)$, the space of functions in $L^{1}(G)$ with compactly supported Fourier transform is dense in any Segal algebra. 
Improving on this result, Reiter showed in \cite{re93-3} that all Segal algebras on a locally compact abelian group contain the space of Schwartz-Bruhat functions with compactly supported Fourier transform as a dense subspace.\footnote{In the same publication Reiter calls $\SO$ the \emph{canonical} Segal algebra.} As shown by Poguntke \cite{po80-1} and Feichtinger \cite{Feichtinger1981}, $\SO(G)$ contains the entire Schwartz-Bruhat space.

\section{Operators on $\SO(G)$} \label{sec:mappings-on-s0}

In this section we consider operators acting on the Feichtinger algebra. In particular, in Theorem~\ref{th:s0-invariant-under-some-automorphisms} we investigate Banach space isomorphisms on $\SO$. In Theorem~\ref{th:stft-and-s0} we consider the short-time Fourier transform $\mathcal{V}_{h}$ with $h\in \SO(G)$ and its Hilbert space adjoint in relation to the Feichtinger algebra. Lastly, in Section \ref{sec:restriction} we take a closer at the restriction, periodization and the zero-extension operator, as well as the Poisson formula.

\subsection{Isomorphisms of the Feichtinger algebra}

We begin with a remark on the Fourier transform: 
the relationship of the $\SO$-norm with the Fourier transform stated in \eqref{eq:2711a}, which followed from the identity in Corollary \ref{co:f-in-s0-properties}(i), implies that the Fourier transform is a linear and bounded operator from $\SO(G)$ into $\SO(\ghat)$. The same holds for the inverse Fourier transform.  Hence the Fourier transform is a Banach space isomorphism from $\SO(G)$ onto $\SO(\ghat)$. The crucial properties of the Fourier transform for \eqref{eq:2711a} to hold are that it commutes with time-frequency shifts in a ``nice'' way, i.e., $E_{x}T_{\omega} \mathcal{F} = \omega(x) \mathcal{F} E_{\omega} T_{-x}$ for all $(x,\omega)\in G\times\ghat$ and that it is a unitary operator from $L^{2}(G)$ onto $L^{2}(\ghat)$.

We generalize this observation in Theorem~\ref{th:s0-invariant-under-some-automorphisms} below. We then apply this theorem in Example~\ref{ex:unitaries-on-S0}  to verify rather easily that a variety of operators on $L^{2}$, such as the partial Fourier transform and automorphisms of $G$, give rise to Banach space isomorphisms of the Feichtinger algebra. 
Theorem~\ref{th:s0-invariant-under-some-automorphisms} is the (seemingly unnoticed) \emph{automorphism theorem} of $\SO(G)$ by Reiter \cite[\S 4.5]{re89}, which is a generalization of a similar result by Weil \cite{we64} on the Schwartz-Bruhat space $\mathcal{S}(G)$. 

In order to state the result we need the following fact, which easily follows from the uniqueness of the Haar measure up to a positive multiplicative constant. Let $G_{1}$ and $G_{2}$ be two locally compact abelian groups with fixed Haar measure. For any topological group isomorphism $\alpha:G_{1}\to G_{2}$, there exists a unique positive constant, denoted by $\vert \alpha\vert$, such that
\begin{equation} \label{eq:isomorphism-integral} \int_{G_{2}} f(x_{2} ) \, dx_{2} = \vert \alpha \vert \int_{G_{1}} f(\alpha (x_{1})) \, dx_{1}, \ \ \text{for all} \ \ f\in L^{1}(G_{2}). \end{equation}
The constant $\vert \alpha\vert$ is the \emph{modulus} of the isomorphism $\alpha$. One can show that $\vert \alpha^{-1} \vert = \vert \alpha \vert^{-1}$. 

\begin{theorem} \label{th:s0-invariant-under-some-automorphisms} Let $G_{1}$ and $G_{2}$ be locally compact abelian groups, and let $T$ be a linear and bounded bijection from $L^{2}(G_{1})$ onto $L^{2}(G_{2})$. If $T$ satisfies
\begin{equation} \label{eq:0611a} \pi(\chi_{2}) T = c(\chi_{2}) \, T \pi(\alpha(\chi_{2})) \ \ \text{for all } \chi_{2}\in G_{2}\times\ghat_{2}, \end{equation} for some topological group isomorphism $\alpha$ from $G_{2}\times\ghat_{2}$ onto $G_{1}\times\ghat_{1}$ and a mapping $c:G_{2}\times\ghat_{2} \to \mathbb{T}$, then the operator $T$ and its $L^{2}$-Hilbert space adjoint operator $T^{*}$ are Banach space isomorphisms between $\SO(G_{1})$ and $\SO(G_{2})$. In fact, for all $f_{1},g_{1}\in \SO(G_{1})$ and $f_{2},g_{2}\in \SO(G_{2})$, we have the equalities
\[ \begin{matrix*}[l]
& \Vert Tf_{1} \Vert_{\SO(G_{2}),(T^{*})^{-1}g_{1}} = \vert \alpha \vert^{-1} \, \Vert f_{1} \Vert_{\SO(G_{1}),g_{1}} \, , & \Vert T^{-1} f_{2} \Vert_{\SO(G_{1}),T^{*}g_{2}} = \vert \alpha \vert \, \Vert f_{2} \Vert_{\SO(G_{2}),g_{2}}, \\
& \Vert T^{*} f_{2} \Vert_{\SO(G_{1}),T^{-1}g_{2}} = \vert \alpha \vert \, \Vert f_{2} \Vert_{\SO(G_{2}),g_{2}} \, ,  & \Vert (T^{*})^{-1} f_{1} \Vert_{\SO(G_{2}),T g_{1}} = \vert \alpha \vert^{-1} \, \Vert f_{1} \Vert_{\SO(G_{1}),g_{1}}. 
\end{matrix*} \]
If, furthermore, $T$ is a unitary operator, then
\[ \Vert T f_{1} \Vert_{\SO(G_{2}),Tg_{1}} = \vert \alpha \vert^{-1} \Vert f_{1} \Vert_{\SO(G_{1}),g_{1}} \ \ \text{and} \ \ \Vert T^{*} f_{2} \Vert_{\SO(G_{1}),T^{*}g_{2}} = \vert \alpha \vert \Vert f_{2} \Vert_{\SO(G_{2}),g_{2}}. \]
\end{theorem}
\begin{proof} 
It is straightforward to recast \eqref{eq:0611a} as an equivalent statement for the (Hilbert space) adjoint operator $T^{*}$
\begin{equation} \label{eq:1012a} T^{*} \pi(\chi_{2})  = \tilde{c}(\chi_{2}) \pi(\alpha(\chi_{2})) T^{*} \ \ \text{for all } \chi_{2}\in G_{2}\times\ghat_{2}, \end{equation} 
where $\vert \tilde{c}(\chi_{2}) \vert = 1$ for all $\chi_{2} \in G_{2}\times \ghat_{2}$ and $\alpha$ is as in \eqref{eq:0611a}. 
From the relation in \eqref{eq:1012a}, we deduce the equalities
\begin{align*}
& \int_{G_{2}\times\ghat_{2}} \vert \langle Tf_{1}, \pi(\chi_{2}) (T^{*})^{-1} g_{1} \rangle \vert \, d\chi_{2} = \int_{G_{2}\times\ghat_{2}} \vert \langle f_{1}, T^{*} \pi(\chi_{2}) (T^{*})^{-1} g_{1} \rangle \vert \, d\chi_{2} \\
& = \int_{G_{2}\times\ghat_{2}} \vert \langle f_{1}, \tilde{c}(\chi_{2}) \pi(\alpha(\chi_{2})) g_{1} \rangle \vert \, d\chi_{2} = \vert \alpha \vert^{-1} \int_{G_{1}\times\ghat_{1}} \vert \langle f_{1}, \pi(\chi_{1})) g_{1} \rangle \vert \, d\chi_{1} < \infty.
\end{align*}
This shows that $\Vert Tf_{1} \Vert_{\SO(G_{2}),(T^{*})^{-1}g_{1}} = \vert \alpha \vert^{-1} \, \Vert f_{1} \Vert_{\SO(G_{1}),g_{1}}$. Via the characterization $\SO(G)=\mathscr{B}$ from Theorem~\ref{th:s0-ABC}, we conclude that $T$ maps $\SO(G_{1})$ into $\SO(G_{2})$. Similar calculations for the operators $T^{-1}$, $T^{*}$ and $(T^{*})^{-1}$ yield the desired equalities and $T$ and $T^{*}$ are thus Banach space isomorphisms. The furthermore part follows easily.\end{proof}

Note that Theorem~\ref{th:s0-invariant-under-some-automorphisms} states sufficient conditions for an operator to be a Banach space isomorphism of $\SO$. One example of a Banach space isomorphism from $\SO(G)$ onto itself which does not satisfy 
\eqref{eq:0611a} is the so-called Gabor frame operator, see 
Example~\ref{ex:unitaries-on-S0}(ix) below.

\begin{example} \label{ex:unitaries-on-S0}
(i). Consider the unitary \emph{time-frequency shift} operator
\[ \pi(x,\omega) : L^{2}(G) \to L^{2}(G), \pi(x,\omega) f(s) = \omega(s) f(s-x), \, s\in G,\ (x,\omega)\in G\times \ghat.\] It satisfies \eqref{eq:0611a} with $\alpha$ being the identity on $G\times\ghat$ and $\vert \alpha \vert = 1$. Hence $\pi(x,\omega)$ for any $(x,\omega)\in G\times \ghat$ is a Banach space isomorphism on $\SO(G)$.  \\
(ii). (\hspace{1sp}\cite{lo80,Feichtinger1981}) Let $\gamma$ be a topological group \emph{isomorphism} from $G_{2}$ onto $G_{1}$ and define the operator
\[ U_{\gamma}:L^{2}(G_{1})\to L^{2}(G_{2}),\ U_{\gamma}f(x_{2}) = 
\vert \gamma \vert^{1/2} 
f(\gamma(x_{2})), \ \ x_{2}\in G_{2}.\]
One easily shows that $U_{\gamma}$ is a unitary operator with $(U_{\gamma})^{-1} = U_{\gamma^{-1}}$.
Moreover, $U_{\gamma}$ satisfies \eqref{eq:0611a} with $\alpha:G_{2}\times\ghat_{2}\to G_{1}\times\ghat_{1}, (x,\omega)\mapsto (\gamma(x),\omega\circ \gamma^{-1})$ and $\vert \alpha \vert = 1$. This shows that any isomorphism from $G_{2}$ onto $G_{1}$ induces a Banach space isomorphism from $\SO(G_{1})$ onto $\SO(G_{2})$ via $U_{\gamma}$. In particular, $\SO(G)$ is invariant under group automorphisms. As two special examples of such operators we mention (a) the \emph{asymmetric coordinate transform} and (b) the \emph{dilation operator}. Concerning (a), the asymmetric coordinate transform $\tau_{a}f(x,s)= f(s,s-x)$, $s,x\in G$, is induced by the automorphism $\gamma:(x,s)\mapsto(s,s-x)$ on $G\times G$. It therefore is a Banach space isomorphism from $\SO(G\times G)$ onto itself. We used $\tau_{a}$ in Section~\ref{sec:prelim} to define the short-time Fourier transform. We will use this knowledge in Theorem \ref{th:stft-and-s0} to show that the short-time Fourier transform is a mapping from $\SO(G)$ into $\SO(G\times\ghat)$.\\
(b). For $G=\R^d$ the unitary \emph{dilation operator} 
\[ D_{C}: L^{2}(\R^d) \to L^{2}(\R^d), \ D_{C}f(x) = \sqrt{\vert \det \, C \vert} \, f(Cx), \ \ C\in \text{GL}_{\R}(d), \ x\in \R^d\]
is induced by the automorphism $\gamma: x\mapsto C x$.
The dilation operator satisfies \eqref{eq:0611a} with $\alpha: \R^{d} \times \widehat{\R}^{d} \to \R^{d} \times \widehat{\R}^{d}, \ \alpha(x,\omega) = (C x,(C^{\top})^{-1} \omega)$ and $\vert \alpha \vert = 1$. Therefore the dilation operator $D_{C}$ is a Banach space isomorphism from $\SO(\R^d)$ onto itself. \\
(iii). (\hspace{1sp}\cite{Feichtinger1981}) The \emph{Fourier transform} $\mathcal{F}$ is a unitary operator from $L^{2}(G)$ onto $L^{2}(\ghat)$ and satisfies \eqref{eq:0611a} with $\alpha: \ghat\times G \to G \times \ghat,\ \alpha(\omega,x) = (-x,\omega)$. One can check that $\vert \alpha \vert = 1$.
Therefore $\mathcal{F}$ is a Banach space isomorphism from $\SO(G)$ onto $\SO(\ghat)$. \\
(iv). (\hspace{1sp}\cite{Feichtinger1981}) The \emph{partial Fourier transform} $\mathcal{F}_{2}$ is a unitary operator from $L^{2}(G_{1} \times G_{2})$ onto $L^{2}(G_{1}\times \ghat_{2})$. On functions in $L^{1}(G\times G)$, it is given by
\[ \textstyle\mathcal{F}_{2}f(x,\omega) = \int_{G} f(x,t) \, \overline{\omega(t)} \, dt, \ (x,\omega)\in G\times\ghat.\]
The partial Fourier transform satisfies \eqref{eq:0611a} with 
\[ \alpha : G_{1}\times \ghat_{2} \times \ghat_{1} \times G_{2} \to G_{1} \times G_{2} \times \ghat_{1} \times \ghat_{2}, \ (\lambda,\gamma, \xi,t)\mapsto (\lambda,-t,\xi,\gamma) \] 
and $\vert \alpha \vert = 1$. Therefore, the partial Fourier transform $\mathcal{F}_{2}$ is a Banach space isomorphism from $\SO(G_{1}\times G_{2})$ onto $\SO(G_{1}\times \ghat_{2})$. A similar statement holds for the partial Fourier transform $\mathcal{F}_{1}$. \\
%This result also appears in \cite{Feichtinger1981}. As with $\tau_{a}$ in Example~\ref{ex:unitaries-on-S0}(ii) above, also $\mathcal{F}_{2}$ is used in the definition of the short-time Fourier transform, and we will use this in Theorem~\ref{th:stft-and-s0} below.
(v). (\hspace{1sp}\cite{feko98}) For functions in $L^{1}(G\times \ghat)$ we define the \emph{symplectic Fourier transform} as
\[ \mathcal{F}_{s}f(x,\omega) = \int_{G} \int_{\ghat} f(t,\xi) \overline{\omega(t)} \xi(x) \, d\xi \, dt, \ \ (x,\omega) \in G\times\ghat. \]
As with the ordinary Fourier transform, also the symplectic Fourier transform can be extended by continuity from the intersection of $L^{1}$ and $L^{2}$ to a unitary operator on $L^{2}(G\times \ghat)$. The symplectic Fourier transform satisfies \eqref{eq:0611a} with
\[ \alpha: G\times \ghat \times \ghat\times G \to G\times \ghat \times \ghat\times G, \ (\lambda,\gamma,\xi,t) \mapsto (-t,\xi,\gamma,-\lambda) \]
and $\vert \alpha \vert = 1$. Thus $\mathcal{F}_{s}$ is a Banach space isomorphism from $\SO(G\times\ghat)$ onto $\SO(G\times \ghat)$. \\
(vi). (\hspace{1sp}\cite{Feichtinger1981}) Let $\psi$ be a \emph{second degree character} on $G$, \ie a continuous function $\psi: G\to \mathbb{T}$ such that
\[ \psi(x+y) = \psi(x) \cdot \psi(y) \cdot B(x,y) \ \ \text{for all} \ \ x,y\in G,\]
where $B(x,y) = (\rho(y))(x)$ for some topological group homomorphism $\rho:G\to\ghat$. $B$ is a so-called  bicharacter on $G$ \cite{re78}. We then define multiplication by a second degree character  
\[ U_{\psi}: L^{2}(G)\to L^{2}(G), \, U_{\psi}f(x) = \psi(x) f(x).\]
With the above definitions one can show that
\[ \pi(x,\omega) U_{\psi} = \psi(-x) U_{\psi} \pi(x,\omega-\rho(x)) \ \ \text{for all} \ \ (x,\omega)\in G\times\ghat.\] 
This shows that the operator $U_{\psi}$ satisfies \eqref{eq:0611a} with $\alpha: G\times \ghat \to G\times\ghat$ given by $\alpha(x,\omega) = (x,\omega-\rho(x))$, and one finds that $\vert \alpha\vert = 1$. We conclude that multiplication by a second degree character is a Banach space isomorphisms from $\SO(G)$ onto itself. In particular, for $G=\R^d$ this shows that $\SO(\R^d)$ is invariant under chirp-multiplication
\[ U_{M}: L^{2}(\R^d) \to L^{2}(\R^d), \ (U_{M}f)(x) = e^{\pi i \langle x, M x \rangle} f(x),\]
where $M$ is a real-valued $d\times d$-matrix. \\
(vii). The examples in (ii), (iii) and (vi) show that the dilation operator, the Fourier transform and the chirp-multiplication operator are Banach space isomorphisms on $\SO(\R^d)$. Since any \emph{metaplectic operator} on $L^{2}(\R^d)$ is a finite composition of these operators any metaplectic operator is a Banach space isomorphism from $\SO(\R^d)$ onto itself. In particular the fractional Fourier transform is a Banach space isomorphisms on $\SO(\R)$. Symplectic matrices and metaplectic operators are important in time-frequency analysis and physics. For more on this we refer to \cite{fo89,go11}. Concerning the theory of metaplectic operators on locally compact abelian groups, we refer the interested reader to \cite{we64,ca64-1,re78,re89}.\\
(viii). As a toy example the reader may verify that also the \emph{complex conjugation, reflection, and involution} satisfy \eqref{eq:0611a}. Of course, since those operators are self-inverse and \eqref{eq:2711a} shows that they are linear and bounded on $\SO(G)$, it already follows that they are isomorphisms of the Feichtinger algebra.   \\
(ix). In this final example we consider the Gabor frame operator. Given a function $g\in \SO(G)$ and a closed subgroup $\Delta\subseteq G\times\ghat$ we say that the \emph{Gabor system} $\mathcal{G}(g,\Delta)=\{ \pi(\nu) g \}_{\nu\in\Delta}$ is a \emph{Gabor frame} for $L^{2}(G)$ if there exists constant $A,B>0$ such that
\[ A \, \Vert f \Vert_{2}^{2} \le \int_{\Delta} \vert \langle f, \pi(\nu) g \rangle \vert^{2} \, d\mu_{\Delta}(\nu) \le B \, \Vert f \Vert_{2}^{2} \ \ \text{for all} \ \ f\in L^{2}(G).\]
If $\mathcal{G}(g,\Delta)$ is a Gabor frame for $L^{2}(G)$, then the \emph{Gabor frame operator} $S_{g,\Delta}$, given weakly by 
\[ \langle S_{g,\Delta}f_{1},f_{2}\rangle = \int_{\Delta} \langle f_{1}, \pi(\nu) g \rangle \langle \pi(\nu) g , f_{2}\rangle \, d\mu_{\Delta} (\nu) \ \ \text{for all} \ \ f_{1},f_{2}\in L^{2}(G),\]
is a linear, bounded, self-adjoint, and invertible operator on $L^{2}(G)$.  
The Gabor frame operator $S_{g,\Delta}$ does not satisfy \eqref{eq:0611a}. However, it does satisfy the following, similar, relationships with time-frequency shifts:
\begin{enumerate}[(a)]
\item $\pi(\chi) S_{g,\Delta} = S_{\pi(\chi)g,\Delta} \pi(\chi)$ for all $\chi\in G\times\ghat$,
\item $\pi(\chi) S_{g,\Delta} = S_{g,\Delta} \pi(\chi)$ for all $\chi\in \Delta$.
\end{enumerate}
One can show that $S_{g,\Delta}$ maps $\SO(G)$ into $\SO(G)$. Surprisingly, it is in fact true (and much more difficult to prove) that $S_{g,\Delta}$ also maps \emph{onto} $\SO(G)$. This is a remarkable and celebrated result by Gr\"ochenig and Leinert \cite{grle04}. For more on this result we refer the interested reader to \cite{fegr97,ba06,bacahela06,bacahela06-1,gr07-2} and Follands review paper \cite{fo06-1}. 
For more on frame theory and aspects of Gabor analysis we refer to the books \cite{ch16newbook,MR1843717,he11}. 

\end{example}

\subsection{The short-time Fourier transform on $\SO$}

As mentioned in Section \ref{sec:prelim}, for $h\in L^{2}(G)$ the short-time Fourier transform 
\[ \mathcal{V}_{h} : f \mapsto \mathcal{F}_{2}\tau(f\otimes\overline{h}), \ \ \mathcal{F}_{2}\tau_{a}(f\otimes\overline{h})(x,\omega) = \langle f,E_{\omega}T_{x}h\rangle, \ (x,\omega)\in G\times\ghat, \]
is a linear and bounded operator from $L^{2}(G)$ into $L^{2}(G\times\ghat)$.
By the characterization of $\SO(G)$ via the set $\mathscr{H}$ it follows that, if $h\in \SO(G)$, then $\mathcal{V}_{h}$ is a linear and bounded operator from $\SO(G)$ into $L^{1}(G\times\ghat)$. In this section we show that $\mathcal{V}_{h}$, $h\in \SO(G)$, is, in fact, a linear and bounded operator from $\SO(G)$ into $\SO(G\times\ghat)$.

\begin{theorem} \label{th:stft-and-s0} Let $G,G_{1}$ and $G_{2}$ be locally compact abelian groups.
\begin{enumerate}[(i)] 
\item The tensor product
\[ \otimes : \SO(G_{1})\times \SO(G_{2})\to \SO(G_{1}\times G_{2}), \ ( f_{1}\otimes f_{2}) (x_{1},x_{2}) = f_{1}(x_{1})\cdot f_{2}(x_{2})\]
is a bilinear and bounded operator. Specifically, for all $f_{i}\in \SO(G_{i})$ and for $g_{i}\in \SO(G_{i})\backslash\{0\}$, $i=1,2$, 
\begin{equation} \label{eq:0403} \Vert f_{1} \otimes f_{2} \Vert_{\SO(G_{1}\times G_{2}),g_{1}\otimes g_{2}} = \Vert f_{1} \Vert_{\SO(G_{1}),g_{1}} \, \Vert f_{2} \Vert_{\SO(G_{2}),g_{2}}.\end{equation}
\item For any $h\in \SO(G)$ the short-time Fourier transform
\[ \mathcal{V}_{h} : \SO(G) \to \SO(G\times \ghat), \ \mathcal{V}_{h}f(x,\omega) = \mathcal{F}_{2} \tau_{a} (f\otimes \overline{h})(x,\omega) = \langle f,E_{\omega}T_{x} h\rangle \]
is a linear and bounded operator. Specifically, for all $f,h\in \SO(G)$ and $g_{1},g_{2}\in \SO(G)\backslash\{0\}$, 
\[ \Vert \mathcal{V}_{h}f \Vert_{\SO(G\times\ghat), \mathcal{V}_{g_{2}}g_{1}} = \Vert f \Vert_{\SO(G),g_{1}} \, \Vert h \Vert_{\SO(G),g_{2}}.\]
\item[(iii)] For any non-zero $h\in \SO(G)$ the $L^{2}$-Hilbert space adjoint operator of $\mathcal{V}_{h}$, $\mathcal{V}_{h}^{*}:L^{2}(G\times\ghat)\to L^{2}(G)$ extends from $L^{1}(G\times\ghat)\cap L^{2}(G\times\ghat)$ to a linear and bounded operator from $L^{1}(G\times\ghat)$ onto $\SO(G)$. In particular, 
\[ \Vert \mathcal{V}_{h}^{*}F \Vert_{\SO(G),g} \le \Vert h \Vert_{\SO(G),g} \, \Vert F \Vert_{1} \ \ \text{for all} \ \ F\in L^{1}(G\times\ghat).\]
\end{enumerate} 
\end{theorem}
\begin{proof}
(i). Note that $\langle f_{1}\otimes f_{2}, E_{(\omega_{1},\omega_{2})} T_{(x_{1},x_{2})} (g_{1} \otimes g_{2})\rangle = \langle f_{1}, E_{\omega_{1}}T_{x_{1}} g_{1}\rangle \, \langle f_{2}, E_{\omega_{2}}T_{x_{2}} g_{2}\rangle$.
By use of this equality we easily show that
\begin{align*}
& \quad \ \Vert f_{1}\otimes f_{2}\Vert_{\SO(G_{1}\times G_{2}),g_{1} \otimes g_{2}} \\
& = \int_{G_{1}\times G_{2}\times \ghat_{1}\times \ghat_{2}} \vert \langle f_{1} \otimes f_{2}, E_{(\omega_{1},\omega_{2})} T_{(x_{1},x_{2})} (g_{1}\otimes g_{2} )\rangle \vert \, d(x_{1},x_{2},\omega_{1},\omega_{2}) \\
& = \bigg( \int_{G_{1}\times \ghat_{1}} \vert \langle f_{1}, E_{\omega_{1}} T_{x_{1}} g_{1} \rangle \vert \, d(x_{1},\omega_{1}) \bigg) \, \bigg( \int_{G_{2}\times \ghat_{2}} \vert \langle f_{2}, E_{\omega_{2}} T_{x_{2}} g_{2} \rangle \vert \, d(x_{2},\omega_{2}) \bigg) \\
& = \Vert f_{1} \Vert_{\SO(G_{1}),g_{1}} \, \Vert f_{2} \Vert_{\SO(G_{2}),g_{2}}.
\end{align*}
(ii). In Example~\ref{ex:unitaries-on-S0} we showed that the partial Fourier transform $\mathcal{F}_{2}$ and the asymmetric coordinate transform $\tau_{a}$ are Banach space isomorphisms of the Feichtinger algebra. 
We just proved that the tensor product operator maps $\SO(G)$ into $\SO(G\times G)$. 
Therefore $\mathcal{V}_{h}:f \mapsto \mathcal{F}_{2}\tau_{a} (f\otimes \overline{h})$ maps $\SO(G)$ into $\SO(G\times \ghat)$. By virtue of Theorem~\ref{th:s0-invariant-under-some-automorphisms} and 
Example~\ref{ex:unitaries-on-S0} we establish the equalities
\begin{align*}
& \Vert \mathcal{V}_{h} f \Vert_{\SO(G\times \ghat),\mathcal{V}_{g_{2}}g_{1}} = \Vert \mathcal{F}_{2} \tau_{a} (f\otimes \overline{h})  \Vert_{\SO(G\times \ghat),\mathcal{F}_{2} \tau_{a} g_{1}\otimes \overline{g_{2}}} = \Vert \tau_{a} (f\otimes \overline{h}) \Vert_{\SO(G\times G), \tau_{a} g_{1}\otimes \overline{g_{2}}} \\
& = \Vert f\otimes \overline{h} \Vert_{\SO(G\times G), g_{1}\otimes \overline{g_{2}}} = \Vert f \Vert_{\SO(G), g_{1}} \,  \Vert \overline{h} \Vert_{\SO(G), \overline{g_{2}}} = \Vert f \Vert_{\SO(G), g_{1}} \,  \Vert h \Vert_{\SO(G), g_{2}}.
\end{align*}
(iii).  Let $F$ be a function in $L^{1}(G\times\ghat)\cap L^{2}(G\times\ghat)$ and fix some non-zero function $g\in \SO(G)$. Using the definition of the $L^{2}$-Hilbert space adjoint operator $\mathcal{V}_{h}^{*}$ we establish 
\begin{align*}
& \int_{G\times\ghat} \vert \langle \mathcal{V}_{h}^{*}F, \pi(\chi)g\rangle \vert \, d\chi = \int_{G\times\ghat} \vert \langle F, \mathcal{V}_{h}\pi(\chi)g\rangle \vert \, d\chi \\ 
& = \int_{G\times\ghat} \Big\vert \int_{G\times\ghat} F(\widetilde{\chi}) \, \langle \pi(\widetilde{\chi})h, \pi(\chi) g\rangle \, d\widetilde{\chi} \, \Big\vert \, d\chi \le \int_{G\times\ghat} \, \vert F(\widetilde{\chi}) \vert \int_{G\times\ghat} \vert \langle \pi(\widetilde{\chi})h, \pi(\chi) g\rangle \vert \, d\chi \, d\widetilde{\chi} \\
& = \int_{G\times\ghat} \, \vert F(\widetilde{\chi}) \vert \, \Vert \pi(\widetilde{\chi}) h \Vert_{\SO,g} \, d\widetilde{\chi} = \Vert F \Vert_{1} \, \Vert h \Vert_{\SO,g} < \infty.
\end{align*}
This shows that $\mathcal{V}_{h}^{*}F\in \SO(G)$ and that $\Vert \mathcal{V}_{h}^{*}F \Vert_{\SO,g} \le \Vert h \Vert_{\SO,g} \, \Vert F \Vert_{1}$. Since $L^{1}(G\times\ghat)\cap L^{2}(G\times\ghat)$ is dense in $L^{1}(G\times\ghat)$ the operator $\mathcal{V}_{h}^{*}$ extends to a linear and bounded operator from all of $L^{1}(G\times\ghat)$ into $\SO(G)$. That the operator is onto can be shown in the following way. Let $f$ be any function in $\SO(G)$ and take $g\in \SO(G)$ such that $\langle h,g\rangle = 1$. It follows from \eqref{eq:STFT-dual} that $\mathcal{V}_{h}^{*}\mathcal{V}_{g} f = f$.  
\end{proof}

\begin{remark} \label{rem:adjoint-Vgf-as-integral}
We will show later, in Lemma \ref{le:STFT-on-SOprime}, that $\mathcal{V}_{h}^{*} F $, $F\in L^{1}(G\times\ghat)$, $h\in \SO(G)$, is the unique element in $\SO(G)$ such that
\[ (\mathcal{V}_{h}^{*} F,\sigma)_{\SO,\SOprime} = \int_{G\times\ghat} F(\chi) \, (\pi(\chi)h,\sigma)_{\SO,\SOprime} \, d\chi \ \ \text{for all} \ \ \sigma\in \SOprime(G).\]
In that sense, we write $\mathcal{V}_{h}^{*} F = \int_{G\times\ghat} F(\chi) \, \pi(\chi)h \, d\chi$.
\end{remark}

The norm-equality in Theorem~\ref{th:stft-and-s0}(ii) immediately yields the following result.

\begin{corollary} \label{cor:SO-char-Vgf-in-SO} Let $g\in \SO(G)\backslash\{0\}$ be given. Then
\begin{enumerate}[(i)]
\item the mapping $f\mapsto \Vert \mathcal{V}_{g}f \Vert_{\SO(G\times\ghat), \mathcal{V}_{g}g}$ defines a norm on $\SO(G)$ which is equivalent to $\Vert \cdot \Vert_{\SO(G),g}$,
\item $f\in L^{2}(G)$ belongs to $\SO(G)$ if and only if $\mathcal{V}_{g}f \in \SO(G\times\ghat)$,
\item $f\in L^{2}(G)$ belongs to $\SO(G)$ if and only if $\mathcal{V}_{f}f \in \SO(G\times\ghat)$.
\end{enumerate}

\end{corollary}
\begin{remark} Let $f,g\in L^{2}(G)$. The characterization of $\SO(G)$ via the set $\mathscr{B}$ and $\mathscr{H}$ together with Corollary~\ref{cor:SO-char-Vgf-in-SO} implies that $\mathcal{V}_{g}f \in L^{1}(G\times\ghat)$ if and only if $\mathcal{V}_{g}f \in \SO(G\times\ghat)$. \end{remark}

\subsection{The restriction, periodization and extension operator on $\SO$}
\label{sec:restriction} 
 
In order to formulate and prove the results of this section we need to recall some theory which can be found in, e.g., \cite{fo95,MR1802924}. If $H$ is a closed subgroup of a locally compact abelian group $G$, then $H$ with the subspace topology, as well as the quotient group $G/H$ with the quotient topology are also locally compact abelian groups. Given some Haar measure $\mu_{G}$ on $G$ and $\mu_{H}$ on $H$, then there is a unique Haar measure $\mu_{G/H}$ on $G/H$ such that
\begin{equation}
\label{eq:0202a}
 \int_{G} f(x) \, d\mu_{G}(x) = \int_{G/H} \, \int_{H} f(x+h) \, d\mu_{H}(h) \, d\mu_{G/H}(\dot{x})
\end{equation}
for all $f\in L^{1}(G)$ and where $\dot{x}=x+H\in G/H$, $x\in G$. The relation in \eqref{eq:0202a} is called \emph{Weil's formula} and if \eqref{eq:0202a} holds, we say that the measures $\mu_{G}$, $\mu_{H}$ and $\mu_{G/H}$ are \emph{canonically related}. We shall always assume that the measures on $G,H$ and $G/H$ are related in this way. For a closed subgroup $H$ in $G$ we define its \emph{annihilator} by
\[ H^{\perp} = \{ \gamma\in \ghat \, : \, \gamma(h) = 1 \ \text{for all} \ h\in H \}.\]
The annihilator $H^{\perp}$ is itself a closed subgroup of $\ghat$. Moreover, we have the isomorphisms $H^{\perp} \cong \widehat{G/H}$ and $\widehat{H} \cong \ghat/H^{\perp}$. 
If the measures $\mu_{G}$, $\mu_{H}$ and $\mu_{G/H}$ are canonically related and we require the Fourier inversion between functions on $G$ and $\ghat$, $H$ and $\ghat/H^{\perp}$, and between functions on $G/H$ and $H^{\perp}$ to hold, then the Haar measures $\mu_{\ghat}$, $\mu_{H^{\perp}}$ and $\mu_{\ghat/H^{\perp}}$ are uniquely determined and they are also canonically related. The in this way determined measure $\mu_{H^{\perp}}$ is called the \emph{orthogonal measure} of $\mu_{H}$.

For a closed subgroup $H$ in $G$, it is well-known that the periodization operator 
\[ P_{H} f(\dot{x}) = \int_{H} f(x+h) \, d\mu_{H}(h), \ \dot{x}=x+H, \, x\in G,\]
is a well-defined, linear and bounded operator from $L^{1}(G)$ onto $L^{1}(G/H)$. 
 Moreover, $P_{H}$ has the following properties: for all $f,g\in L^{1}(G)$
\begin{align} 
\Vert P_{H} f\Vert_{L^{1}(G/H)} \le \Vert f\Vert_{L^{1}(G)}, \ \ (P_{H}(f))^{\dagger} = P_{H}(f^{\dagger}), \ \ P_{H}f *^{G/H} P_{H}g = P_{H}(f*^{G} g),\label{eq:0102a}
\end{align}
where $*^{G}$ and $*^{G/H}$ denote convolution over $G$ and $G/H$, respectively, and $f^{\dagger}(x) = \overline{f(-x)}$. 

%See \cite[p.87-88, Theorem~3.4.6, Definition~5.5.1]{MR1802924}.

%In particular, the periodization operator $P_{H}$ is a morphism of the Banach algebra $L^{1}(G)$ onto the Banach algebra $L^{1}(G/H)$ under their respective convolutions.

We are now ready to show that the periodization and restriction of a function in the Feichtinger algebra with respect to a closed subgroup $H$ is again a function in $\SO(G/H)$ and $\SO(H)$, respectively. 
These results were already mentioned in the early papers \cite{fe79-5}, \cite{fe80} and then proven in \cite{Feichtinger1981}.
%In the proof we need to refer to some technical calculations from \cite{MR1802924} and results on Segal algebras from the same reference.
\begin{theorem} \label{th:periodization-restricion-map-in-SO}
Let $H$ be closed subgroup of $G$. 
\begin{enumerate}[(i)]
\item The periodization operator 
\[ P_{H}: \SO(G) \to \SO(G/H), \ P_{H}f(\dot{x}) = \int_{H} f(x+h) \, d\mu_{H}(h), \ \dot{x}=x+H, \ x\in G,\] 
is a linear and bounded operator from $\SO(G)$ onto $\SO(G/H)$. 
\item The restriction operator 
\[ R_{H}: \SO(G) \to \SO(H), \ R_{H}f(x) = f(x), \   x\in H,\]
is a linear and bounded operator from $\SO(G)$ onto $\SO(H)$.
\item For all functions $f\in \SO(G)$ the Poisson formula holds, \ie
\[ \int_{H} f(h) \, d\mu_{H}(h) = \int_{H^{\perp}} \hat{f}(\gamma) \, d\mu_{H^{\perp}}(\gamma).\]
\end{enumerate}
\end{theorem}
\begin{proof}
(i). So far we have not emphasized any specific choice of function $g$ with respect to which we define the $\SO$-norm. However, for this proof it is important to choose a compactly supported function in $\SO(G)$, the reason for this will become clear in a moment. Let now $f$ be any function in $\SO(G)$. Since $\SO(G)\subseteq L^{1}(G)$, we have that $P_{H}f,P_{H}g\in L^{1}(G/H)$. If we can show that
\[ \Vert P_{H} f \Vert_{\SO(G/H), P_{H}g} = \int_{\widehat{G/H}} \Vert E_{\gamma} P_{H}f *^{G/H} (P_{H}g)^{\dagger} \Vert_{L^{1}(G/H)} \, d\mu_{\widehat{G/H}}(\gamma) < \infty,\] then Theorem~\ref{th:s0-ABC} implies that $P_{H}f, P_{H}g\in \SO(G/H)$. Note that $E_{\gamma} P_{H} = P_{H} E_{\gamma}$ for all $\gamma\in H^{\perp} \cong \widehat{G/H}$. Using this, together with \eqref{eq:0102a}, yields the inequality
\begin{align}
& \phantom{=} \ \int_{\widehat{G/H}} \Vert E_{\gamma} P_{H}f *^{G/H} (P_{H}g)^{\dagger} \Vert_{L^{1}(G/H)} \, d\mu_{\widehat{G/H}}(\gamma) \nonumber \\ & \le \int_{H^{\perp}} \Vert E_{\gamma}f*^{G} g^{\dagger} \Vert_{L^{1}(G)} \, d\mu_{H^{\perp}}(\gamma) \nonumber \\
& =  \int_{G} \int_{H^{\perp}} \vert (E_{\gamma}f*^{G} g^{\dagger})(x) \vert \, d\mu_{H^{\perp}}(\gamma) \, d\mu_{G}(x). \label{eq:period-1}
\end{align}
Consider now the function $\varphi_{x,f}(s)= (g\cdot T_{-x} \overline{f})(s)$, $s,x\in G$. The reader may verify that $\mathcal{F}\varphi_{x,f}(\omega) = \omega(x) \overline{(E_{\omega}f*^{G}g^{\dagger})(x)}$. Note that the support of $\varphi_{x,f}$ is always a subset of the compact set $\supp \, g$.
Take now a function $h\in C_{c}(G)$ such that $h(s)=1$ for  all $s\in \supp \, g$. Since $h$ is equal to $1$ on the support of $g$ we easily get the identities $\varphi_{x,f} = \varphi_{x,f} \cdot h$ and $\mathcal{F} \varphi_{x,f} = \mathcal{F} \varphi_{x,f}  *^{\ghat} \mathcal{F}h$. We can thus make the following calculation:
\begin{align}
& \int_{H^{\perp}} \vert (E_{\gamma} f*^{G} g^{\dagger})(x) \vert \, d\mu_{H^{\perp}}(\gamma) = \int_{H^{\perp}} \vert \mathcal{F} \varphi_{x,f}(\gamma) \vert \, d\mu_{H^{\perp}}(\gamma) \nonumber \\
& = \int_{H^{\perp}} \vert \mathcal{F} \varphi_{x,f}*^{\ghat} \mathcal{F}h(\gamma) \vert \, d\mu_{H^{\perp}}(\gamma) \le \int_{\ghat} \vert \mathcal{F}\varphi_{x,f}(\omega) \vert \int_{H^{\perp}} \vert \mathcal{F}h(\gamma-\omega) \vert \, d\mu_{H^{\perp}} \, d\mu_{\ghat}(\omega) .\label{eq:period-2}
\end{align}
By virtue of \cite[Lemma 5.5.5]{MR1802924}, we can bound $\int_{H^{\perp}} \vert \mathcal{F}h(\gamma-\omega) \vert \, d\mu_{H^{\perp}}$ by a constant $C$, which is independent of $\omega$. Indeed, as detailed in the reference: take $h_{1}$ and $h_{2}$ to be functions in $C_{c}(G)$ such that $h_{1}*h_{2}$ is equal to $1$ on the compact set $\supp\,g$, then one can take $C= \Vert P_{H}(\vert h_{1}\vert)\Vert_{2} \, \Vert P_{H}(\vert h_{2}\vert)\Vert_{2}$. Continuing the calculation from above yields the desired estimate:
\begin{align*}
 \Vert P_{H}f \Vert_{\SO(G/H),P_{H}g}   & \stackrel{\eqref{eq:period-1}}{\le} \int_{G} \int_{H^{\perp}} \vert (E_{\gamma} f*^{G} g^{\dagger})(x) \vert \, d\mu_{H^{\perp}}(\gamma) \, d\mu_{G}(x) 
\\ & \stackrel{\eqref{eq:period-2}}{\le} C \int_{G} \int_{\ghat} \vert (E_{\omega} f*^{G} g^{\dagger})(x) \vert \, d\mu_{\ghat}(\omega) \, d\mu_{G}(x) \\
& \ \, = C \, \Vert f \Vert_{\SO(G),g}.
\end{align*}
This shows that $P_{H}$ is a bounded operator from $\SO(G)$ into $\SO(G/H)$. The proof that the periodization operator is surjective will have to wait until Section \ref{sec:apply-minimal}. \\
(ii). It is a straightforward calculation to verify that, for all $f\in L^{1}(G)$ with $\hat{f}\in L^{1}(\ghat)$,
\begin{equation} \label{eq:poisson-derivation} R_{H} f= \mathcal{F}_{H}^{-1} P_{H^{\perp}} \mathcal{F}_{G}f.\end{equation}
In particular this is true for all $f\in \SO(G)$. Here, $\mathcal{F}_{G}$ and $\mathcal{F}_{H}$ are the Fourier transform on $L^{1}(G)$ and $L^{1}(H)$, respectively.
It follows from the result in (i) that the periodization with respect to the closed subgroup $H^{\perp}$ is a linear and bounded operator from $\SO(\ghat)$ onto $\SO(\ghat/H^{\perp})$. By Example~\ref{ex:unitaries-on-S0} we know that the Fourier transform is a unitary Banach space isomorphism on $\SO$. We conclude that the restriction operator $R_{H} = \mathcal{F}_{H}^{-1} P_{H^{\perp}} \mathcal{F}_{G}$ is linear and bounded from $\SO(G)$ onto $\SO(H)$. \\
(iii). Since the Fourier transform is a Banach space isomorphism on $\SO$ and $R_{H}$ and $P_{H^{\perp}}$ are operators on $\SO$, it follows from \eqref{eq:poisson-derivation} that $\mathcal{F}_{H} R_{H}f = P_{H^{\perp}} \mathcal{F}_{G}f$ for all $f\in \SO(G)$. From this equality we find that, for all $f\in \SO(G)$,  
\[ \mathcal{F}_{H} R_{H} f(0) = P_{H^{\perp}} \mathcal{F}_{G}f(0), \ \ \text{i.e.,} \ \ \int_{H} f(h) \, d\mu_{H}(h) = \int_{H^{\perp}} \hat{f}(\gamma) \, d\mu_{H^{\perp}}(\gamma),\]
which is the the Poisson formula.
\end{proof}

% To do this we define the size of $H$, $s(H) = \mu_{G/H}(G/H)$. Note that if $H$ is a discrete subgroup with $G/H$ compact, and $H$ is equipped with the counting measure, then $s(H)$ is the size of the fundamental domain of $H$, \ie $s(H)$ coincides with the lattice size.
\begin{remark} Let $G_{1}$ and $G_{2}$ be two locally compact abelian groups. Theorem~\ref{th:periodization-restricion-map-in-SO}(i) applied to $\SO(G_{1}\times G_{2})$ with $H=\{0\}\times G_{2}$ shows that
\[ T: \SO(G_{1}\times G_{2}) \to \SO(G_{1}), \ T f(x_{1}) = \int_{G_{2}} f (x_{1},x_{2}) \, dx_{2}, \ \ x_{1}\in G_{1},\]
is a linear and bounded operator from $\SO(G_{1}\times G_{2})$ onto $\SO(G_{1})$.
\end{remark}

\begin{remark} If $H$ is a discrete subgroup of $G$, then, by Lemma \ref{le:s0-discrete-compact}, $\SO(H) = \ell^{1}(H)$. Hence, in this case,
\[ R_{H}:\SO(G)\to \ell^{1}(H), \ R_{H}f(x) = f(x), \ x\in H, \]
is a linear and bounded operator from $\SO(G)$ onto $\ell^{1}(H)$. As a concrete example, this shows that the restriction of any function in $\SO(\R^{n})$ to $M \Z^{n}$ for some $n\times n$-matix $M$ is a sequence in $\ell^{1}(M\Z^{n})$.
\end{remark}

\begin{remark} \label{rem:poisson} Assume that $H$ is a closed subgroup in $G$ such that the quotient group $G/H$ is compact (equivalently the closed subgroup $H^{\perp}$ is discrete). Then we define the size of $H$ by $s(H)= \mu_{G/H}(G/H)$. In this case the orthogonal measure on $\mu_{H^{\perp}}$ satisfies
\[ \int_{H^{\perp}} \hat{f}(\gamma) \, d\mu_{H^{\perp}}(\gamma) = \frac{1}{s(H)} \sum_{\gamma\in H^{\perp}} \hat{f}(\gamma)\] 
and the Poisson formula takes the form
\[ \int_{H} f(h) \, d\mu_{H}(h) = \frac{1}{s(H)} \sum_{\gamma\in H^{\perp}} \hat{f}(\gamma) \ \ \text{for all} \ f\in \SO(G).\]
If $G/H$ is compact, $H$ is discrete and $H$ is equipped with the counting measure, then we arrive at the familiar  Poisson \emph{summation} formula
\[ \sum\limits_{h\in H} f(h) = \frac{1}{s(H)} \sum\limits_{\gamma\in H^{\perp}} \hat{f}(\gamma) \ \ \text{for all} \ f\in \SO(G). \]
For example, let $G=\R^{n}$ be equipped with the usual Lebesgue measure and let $H=M\Z^{n}$ with $M\in \textnormal{GL}_{\R}(n)$ be equipped with the counting measure. Furthermore, let the Fourier transform on $\SO(\R^{n})$ be the one given by
\[ \hat{f} (\omega) = \int_{\R^{n}} f(x) \, e^{2\pi i x\cdot \omega} \, dx, \ \omega\in \R^{n}.\]
Then $s(H) = \vert \textnormal{det} \, M \vert$ and the Poisson formula takes the form 
\[ \sum_{k\in M\Z^{n}} f(k) = \frac{1}{\vert \textnormal{det} \, M \, \vert} \, \sum_{k\in (M^{-1})^{\top} \Z^{n}} \hat{f}(k),\]
which is valid for all $f\in \SO(\R^{n})$.
\end{remark}

That a function $f$ belongs to the Feichtinger algebra is merely a sufficient condition for the Poisson formula to hold. Indeed, Poisson's formula holds for functions outside of $\SO(G)$ (see for example \cite{MR1802924}), however, in general, one has to be very careful with the validity of the formula. The failure of the Poisson formula is demonstrated 
in \cite[Chap.\ VI, Sec.\ 1, Exercise 15]{MR2039503}. There, a continuous function $f\in L^{1}(
\R)$ is constructed so that $\hat{f}\in L^{1}(\R)$ and both sides of the Poisson formula are finite, yet the equality fails. Further references to the Poisson formula are \cite{kale94,bezi97,hola05}. See also the paper by Gr\"ochenig \cite{gr96} in which the relationship between $\SO(\R^n)$, functions in $L^{2}(\R^n)$ with decay conditions in time and frequency, and the Poisson formula is investigated. For remarks on the importance and usefulness of the Poisson formula see, e.g.,  \cite[Chap.\ 2, \S 9.9, \S 9.10]{hani98} and the two appendices in that section.

The results in this section have so far been concerned with \emph{closed} subgroups in $G$. We will now add the assumption that $H$ should also be an \emph{open} subgroup (indeed, if a subgroup of a topological group is open, then it is also closed). If $H$ is an open subgroup, then $G/H$ is discrete. Without loss of generality we will assume that the Haar measure on $G/H$ is the usual counting measure. Hence, given a Haar measure $\mu_{G}$ on $G$ and the counting measure on the discrete group $G/H$, there is a unique measure on the open subgroup $H$ such that the measures $\mu_{G}$, $\mu_{H}$ and $\mu_{G/H}$ are canonically related. In particular, for all $f\in L^{1}(G)$, Weil's formula takes the form
\[ \int_{G} f(x) \, d\mu_{G}(x) = \sum_{\dot{x}\in G/H} \int_{H} f(x+h) \, d\mu_{H}(h), \ \dot{x} = x + H, \ x\in G.\]
We note that for the Euclidean space the only open subgroup is the space itself. A far less mundane situation occurs for the group of the $p$-adic numbers, here all subgroups are open.  

The following result is from \cite{Feichtinger1981} and is also proven in \cite{re89}.

\begin{proposition} \label{pr:zero-ext} Let $H$ be an open subgroup of a locally compact abelian group $G$. Then the zero-extension operator
\[ Q_{H} : \SO(H) \to \SO(G), \ Q_{H}f(x) = \begin{cases} f(x) & x\in H, \\ 0 & x\notin H,\end{cases} \ \ x\in G,\]
is well-defined, linear and bounded. In particular, for any non-zero $g\in \SO(H)$, 
\[ \Vert Q_{H} f \Vert_{\SO(G), Q_{H}g} = \Vert f \Vert_{\SO(H),g} \ \ \text{for all} \ \ f\in \SO(H).\]
\end{proposition}
\begin{proof}
In order to show that the operator is well-defined, i.e., $Q_{H}f\in \SO(G)$, we let $f,g\in \SO(H)$, $g\ne 0$ and we will show that
\[ \int_{\ghat} \Vert E_{\omega} Q_{H}f *^{G} (Q_{H}g)^{\dagger} \Vert_{L^{1}(G)} \, d\mu_{\ghat}(\omega) < \infty.\]
We do this in three parts. First, let us consider the function $E_{\omega} Q_{H}f *^{G} (Q_{H}g)^{\dagger}$.
Using Weil's formula with respect to the open subgroup $H$ (as detailed above), we find, for all $x\in G$,
\begin{align*}
 E_{\omega} Q_{H}f *^{G} (Q_{H}g)^{\dagger}(x) & = \int_{G} \omega(s) Q_{H}f(s) (Q_{H}g)^{\dagger}(x-s) \, d\mu_{G}(s) \\
& = \sum_{\dot{s}\in G/H} \int_{H} \omega(s+h) Q_{H}f(s+h) (Q_{H}g)^{\dagger}(x-s-h) \, d\mu_{H}(h) \\
& = \int_{H} \omega(h) f(h) (Q_{H}g)^{\dagger}(x-h) \, d\mu_{H}(h),
\end{align*}
where the last equality is due to the fact that $Q_{H}f(s+h) = 0$ for $s\ne 0$.
Now, using the just shown equality and the same arguments once more we establish that
\begin{align*}
\Vert E_{\omega} Q_{H}f *^{G} (Q_{H}g)^{\dagger} \Vert_{L^{1}(G)} & = \int_{G} \vert Q_{H}f *^{G} (Q_{H}g)^{\dagger}(x) \vert \, d\mu_{G}(x) \\
& = \int_{G} \big\vert \int_{H} \omega(h) f(h) (Q_{H}g)^{\dagger}(x-h) \, d\mu_{H}(h) \, \big\vert \, d\mu_{G}(x) \\
& = \sum_{\dot{x}\in G/H} \int_{H} \big\vert \int_{H} \omega(h) f(h) (Q_{H}g)^{\dagger}(x+h'-h) \, d\mu_{H}(h) \, \big\vert \, d\mu_{H}(h') \\
& = \int_{H} \big\vert \int_{H} \omega(h) f(h) g^{\dagger}(h'-h) \, d\mu_{H}(h) \, \big\vert \, d\mu_{H}(h')
\end{align*}
Recall that the quotient group $G/H$ is discrete and equipped with the counting measure. Therefore its dual group $\widehat{G/H} \cong H^{\perp}$ is compact and has a normalized Haar measure. Using this together with the fact that $\ghat/H^{\perp}\cong \widehat{H}$ we establish the desired estimate. 
\begin{align*}
& \int_{\ghat} \Vert E_{\omega} Q_{H}f *^{G} (Q_{H}g)^{\dagger} \Vert_{L^{1}(G)} \, d\mu_{\ghat}(\omega) \\
& = \int_{\ghat} \int_{H} \big\vert \int_{H} \omega(h) f(h) g^{\dagger}(h'-h) \, d\mu_{H}(h) \, \big\vert \, d\mu_{H}(h') \, d\mu_{\ghat}(\omega) \\
& = \int_{\ghat/H^{\perp}} \int_{H^{\perp}} \int_{H} \big\vert \int_{H} \omega(h) \gamma(h) f(h) g^{\dagger}(h'-h) \, d\mu_{H}(h) \, \big\vert \, d\mu_{H}(h') \, d\mu_{H^{\perp}}(\gamma) \, d\mu_{\ghat/H^{\perp}}(\dot{\omega)}) \\
& = \int_{\ghat/H^{\perp}} \int_{H^{\perp}} \int_{H} \big\vert \int_{H} \omega(h) f(h) g^{\dagger}(h'-h) \, d\mu_{H}(h) \, \big\vert \, d\mu_{H}(h') \, d\mu_{H^{\perp}}(\gamma) \, d\mu_{\ghat/H^{\perp}}(\dot{\omega)}) \\
& = \int_{\ghat/H^{\perp}} \int_{H} \big\vert \int_{H} \omega(h) f(h) g^{\dagger}(h'-h) \, d\mu_{H}(h) \, \big\vert \, d\mu_{H}(h') \, d\mu_{\ghat/H^{\perp}}(\dot{\omega)}) \,  \int_{H^{\perp}} \, d\mu_{H^{\perp}}(\gamma) \\
& = \int_{\widehat{H}} \int_{H} \big\vert \int_{H} \omega(h) f(h) g^{\dagger}(h'-h) \, d\mu_{H}(h) \, \big\vert \, d\mu_{H}(h') \, d\mu_{\widehat{H}}(\omega) \\
& = \int_{\widehat{H}} \Vert E_{\omega} f*^{H} g^{\dagger} \Vert_{L^{1}(H)} \, d\mu_{\widehat{H}}(\omega) = \Vert f \Vert_{\SO(H),g} < \infty.
\end{align*}
By the characterization of $\SO(G)$ by the set $\mathscr{A}$ it follows that $Q_{H}f\in \SO(G)$. Hence $Q_{H}$ is well-defined. The linearity of $Q_{H}$ is clear. Finally, the above calculation together with \eqref{eq:s0-norm-with-modulation-space-and-wiener-amalgam} implies the desired norm equality. \end{proof}

\section{The space of translation bounded quasimeasures -- $\SOprime(G)$}
\label{sec:s0prime}
In this section we consider $\SOprime(G)$ -- the dual space of $\SO(G)$. In \cite{Feichtinger1981} elements in $\SOprime(G)$ are called \emph{translation bounded quasimeasures}. %\footnote{\emph{Translation bounded measures} is the name of elements in the dual group of the Wiener algebra $\textbf{W}(G)$ \cite{fe77-3} and elements in the dual group of $C_c(G)\cap A(G)$ are called \emph{quasimeasures}. The dense inclusions $C_{c}(G)\cap A(G) \subseteq \SO(G) \subseteq \textbf{W}(G)$ imply the converse inclusions for the dual spaces, which motivates the name translation bounded quasimeasures for elements in $\SOprime(G)$.} 
Since then they have, as we will do here, also been called \emph{distributions}. The dual of $\SO(G)$ is discussed in, e.g., \cite{fe80,ma87,fe89-2,ho89,MR1843717,go11}. Naturally, $\SOprime(G)$ forms a Banach space with respect to the operator norm 
\[ \Vert \sigma \Vert_{\SOprime,g} = \sup_{\substack{f\in \SO(G) \\ \Vert f \Vert_{\SO,g}=1}} \vert \sigma(f) \vert, \ \sigma \in \SOprime(G),\]
for some fixed $g\in \SO(G)\backslash\{0\}$.
The norm satisfies the inequality $\vert \sigma(f) \vert \le \Vert \sigma \Vert_{\SOprime,g} \, \Vert f \Vert_{\SO,g}$ for all $f\in \SO(G)$. Since all $g\in \SO(G)\backslash\{0\}$ induce equivalent norms on $\SO(G)$ (Proposition \ref{pr:equivalent-norm-on-s0}) it is straight forward to show that they also induce equivalent norms on $\SOprime(G)$.
Besides the norm topology on $\SOprime(G)$, the weak$^{*}$ topology also plays a crucial role. Recall that a net\footnote{The weak$^{*}$ topology of the dual of a Banach space is metrizable only if the Banach space is finite dimensional. If $\SO(G)$ is separable (e.g., if $G$ is $\sigma$-compact and metrizable) then the relative weak$^{*}$ topology on \emph{bounded} sets in $\SOprime(G)$ is metrizable and one can work with sequences. Otherwise nets have to be considered. For more on this, we refer to the book by Megginson \cite{me98}.} $(\sigma_{\lambda})$ in $\SOprime(G)$ converges to $\sigma\in\SOprime(G)$ in the weak$^{*}$-sense if
\[ \lim_{\lambda} \sigma_{\lambda}(f) = \sigma(f) \ \ \text{for all} \ \, f\in \SO(G).\]
We introduce the common notation for the action of functionals 
\[ ( f, \sigma )_{\SO,\SOprime(G)} = \sigma(f) \ \ \text{for all} \ \ f\in \SO(G), \ \sigma\in \SOprime(G).\]
For convenience we shall sometimes write $(f,\sigma)_{\SO,\SOprime}$ or $( f, \sigma)$ instead of $( f, \sigma)_{\SO,\SOprime(G)}$. Note that $(\, \cdot \, , \, \cdot \, )$ is bilinear. For $f,g\in \SO(G)$ and $\sigma\in \SOprime(G)$ we define the following operations:
\begin{align} & ( f , g\cdot \sigma ) = ( f \cdot g, \sigma ), \label{eq:multiplication-distribution}\\
& ( f, g*\sigma ) = ( f * g^{r} ,\sigma ), \ \ g^{r}(x) = g(-x), \label{eq:convolution-distribution}\\ 
& ( f, \overline{\sigma} ) = \overline{( \overline{f} , \sigma )}\label{eq:conjugation-distribution}.\end{align}
All functions $h\in \SO(G)$ induce a distribution $\iota(h)$ given by
\begin{equation} \label{eq:regular-distribution} (f,\iota(h)) = \int_{G} f(x) \, h(x) \, dx, \ \ f\in \SO(G).\end{equation}
Note that $(f,\iota(h))=(h,\iota(f))$. If a distribution $\sigma$ is induced by a function $h$ via \eqref{eq:regular-distribution}, then we may simply write $(f,h)_{\SO,\SOprime}$ rather than $(f,\iota(h))_{\SO,\SOprime}$. In fact, all functions $h$ for which the above integral makes sense induce an element in $\SOprime(G)$ in this way. 

\begin{lemma} \label{le:embeddings-in-SOprime} All functions in $\SO(G)$, $L^{p}(G)$ for $p\in [1,\infty]$, $C_{0}(G)$ and $A(G)$ define elements in $\SOprime(G)$ via \eqref{eq:regular-distribution}. In particular, for $g\in \SO(G)\backslash\{0\}$ the following holds:
\begin{enumerate}[(i)]
\item $\Vert \iota(h) \Vert_{\SOprime,g} \le c \, \Vert h \Vert_{\SO,g} \ $ for all $h\in \SO(G)$ with $c = \Vert g \Vert_{\infty}^{-1} \, \Vert g \Vert_{1}^{-1}$.
\item $\Vert \iota(h) \Vert_{\SOprime,g} \le c \, \Vert h \Vert_{p} \ $ for all $h\in L^{p}(G)$, $p\in [1,\infty]$ with $c = \Vert g \Vert_{1}^{-1/p} \, \Vert g \Vert_{\infty}^{-1+1/p}$.
\item $\Vert \iota(h) \Vert_{\SOprime,g} \le c \, \Vert h \Vert_{A(G)} \ $ for all $h\in A(G)$ with $c= \Vert g \Vert_{\infty}^{-1}$.
\end{enumerate}
If $g$ is chosen with the normalization as in Lemma~\ref{le:S0-contains-piecewise-constant-functions}(v), then $c=1$ in all cases.
\end{lemma}
\begin{proof} Proof of (iii). Let $h\in A(G)$ be given. Then there is a function $\tilde{h}\in L^{1}(\ghat)$ such that $h = \mathcal{F} \tilde{h}$, where $\mathcal{F}$ is the Fourier transform from $\ghat$ to $G$. Thus, by applying 
Lemma~\ref{le:cont-embedded-in-Lp} we find that
\begin{align*} \vert (f,\iota(h))  \vert & \le \textstyle\int_{G} \vert f(x) \vert \, \vert \mathcal{F}\tilde{h}(x) \vert \, dx \le \Vert f \Vert_{1} \Vert \mathcal{F}\tilde{h}\Vert_{\infty} \\
& \le \Vert g \Vert_{\infty}^{-1} \Vert \tilde{h} \Vert_{1} \Vert f \Vert_{\SO,g} = \Vert g \Vert_{\infty}^{-1} \Vert h \Vert_{A(G)} \Vert f \Vert_{\SO,g}.\end{align*}
This implies (iii). Statements (i) and (ii) follow by similar calculations and applications of the estimates from Lemma~\ref{le:cont-embedded-in-Lp}.
\end{proof}
%
%$ C'_{0} = M_ b$ Bounded Radon measures
%
%$ A(G)' = P(G)$ pseudomeasures. 
%
%$ A(G) \subseteq C_{0}$ hence $M_b \subset P(G)$
%
%$ A_{c}(G) \subseteq A(G)$ hence $Pseudomeas(G) \subset Quasimeas(G)$
%
%$A_{c}(G) \subset \SO\subset A(G)$ hence $P(G) \subseteq \SOprime \subseteq Q(G)$.
%
%

The following lemma is a general result for Banach spaces adapted to $\SO$. 

\begin{lemma} \label{le:Bprime-embeddedinto-Soprime} 
Let $B$ be a Banach space which contains $\SO(G)$. The following statements are equivalent.
\begin{enumerate} 
\item[(i)] $\SO(G)$ is continuously embedded into and dense in $B$. 
%I.e., $\iota: \SO(G) \to B$ is a well-defined, linear, bounded and injective operator and the closure of $\iota(\SO(G))$ is the $B$ norm is $B$.
\item[(ii)] $B'$ is weak$^{*}$-continuously embedded into and weak$^{*}$-dense in $\SOprime(G)$. 
\end{enumerate}
\end{lemma}
\begin{proof} From the corollary of \cite[Theorem 4.12]{ru91} it follows that there is a linear, bounded and injective operator from $\SO(G)$ into $B$ with dense range, if, and only if, there is a linear, bounded and injective operator from $B'$ into $\SOprime(G)$ with weak$^{*}$-dense range, which, furthermore, is continuous with respect to the weak$^{*}$ topologies on both spaces.\end{proof}

Together with Lemma \ref{le:cont-embedded-in-Lp} the statement of Lemma \ref{le:Bprime-embeddedinto-Soprime} implies that the bounded measure $M_{b}(G) = C_{0}'(G)$, the pseudo-measures $PM(G) = A'(G)$, and the $L^{p}$-spaces for $p\in (1,\infty]$ are weak$^{*}$-dense in and continuously embedded into $\SOprime(G)$ (concerning the identification of $(L^{1})'$ with $L^{\infty}$ see the footnote on page \pageref{page:l1-linfty}). In case of the $L^{p}$-spaces, this embedding is of course via the mapping $\iota$ as in \eqref{eq:regular-distribution}.

The following lemma shows that $\SOprime(G)$ contains pointwise evaluations, \ie the Dirac-delta distribution, $\delta_{x}(f) = f(x)$, $x\in G$ belongs to $\SOprime(G)$.

\begin{lemma} \label{le:delta-in-SOprime} For any $x\in G$ the functional $\delta_{x}$ belongs to $\SOprime(G)$. In particular, for $g\in \SO(G)\backslash\{0\}$ it holds that $\Vert \delta_{x} \Vert_{\SOprime,g} \le \Vert g \Vert_{1}^{-1}$.
\end{lemma}
\begin{proof}
Since all functions in $\SO(G)$ are continuous, the pointwise evaluation $( f, \delta_{x})_{\SO,\SOprime(G)} = f(x)$ is well-defined for all $x\in G$ and clearly linear.  By use of Lemma~\ref{le:cont-embedded-in-Lp}, we establish the inequalities
\[ \vert ( f, \delta_{x}) \vert = \vert f(x) \vert \le \Vert f \Vert_{\infty} \le \Vert g \Vert_{1}^{-1} \, \Vert f \Vert_{\SO,g}\] for all $f\in \SO(G), x\in G$ and $g\in \SO(G)\backslash\{0\}$. 
\end{proof}

We now show that any $\sigma\in \SOprime(G)$ can be approximated arbitrarily well in the weak$^{*}$-sense by the distributions induced by functions in $\SO(G)$ or by elements in the linear span of $\{\delta_{x}\}_{x\in G}$. 

\begin{lemma} \label{le:weak-star-dense-soprime} The following holds:
\begin{enumerate}[(i)]
\item $\iota(\SO(G))$ is weak$^{*}$-dense in $\SOprime(G)$.
\item $\textnormal{span}\{\delta_{x}\}_{x\in G}$ is weak$^{*}$-dense in $\SOprime(G)$.
\end{enumerate}
\end{lemma}
\begin{proof} Let $N=\iota(\SO(G))$ and consider $^{\perp}N = \{ f\in \SO(G)\, : \, ( f, \sigma)_{\SO,\SOprime} = 0 \ \ \text{for all} \ \ \sigma \in N\}$. Then $^{\perp}N= \{0\}$. It follows from \cite[Theorem. 4.7]{ru91} that the weak$^{*}$-closure of $\iota(\SO(G))$ is $\SOprime(G)$. If $N=\textnormal{span}\{\delta_{x}\}_{x\in G}$, then again $^{\perp}N=\{0\}$ and (ii) follows.
\end{proof}
%The result in \cite[Thrm. 4.7]{MR1157815} is based on the Hahn-Banach theorem and is applicable in our setting because $\SO(G)$ is a Banach space. 
In \cite{fezi98} a different proof of Lemma~\ref{le:weak-star-dense-soprime}(ii) is sketched: By using the fact that $(\SOprime(G) * \SO(G))\cdot \SO(G) \subseteq \SO(G)$ one can construct a sequence (or a net) of functions in $\SO(G)$ that converges towards any given $\sigma\in\SOprime(G)$. 
This approach is the same one used for the test functions $C_{c}^{\infty}(\R^n)$ and their dual space in, \eg \cite{MR1721032}.
We prove that $(\SOprime(G) * \SO(G))\cdot \SO(G) \subseteq \SO(G)$ and $(\SOprime(G) \cdot \SO(G))* \SO(G) \subseteq \SO(G)$ in 
Lemma~\ref{le:SOprime-time-SO-conv-S0-is-S0}.

\subsection{Operators on $\SOprime$} \label{sec:banach-space-adjoin}
In this section we will show how and when one can extend operators on $L^{2}$ to operators on $\SOprime$. We will tackle the same question for operators on $\SO$ in Lemma \ref{le:SOop-to-SOprime}. In order to formulate the results we need the notion of the \emph{Banach space adjoint operator}.
The following can be found in, \eg \cite{MR3289046,me98}. Let $T$ be a linear and bounded operator from a Banach space $X$ to a Banach space $Y$. Then its Banach space adjoint
\[ T^{\times} : Y'\to X', \ ( x, T^{\times} y')_{X,X'} = ( T x, y')_{Y,Y'}, \ \ x\in X, \ y'\in Y',\]
is a well-defined linear and bounded operator with $\Vert T^{\times} \Vert_{\textnormal{op}} = \Vert T \Vert_{\textnormal{op}}$. The adjoint operator $T^{\times}$ maps norm convergent sequences in $Y'$ into norm convergent sequences in $X'$. We say that $T^{\times}$ is \emph{norm-norm continuous} from $Y'$ into $X'$. Furthermore, the adjoint operator $T^{\times}$ maps weak$^{*}$-convergent nets in $Y'$ into weak$^{*}$-convergent nets in $X'$. We say that $T^{\times}$ is \emph{weak$^{*}$-weak$^{*}$ continuous} from $Y'$ into $X'$. In fact, every weak$^{*}$-weak$^{*}$ continuous linear operator from $Y'$ into $X'$ is also norm-norm continuous. In general the converse is not true, see \cite[Section~3.1]{me98}. Furthermore, if $S:Y'\to X'$ is a linear and weak$^{*}$-weak$^{*}$ (thus also norm-norm) continuous operator, then
\[ S^{\times}: X\to Y, \ (S^{\times}x,y')_{Y,Y'} = (x,Sy')_{X,X'}, \ \ x\in X, \ y'\in Y'\]
defines a linear and bouned operator which satisfies $(S^{\times})^{\times}=S$. Also, with $T$ as above, $(T^{\times})^{\times} = T$.

%These continuity properties of $T^{\times}$ are no coincidence. As stated in, e.g., 
%\cite[Proposition 5.42]{MR3289046}, the norm-norm, weak$^{*}$-weak$^{*}$, and the norm-weak$^{*}$ continuity of an operator between two Banach spaces are equivalent properties. WRONG!! only for weak topology!

In the previous section we established that $L^{2}(G)$ is embedded into $\SOprime(G)$ by the mapping $\iota$ from \eqref{eq:regular-distribution}. Thus, if $T:L^{2}(G_{1}) \to L^{2}(G_{2})$ is a linear and bounded operator, then $\iota \circ T: L^{2}(G_{1}) \to \SOprime(G_{2})$ is also linear and bounded. Since $L^{2}(G_{1})$ is weak$^{*}$ dense in $\SOprime(G_{1})$ we may ask if one can extend $\iota\circ T$ to a weak$^{*}$-weak$^{*}$ continuous operator $\widetilde{T}$ from $\SOprime(G_{1})$ into $\SOprime(G_{2})$ such that, in the sense of $\SOprime(G_{2})$, 
\begin{equation} \label{eq:extension} \widetilde{T} \iota(f_{1}) = \iota( T f_{1}) \ \ \text{for all} \ \ f_{1}\in L^{2}(G_{1}). \end{equation}
The following result characterizes when this is possible.

\begin{lemma} \label{le:banach-vs-hilbert-adjoint} Let $T$ be a linear and bounded operator from $L^{2}(G_{1})$ into $L^{2}(G_{2})$. There is an extension of the operator $\iota \circ T: L^{2}(G_{1})\to \SOprime(G_{2})$ to a weak$^{*}$-weak$^{*}$ continuous operator $\widetilde{T}$ from $\SOprime(G_{1})$ into $\SOprime(G_{2})$, if, and only if, the Hilbert space adjoint operator $T^{*}:L^{2}(G_{2})\to L^{2}(G_{1})$ is a linear and bounded operator from $\SO(G_{2})$ into $\SO(G_{1})$. In this case, the extension is unique, and it is defined by the relation 
\begin{equation} \label{eq:L2extension}
 (f_{2},\widetilde{T}\sigma_{1})_{\SO,\SOprime(G_{2})} = ( \overline{T^{*}\overline{f_{2}}}, \sigma_{1})_{\SO,\SOprime(G_{1})} \ \ \text{for all} \ \ f_{2}\in\SO(G_{2}), \ \sigma_{1}\in\SOprime(G_{1}),
\end{equation}
or, equivalently,
\begin{equation} \label{eq:L2extension-v2} (f_{2},\overline{\widetilde{T}\sigma_{1}})_{\SO,\SOprime(G_{2})} = ( T^{*}f_{2}, \overline{\sigma_{1}})_{\SO,\SOprime(G_{1})} \ \ \text{for all} \ \ f_{2}\in\SO(G_{2}), \ \sigma_{1}\in\SOprime(G_{1}).
\end{equation}
\end{lemma}
\begin{proof} Assume that $T^{*}$ is a linear and bounded operator from $\SO(G_{2})$ into $\SO(G_{1})$. Then we define $\widetilde{T}$ by \eqref{eq:L2extension} as the Banach space adjoint of the operator $f_{2} \mapsto \overline{T^{*} \overline{f_{2}}}$. By this definition $\widetilde{T}$ is weak$^{*}$-weak$^{*}$ continuous. With \eqref{eq:regular-distribution} it is straight forward to establish that, for all $f_{1}\in L^{2}(G_{1})$ and $f_{2}\in \SO(G_{2})$,
\begin{align*} (f_{2},\widetilde{T}\iota(f_{1}))_{\SO,\SOprime(G_{2})} & = (\overline{T^{*}\overline{f_{2}}},\iota(f_{1}))_{\SO,\SOprime(G_{1})} = \langle f_{1}, T^{*}\overline{f_{2}}\rangle_{L^{2}(G_{1})} \\
& = \langle Tf_{1},\overline{f_{2}}\rangle_{L^{2}(G_{2})} = (f_{2},\iota(Tf_{1}))_{\SO,\SOprime(G_{2})}.\end{align*}
This proves \eqref{eq:extension} and so $\widetilde{T}$ is an extension of $\iota\circ T$. 

Conversely, assume now that $\widetilde{T}$ is a weak$^{*}$-weak$^{*}$ continuous extension of $\iota\circ T$. Then 
\[ (\widetilde{T})^{\times} : \SO(G_{2})\to \SO(G_{1}), \ ( ( \widetilde{T})^{\times} f_{2}, \sigma_{1})_{\SO,\SOprime(G_{1})} = ( f_{2}, \widetilde{T} \sigma_{1})_{\SO,\SOprime(G_{2})} ,\]
with $f_{2}\in \SO(G_{2})$ and $\sigma_{1}\in \SOprime(G_{1})$, defines a linear and bounded operator. Furthermore, for all $f_{1}\in \SO(G_{1})$ and $f_{2}\in \SO(G_{2})$ we find that
\begin{align*}
\langle \overline{(\widetilde{T})^{\times} \overline{f_{2}}},f_{1}\rangle_{L^{2}(G_{1})} & = ( \overline{(\widetilde{T})^{\times} \overline{f_{2}}}, \iota(\overline{f_{1}}) )_{\SO,\SOprime(G_{1})} \stackrel{\eqref{eq:conjugation-distribution}}{=} \overline{( (\widetilde{T})^{\times}\overline{f_{2}} , \overline{\iota(\overline{f_{1}})}  )}_{\SO,\SOprime(G_{1})} \\
& = \overline{( \overline{f_{2}} , \widetilde{T}\iota(f_{1})   )}_{\SO,\SOprime(G_{1})} \stackrel{\eqref{eq:extension}}{=} \overline{( \overline{f_{2}} , \iota(T f_{1}) )}_{\SO,\SOprime(G_{1})} \\
& = \overline{\langle Tf_{1}, f_{2} \rangle}_{L^{2}(G_{2})} = \langle T^{*} f_{2}, f_{1}\rangle_{L^{2}(G_{1})}
\end{align*}
We conclude that 
$T^{*} f_{2} = \overline{(\widetilde{T})^{\times} \overline{f_{2}}}$. 
Since $(\widetilde{T})^{\times}$ and complex conjugation are bounded operators on $\SO$, we conclude that also $T^{*}$ is a linear and bounded operator from $\SO(G_{2})$ into $\SO(G_{1})$. The equality $T^{*} f_{2} = \overline{(\widetilde{T})^{\times} \overline{f_{2}}}$ implies the relation in \eqref{eq:L2extension}.
If \eqref{eq:L2extension} holds, then 
\begin{align*}
 (f_{2},\overline{\widetilde{T} \sigma_{1}})_{\SO,\SOprime(G_{2})} \stackrel{\eqref{eq:conjugation-distribution}}{=} \overline{(\overline{f_{2}} , \widetilde{T} \sigma_{1}) }_{\SO,\SOprime(G_{2})} \stackrel{\eqref{eq:L2extension}}{=} \overline{( \overline{T^{*} f_{2}} , \sigma_{1} )}_{\SO,\SOprime(G_{1})} \stackrel{\eqref{eq:conjugation-distribution}}{=} (T^{*}f_{2},\overline{\sigma_{1}} )_{\SO,\SOprime(G_{1})},
\end{align*}
which is \eqref{eq:L2extension-v2}. The converse implication is proven in a similar way. It is left to show that the extension $\widetilde{T}$ is unique. Assume therefore that $S:\SOprime(G_{1})\to\SOprime(G_{2})$ is another weak$^{*}$-weak$^{*}$ continuous operator which is an extension of $\iota\circ T$. Then $S$ must, necessarily, coincide with $\widetilde{T}$ on all distributions induced by functions on $L^{2}(G)$. This is a weak$^{*}$-dense subspace of $\SOprime(G_{2})$. Due to the weak$^{*}$-weak$^{*}$ continuity we conclude that the operators $S$ and $\widetilde{T}$ coincide.
\end{proof}

\begin{corollary} \label{cor:extension-of-unitaries} Let $T$ be a unitary operator from $L^{2}(G_{1})$ onto $L^{2}(G_{2})$. The operator $T$ is a Banach space isomorphism from $\SO(G_{1})$ onto $\SO(G_{2})$, if, and only if, there is an extension of the operator $\iota\circ T: L^{2}(G_{1})\to \SOprime(G_{2})$ to weak$^{*}$-weak$^{*}$ continuous Banach space isomorphism $\widetilde{T}$ from $\SOprime(G_{1})$ onto $\SOprime(G_{2})$. In this case, $(\widetilde{T})^{-1} = (T^{-1})^{\sim}$, and the operators $\widetilde{T}$ and $(\widetilde{T})^{-1}$ are defined by the relations   
\begin{align}
%& ( f_{2}, \widetilde{T} \sigma_{1})_{\SO,\SOprime(G_{2})} = 
%( \overline{T^{-1}\overline{f_{2}}}, \sigma_{1})_{\SO,\SOprime(G_{1})}, && (f_{1}, (\widetilde{T})^{-1} \sigma_{2})_{\SO,\SOprime(G_{1})} = ( \overline{T\overline{f_{1}}}, \sigma_{2})_{\SO,\SOprime(G_{2})}.
%\\
%& (f_{1}, \overline{\sigma_{1}})_{\SO,\SOprime(G_{1})} = ( Tf_{1}, \overline{\widetilde{T} \sigma_{1}})_{\SO,\SOprime(G_{2})}, && ( f_{2}, \overline{\sigma_{2}} )_{\SO,\SOprime(G_{2})} = ( T^{-1}f_{2}, \overline{(\widetilde{T})^{-1}\sigma_{2}} )_{\SO,\SOprime(G_{2})},\\
& (T^{-1} f_{2}, \overline{\sigma_{1}})_{\SO,\SOprime(G_{1})} = (f_{2}, \overline{\widetilde{T}\sigma_{1}})_{\SO,\SOprime(G_{2})}, && (Tf_{1}, \overline{\sigma_{2}})_{\SO,\SOprime(G_{2})} = (f_{1}, \overline{(\widetilde{T})^{-1} \sigma_{2}})_{\SO,\SOprime(G_{1})}, \label{eq:L2-extension-unitary}
\end{align}
for all $f_{i}\in \SO(G_{i})$ and $\sigma_{i}\in \SOprime(G_{i})$, $i=1,2$.
\end{corollary}

\begin{proof} The assumptions imply that we are in the position to apply Lemma \ref{le:banach-vs-hilbert-adjoint} to $T$ and $T^{-1}$. Let us show that $(T^{-1})^{\sim} = (\widetilde{T})^{-1}$.
Indeed, for all $f_{1}\in\SO(G_{1})$ and $\sigma_{1}\in \SOprime(G_{1})$,
\begin{align*}
(f_{1},(T^{-1})^{\sim} \, \widetilde{T}\sigma_{1})_{\SO,\SOprime(G_{1})} & \stackrel{\eqref{eq:L2extension}}{=} ( \overline{T\overline{f_{1}}}, \widetilde{T}\sigma_{1})_{\SO,\SOprime(G_{2})} \stackrel{\eqref{eq:L2extension}}{=} ( \overline{T^{-1} T \overline{f_{1}}}, \sigma_{1})_{\SO,\SOprime(G_{1})} \\
& = ( f_{1}, \sigma_{1})_{\SO,\SOprime(G_{1})}.
\end{align*}
Hence $(T^{-1})^{\sim} \, \widetilde{T} = \textnormal{Id}_{\SOprime(G_{1})}$. Similarly we show that $\widetilde{T}(T^{-1})^{\sim} = \textnormal{Id}_{\SOprime(G_{2})}$. We conclude that $(T^{-1})^{\sim} = (\widetilde{T})^{-1}$. 
It follows from Lemma \ref{le:banach-vs-hilbert-adjoint} and \eqref{eq:L2extension-v2} that $\widetilde{T}$ and $(\widetilde{T})^{-1}$ are defined by the equalities in  \eqref{eq:L2-extension-unitary}.
\end{proof}

Of course, the equalities in Corollary \ref{cor:extension-of-unitaries} are the generalization of the well-known relations for unitary operators on $L^{2}$, i.e., if $T:L^{2}(G_{1})\to L^{2}(G_{2})$ is a unitary operator, then
\[ \langle T f, h\rangle_{L^{2}(G_{2})} = \langle f , T^{-1} h\rangle_{L^{2}(G_{1})} \ \ \text{for all} \ \ f\in L^{2}(G_{1}), \ h\in L^{2}(G_{2}).\]

Lemma \ref{le:banach-vs-hilbert-adjoint} and Corollary \ref{cor:extension-of-unitaries} can be applied to the operators considered in Example \ref{ex:unitaries-on-S0}.
Let us demonstrate the results on the Fourier transform.

\begin{example} \label{ex:FT} The Fourier transform $\mathcal{F}$ is a unitary operator from $L^{2}(G)$ onto $L^{2}(\ghat)$ and, as shown in Example \ref{ex:unitaries-on-S0}, it is also a Banach space isomorphism from $\SO(G)$ onto $\SO(\ghat)$. By Lemma \ref{le:banach-vs-hilbert-adjoint} and Corollary \ref{cor:extension-of-unitaries} its unique weak$^{*}$-weak$^{*}$ continuous extension $\widetilde{\mathcal{F}}$ to an operator from $\SOprime(G)$ onto $\SOprime(\ghat)$ is
\[ \widetilde{\mathcal{F}}:\SOprime(G)\to\SOprime(\ghat), \ (h,\widetilde{\mathcal{F}}\sigma)_{\SO,\SOprime(\ghat)} =  (\overline{\mathcal{F}^{-1}\overline{h}},\sigma)_{\SO,\SOprime(G)} = (\mathcal{F}_{\ghat}h,\sigma)_{\SO,\SOprime(G)},\]
where $\mathcal{F}_{\ghat}$ is the Fourier transform from $\SO(\ghat)$ onto $\SO(G)$, $h\in \SO(\ghat)$ and $\sigma\in\SOprime(G)$. 
Equivalently, the extended Fourier transform $\widetilde{\mathcal{F}}$ is defined by the relation
\[ ( \mathcal{F}^{-1} h, \overline{\sigma})_{\SO,\SOprime(G)} = (h,\overline{\tilde{\mathcal{F}}\sigma})_{\SO,\SOprime(\ghat)} \ \ \textnormal{for all} \ \ h\in \SO(\ghat), \ \sigma\in \SOprime(G).\]
Using the definition of multiplication and convolution of functions with distributions \eqref{eq:multiplication-distribution}, \eqref{eq:convolution-distribution} it is straight forward to show that, for all $f\in\SO(G)$ and $\sigma\in\SOprime(G)$,
\[ \widetilde{\mathcal{F}}(f \cdot \sigma) = \mathcal{F}f * \widetilde{\mathcal{F}}\sigma \ \ \text{and} \ \ \widetilde{\mathcal{F}}(f * \sigma) = \mathcal{F}f \cdot \widetilde{\mathcal{F}}\sigma.\]
For a given $\omega\in\ghat$ let $e_{\omega}$ denote the distribution in $\SOprime(G)$ induced by the function $x\mapsto \omega(x)$ and let $\delta_{\omega}\in\SOprime(\ghat)$ be the Dirac delta distribution at the point $\omega\in\ghat$. Then, for all $f\in\SO(G)$,
\[ (f,e_{\omega})_{\SO,\SOprime(G)} = \int_{G} f(x) \, \omega(x) \, dx = \mathcal{F}f(-\omega).\] 
Let us find the distributional Fourier transform of $e_{\omega}$. For any $h\in \SO(\ghat)$,
\[ (h, \widetilde{\mathcal{F}}e_{\omega})_{\SO,\SOprime(\ghat)} = (\overline{\mathcal{F}^{-1}\overline{h}}, e_{\omega})_{\SO,\SOprime(G)} = \mathcal{F}( \overline{\mathcal{F}^{-1}\overline{h}})(-\omega) = h(\omega) = (h,\delta_{\omega})_{\SO,\SOprime(\ghat)}.\] 
This shows that $\widetilde{\mathcal{F}}e_{\omega} = \delta_{\omega}$. 
If we let $\delta_{x}\in\SOprime(G)$ be the Dirac delta distribution at the point $x\in G$ and $e_{x}$ be the distribution in $\SOprime(\ghat)$ induced by the function $\omega\mapsto \omega(x)$. Then, for any $h\in\SO(\ghat)$,
\[ (h,\widetilde{\mathcal{F}}\delta_{x})_{\SO,\SOprime(\ghat)} = (\overline{\mathcal{F}^{-1} \overline{h}}, \delta
_{x})_{\SO,\SOprime(G)} = \overline{\mathcal{F}^{-1}\overline{h(x)}} = \int_{\ghat} h(\omega) \, \overline{\omega(x)} \, d\omega = (h,e_{-x})_{\SO,\SOprime(\ghat)}, \]
 hence $\widetilde{\mathcal{F}}\delta_{x} = e_{-x}$.

Let $H$ be a closed subgroup of $G$ and let $H^{\perp}$ be its annihilator (see Section \ref{sec:restriction}). Furthermore, we let $\mu_{H}$ be some Haar measure on $H$ and $\mu_{H^{\perp}}$ be the orthogonal Haar measure on $H^{\perp}$. We can think of $\mu_{H}$ and $\mu_{H^{\perp}}$ as elements in $\SOprime(G)$ and $\SOprime(\ghat)$ in the following way: for all $f\in \SO(G)$,
\[ (f,\mu_{H})_{\SO,\SOprime(G)} = \int_{H} f(h) \, d\mu_{H}(h), \ \ (\hat{f},\mu_{H^{\perp}})_{\SO,\SOprime(\ghat)} = \int_{H^{\perp}} \hat{f}(\gamma) \, \mu_{H^{\perp}}(\gamma).\]
Using the Poisson formula and the Fourier transform $\tilde{\mathcal{F}}$ from $\SOprime(G)$ onto $\SOprime(\ghat)$ as considered in Example \ref{ex:FT}, we can establish that $\widetilde{\mathcal{F}}\mu_{H} = \mu_{H^{\perp}}$.
For example, if we define the Fourier transform on $\SO(\R)$ as in Remark \ref{rem:poisson}, then, in the sense of $\SOprime(\R)$,
\[ \tilde{\mathcal{F}} \, \Big( \sum_{k\in a\Z} \delta_{k} \Big) = \vert a\vert^{-1}   \sum_{k\in a^{-1}\Z} \delta_{k} \ \ \text{for all} \ \ a\in \R\backslash\{0\}.\]
\end{example} 
%In a more concrete setting, let $H$ be a discrete subgroup of $G$, such that $G/H$ is compact and let $\delta_{x}$, $x\in G$ be the Dirac delta distribution as in Lemma \ref{le:delta-in-SOprime}. Then for $\sigma = \sum_{h\in H} \delta_{h}$ we have that $\mathcal{F}\sigma = s(H)^{-1} \sum_{\gamma\in H^{\perp}} \delta_{\gamma}$. 

In Lemma \ref{le:banach-vs-hilbert-adjoint} we characterized when operators defined on $L^{2}$ can be extended to operators on $\SOprime$.
In the following example we extend two operators to $\SOprime$, which, a priori, are only defined on $\SO$. 
\begin{example}
Let $H$ be an open (and therefore also closed) subgroup of $G$. Let 
\[ R_{H}: \SO(G)\to \SO(H) \ \ \text{and} \ \ Q_{H}: \SO(H)\to \SO(G)\]
 be the restriction and zero-extension operator as in Theorem~\ref{th:periodization-restricion-map-in-SO} and Proposition~\ref{pr:zero-ext}.

The unique weak$^{*}$-weak$^{*}$ continuous extension $\widetilde{Q}_{H}$ of 
$\iota \circ Q_{H}: \SO(H) \to \SOprime(G)$ to an operator  from $\SOprime(H)$ into $\SOprime(G)$ is the Banach space adjoint of the restriction operator. That is,
\[ \widetilde{Q}_{H} : \SOprime(H)\to \SOprime(G) , \ ( f, \widetilde{Q}_{H} \sigma_{H})_{\SO,\SOprime(G)} = ( R_{H} f, \sigma_{H} )_{\SO,\SOprime(H)} , \, f\in \SO(G),\,\sigma_{H}\in \SOprime(H).\]
Indeed, for all functions $\varphi \in\SO(H)$ and $f\in \SO(G)$,
\begin{align*}
( f, \iota (Q_{H} \varphi))_{\SO,\SOprime(G)} & = \int_{G} f(x) \, Q_{H} \varphi(x) \, d\mu_{G}(x) \\
& = \sum_{\dot{x}\in G/H} \int_{H} f(x+t) \, Q_{H} \varphi(x+t) \, d\mu_{H}(t) \\
& = \int_{H} f(t) \, \varphi(t) \, d\mu_{H}(t) = ( R_{H} f, \iota(\varphi))_{\SO,\SOprime(H)} \\ & = (f, \widetilde{Q}_{H} \iota(\varphi))_{\SO,\SOprime(G)}.
\end{align*}
Hence $\widetilde{Q}_{H} \iota(\varphi) = \iota(Q_{H} \varphi)$ for all $\varphi\in \SO(H)$ and thus $\widetilde{Q}_{H}$ is an extension of $\iota\circ Q_{H}$.

In a similar way one can show that the Banach space adjoint of the zero-extension operator, $Q_{H}^{\times}$, is the weak$^{*}$-weak$^{*}$ continuous extension of the operator $\iota\circ R_{H}:\SO(G)\to \SOprime(H)$. And we therefore define
\[ \widetilde{R}_{H} : \SOprime(G)\to \SOprime(H), \ (f, \widetilde{R}_{H} \sigma)_{\SO,\SOprime(H)} = ( Q_{H} f, \sigma)_{\SO,\SOprime(G)},\, f\in \SO(H),\, \sigma\in \SOprime(G).\]

\end{example}

The results in this section are related to the fact that $(\SO,L^{2},\SOprime)$ forms a so-called \emph{rigged Hilbert space}, also called a \emph{Gelfand triple}. The idea of rigged Hilbert spaces was introduced by Gelfand in \cite{MR0173945}. Notably, rigged Hilbert spaces find applications in the mathematical formulation of quantum mechanics \cite{an98-2,de05-1}. The triplet $(\SO,L^{2},\SOprime)$ and especially its applications to time-frequency analysis is further explored in \cite{feko98,dofegr06,feluwe07,cofelu08,fe09}.

\subsection{The Short-time Fourier transform on $\SOprime(G)$} \label{sec:STFT-on-SOprime}

In this section we extend the short-time Fourier transform from an operator on $\SO$ and $L^{2}$ to an operator on $\SOprime$. There are two immediate possibilities to do this:

(I) By definition of the short-time Fourier transform 
\[ \mathcal{\mathcal{V}}_{h}f (\chi) = \langle f, \pi(\chi) h \rangle = ( \overline{\pi(\chi) h}, \iota(f) )_{\SO,\SOprime(G)}, \ \ f,h\in \SO(G).\]
We are therefore inclined to define the short-time Fourier transform of a distribution $\sigma\in \SOprime(G)$ with respect to a function $h\in \SO(G)$ to be the function $\chi \mapsto (\overline{\pi(\chi)h},\sigma)_{\SO,\SOprime(G)}$. 

(II) In Theorem \ref{th:stft-and-s0} we proved that $\mathcal{V}_{h}$ with $h\in \SO(G)$ is a linear and bounded operator from $\SO(G)$ into $\SO(G\times\ghat)$. Furthermore, we showed that the $L^{2}$-Hilbert space adjoint $\mathcal{V}_{h}^{*}$ of the short-time Fourier transform $\mathcal{V}_{h}$ defines a linear and bounded operator from $L^{1}(G\times\ghat)$ onto $\SO(G)$ (or, in fact, as an operator from $\SO(G\times\ghat)$ onto $\SO(G)$). With these results and with the methods from the previous section we can extend the short-time Fourier transform and its adjoint to operators on $\SOprime$. 

Below, in Lemma \ref{le:STFT-on-SOprime}, we detail the process of extending the operator $\mathcal{V}_{h}$ and $\mathcal{V}_{h}^{*}$ as mentioned in (II) and show that this coincides with the idea in (I).
In order to state the result we let $j$ be the isometric isomorphism from $L^{\infty}(G\times\ghat)$ onto $(L^{1}(G\times\ghat))'$,
\[ j(H) = F \mapsto \int_{G\times\ghat} F(\chi) \, H(\chi)\, d\chi, \ \ H\in L^{\infty}(G\times\ghat), \ F\in L^{1}(G\times\ghat). \label{page:l1-linfty}
\footnote{If $G$ is $\sigma$-compact and metrizable then the Haar measure on $G\times\ghat$ is $\sigma$-finite. Hence $(L^{1}(G\times\ghat))'$ can be identified with $L^{\infty}(G\times\ghat)$ in the usual way. If $G$ is a general locally compact abelian Hausdorff group, then the Haar measure on $G\times\ghat$ may not be $\sigma$-finite. In order to still have the identification of $(L^{1})'$ with $L^{\infty}$ one redefines $L^{\infty}$ to be the set of all mesureable functions that are locally almost everywhere bounded. I.e., every function in $L^{\infty}$ is bounded except on a set $N$, where for all compact sets $K$ the intersection $N\cap K$ has measure zero. See \cite[\S 2.3]{fo89} and \cite[\S 12]{hero63}.}\]

\begin{lemma} \label{le:STFT-on-SOprime} Let $h$ be a non-zero function in $\SO(G)$.
\begin{enumerate} 
\item[(i)] The operator 
\[ \widetilde{\mathcal{V}}_{h} : \SOprime(G) \to (L^{1}(G\times\ghat))', \ (F,\widetilde{\mathcal{V}}_{h} \sigma)_{L^{1},(L^{1})'} = ( \overline{\mathcal{V}_{h}^{*} \overline{F}}, \sigma)_{\SO,\SOprime(G)}, \ \ F\in L^{1}(G\times\ghat)\]
is the unique weak$^{*}$-weak$^{*}$ continuous extension of the operator $j \circ \mathcal{V}_{h} : \SO(G)\to (L^{1}(G\times\ghat))'$ to an operator from $\SOprime(G)$ into $(L^{1}(G\times\ghat))'$. In particular, in the sense of $(L^{1}(G\times\ghat))'$, 
\[ \widetilde{\mathcal{V}}_{h} \iota (f) = j( \mathcal{V}_{h}f) \ \ \text{for all} \ \ f\in \SO(G).\]
Moreover, for all $F\in L^{1}(G\times\ghat)$ and all $\sigma\in \SOprime(G)$,
\[ (F,\widetilde{\mathcal{V}}_{h}\sigma)_{L^{1},(L^{1})'} = \int_{G\times\ghat} F(\chi) \, (\overline{\pi(\chi)h},\sigma)_{\SO,\SOprime} \, d\chi.\]
That is, the $L^{1}(G\times\ghat)$-functional $\widetilde{\mathcal{V}}_{h}\sigma$ is induced by the uniform continuous and bounded function
\[ G\times\ghat\to \C, \ \chi \mapsto ( \overline{\pi(\chi)h}, \sigma)_{\SO,\SOprime(G)}, \ \ h\in \SO(G), \, \sigma\in\SOprime(G).\]
\item[(ii)] Let $\mathcal{V}_{h}^{*}$ be the linear and bounded operator from $L^{1}(G\times\ghat)$ onto $\SO(G)$ defined in Theorem \ref{th:stft-and-s0}. Then  $\mathcal{V}_{h}^{*}F$, $F\in L^{1}(G\times\ghat)$ is the unique element in $\SO(G)$ such that
\[ (\mathcal{V}_{h}^{*} F,\sigma)_{\SO,\SOprime(G)} = \int_{G\times\ghat} F(\chi) \, (\pi(\chi)h,\sigma)_{\SO,\SOprime} \, d\chi \ \ \text{for all} \ \ \sigma\in \SOprime(G).\]
\item[(iii)] The operator 
\[ \widetilde{\mathcal{V}}_{h}^{*} : \SOprime(G\times\ghat) \to \SOprime(G), \ (f,\widetilde{\mathcal{V}}_{h}^{*}\psi)_{\SO,\SOprime(G)} = ( \overline{\mathcal{V}_{h}\overline{f}}, \psi)_{\SO,\SOprime(G\times\ghat)}, \ \ f\in \SO(G) \]
is the unique weak$^{*}$-weak$^{*}$ continuous extension of the operator $\iota\circ \mathcal{V}_{h}^{*}:L^{1}(G\times\ghat)\to\SOprime(G)$ to an operator from $\SOprime(G\times\ghat)$ onto $\SOprime(G)$. In particular, in the sense of $\SOprime(G)$,
\[ \widetilde{\mathcal{V}}_{h}^{*} \iota(F) = \iota( \mathcal{V}_{h}^{*} F) \ \ \text{for all} 
\ \ F \in L^{1}(G\times\ghat).\]
\item[(iv)] \label{le:STFT-on-SOprime-iii} If $g,h\in \SO(G)$ and $\sigma\in\SOprime(G)$, then $\widetilde{\mathcal{V}}_{h}^{*} \widetilde{\mathcal{V}}_{g} \sigma = \langle h,g\rangle \, \sigma$. Furthermore, for all $f\in \SO(G)$, 
\[ \langle h,g\rangle (f,\sigma)_{\SO,\SOprime(G)} = ( \mathcal{V}_{g} f,\overline{\widetilde{\mathcal{V}}_{h}\overline{\sigma}})_{L^{1},(L^{1})'} = \int_{G\times\ghat} \langle f, \pi(\chi) g\rangle \, ( \pi(\chi)h,\sigma)_{\SO,\SOprime}\, d\chi.\]
\end{enumerate}
\end{lemma}
\begin{proof}(i). Note that $\widetilde{\mathcal{V}}_{h}$ is the Banach space adjoint operator of $F\mapsto \overline{\mathcal{V}_{h}^{*}\overline{F}}$, $F\in L^{1}(G\times\ghat)$, which is a linear and bounded operator from $L^{1}(G\times\ghat)$ onto $\SO(G)$. Hence $\widetilde{\mathcal{V}}_{h}$ is well-defined, linear, bounded and weak$^{*}$-weak$^{*}$ continuous from $\SOprime(G)$ into $(L^{1}(G\times\ghat))'$. Let now $f$ and $F$ be functions in $\SO(G)$ and $\SO(G\times\ghat)$, respectively. Then
\begin{align*}
& (F, \widetilde{\mathcal{V}}_{h}\iota(f))_{L^{1},(L^{1})'} = (\overline{\mathcal{V}_{h}^{*} \overline{F}}, \iota(f))_{\SO,\SOprime(G)} = \langle f, \mathcal{V}_{h}^{*} \overline{F}\rangle  \\
& = \langle \mathcal{V}_{h}f, \overline{F}\rangle = \int_{G\times\ghat} F(\chi) \, \langle f, \pi(\chi)h\rangle \, d\chi = (F,j(\mathcal{V}_{h}f))_{L^{1},(L^{1})'}.
\end{align*}
Since $F$ was an arbitrary function in $\SO(G\times\ghat)$ and $\SO(G\times\ghat)$ is dense in $L^{1}(G\times\ghat)$ it follows that $\widetilde{\mathcal{V}}_{h}\iota(f)$ and $j(\mathcal{V}_{h}f)$ are the same elements in $(L^{1}(G\times\ghat))'$. Hence $\widetilde{\mathcal{V}}_{h}$ is an extension of $j\circ \mathcal{V}_{h}:\SO(G)\to (L^{1}(G\times\ghat))'$. The uniqueness of this extension follows as in the proof of Lemma \ref{le:banach-vs-hilbert-adjoint} and the fact that $\SO(G)$ is weak$^{*}$ dense in $\SOprime(G)$ (Lemma \ref{le:weak-star-dense-soprime}). Let us now prove the moreover part. To this end we define the operator 
\[ T_{h}: L^{1}(G\times\ghat)\to \SO''(G), \ T_{h}(F) =\sigma \mapsto \int_{G\times\ghat} F(\chi) \, (\overline{\pi(\chi)h},\sigma)_{\SO,\SOprime} \, d\chi.\]
It is straight forward to show that this is a well-defined, linear and bounded operator. In particular, for some $g\in \SO(G)\backslash\{0\}$,
\[ \vert T_{h} (F)(\sigma) \vert \le \Vert F \Vert_{1} \, \Vert \overline{h} \Vert_{\SO,g} \, \Vert \sigma \Vert_{\SOprime,g}\]
and $\Vert T_{h} \Vert_{\textnormal{op},L^{1}\to \SO''} \le \Vert \overline{h} \Vert_{\SO,g}$. Let us now show that for all $F\in L^{1}(G\times\ghat)$ the functional $T_{h}F:\SOprime(G)\to \C$ is in fact weak$^{*}$ continuous. I.e., $T_{h}F$ maps weak$^{*}$ convergent nets in $\SOprime(G)$ into weak$^{*}$ convergent (and since $\C$ is finite dimensional also norm convergent) nets in $\C$.
Since we consider weak$^{*}$-weak$^{*}$ continuity for a functional, it follows by \cite[Theorem 2.7.8]{me98} that it is necessary and sufficient to confirm the weak$^{*}$ continuity for every \emph{bounded} weak$^{*}$ convergent net $(\sigma_{\lambda})$. Hence, if $(\sigma_{\lambda})$ is a bounded weak$^{*}$ convergent net in $\SOprime(G)$ with limit $\sigma\in\SOprime(G)$ we wish to show that
\[ \lim_{\lambda} T_{h}(F)(\sigma_{\lambda}) = T_{h}(F)(\sigma).\]  
For the moment let us restrict ourselves to functions $F$ which are the indicator function on a compact set $K\subseteq G\times\ghat$. We can then make the estimate
\begin{align*} 
\lim_{\lambda} \vert T_{h}(F)(\sigma_{\lambda}) - T_{h}(F)(\sigma)\vert & = \lim_{\lambda} \Big\vert \int_{K} (\overline{\pi(\chi)h},\sigma_{\lambda})_{\SO,\SOprime} - (\overline{\pi(\chi)h},\sigma)_{\SO,\SOprime} \Big\vert \\
& \le \mu_{G\times\ghat}(K) \, \lim_{\lambda} \sup_{\chi\in K} \vert (\overline{\pi(\chi)h},\sigma_{\lambda}-\sigma)_{\SO,\SOprime}\vert.
\end{align*}
Thus, if we can verify uniform convergence of $(\overline{\pi(\chi)h},\sigma_{\lambda})_{\SO,\SOprime}$ on the compact set $K$, then we are done. 
Let $c = \sup_{\lambda} \Vert \sigma_{\lambda} \Vert_{\SOprime,g}<\infty$ and let $U_{\epsilon}$ be a neighbourhood around the identity of $G\times\ghat$ such that $\Vert \overline{\pi(\chi) h -  h} \Vert_{\SO,g} < \epsilon / c$ for all $\chi\in U_{\epsilon}$ (this is possible due to Lemma \ref{le:time-freq-shift-continuous-on-SO}). Then, for all $\chi_{0}\in G \times\ghat$ and for all $\lambda$, we establish that
\[ \vert (\overline{\pi(\chi_{0}+\chi) h},\sigma_{\lambda})-(\overline{\pi(\chi_{0}) h},\sigma_{\lambda})\vert \le \Vert \overline{\pi(\chi)h-h }\Vert_{\SO,g} \, \Vert \sigma_{\lambda} \Vert_{\SOprime,g} < \epsilon \ \ \textnormal{for all} \ \ \chi\in U_{\epsilon}. \]
Hence the functions $\{(\overline{\pi(\chi)h},\sigma_{\lambda})_{\SO,\SOprime}\}_{\lambda}$ are equicontinuous. Since this estimate also does not depend on $\chi_{0}$, the functions are, in fact, uniformly equicontinuous. 
Because of this, the a-priori pointwise convergence $\lim_{\lambda}(\overline{\pi(\chi)h},\sigma_{\lambda}) =(\overline{\pi(\chi)h},\sigma)$ guaranteed by the assumed weak$^{*}$ convergence for all $\chi\in G\times\ghat$ implies uniform convergence over compact sets. I.e.,
\[ \lim_{\lambda} \sup_{\chi\in K} \vert (\overline{\pi(\chi)h},\sigma_{\lambda}-\sigma)_{\SO,\SOprime}\vert = 0. \]
We have thus shown that for the indicator function on any compact set $F$ the functional $T_{h}(F):\SOprime(G)\to \C$ is weak$^{*}$ continuous. Due to linearity of $T_{h}$ the weak$^{*}$ continuity of $T_{h}(F)$ extends from constant functions on compact sets to all simple functions $F$, i.e., piecewise constant functions on compact sets with finitely many distinct function values. This space is dense in $L^{1}(G\times\ghat)$. By a standard $\epsilon/3$ proof we can pass the weak$^{*}$ continuity of $T_{h}F$ from simple functions to all functions in $L^{1}(G\times\ghat)$. We have thus shown that $T_{h}$ maps every integrable function into a weak$^{*}$ continuous functional on $\SOprime(G)$. As for every Banach space, the space of all weak$^{*}$ continuous functionals on $\SOprime(G)$ is isometrically isomorphic to $\SO(G)$ itself. We can therefore think of $T_{h}$ as a linear and bounded operator from $L^{1}(G\times\ghat)$ into $\SO(G)$ and we write
\[ T_{h}:L^{1}(G\times\ghat)\to \SO(G), \ T_{h} F = \int_{G\times\ghat} F(\chi) \, \overline{\pi(\chi) h} \, d\chi.\]  
The integral is to be understood exactly as in the definition of $T_{h}$ from before, i.e., $T_{h}F$, $F\in L^{1}(G\times\ghat)$ is the unique element in $\SO(G)$ such that
\[ (T_{h}F, \sigma)_{\SO,\SOprime(G)} = \int_{G\times\ghat} F(\chi) \, (\overline{\pi(\chi)h}, \sigma)_{\SO,\SOprime(G)} \ \ \text{for all} \ \ \sigma\in \SOprime(G).\]
It is straight forward to show that $T_{h}F = \overline{\mathcal{V}_{h}^{*}\overline{F}}$ for all $F\in \SO(G\times\ghat)$. Indeed, for all $f\in L^{2}(G)$,  
\begin{align*} \langle \overline{\mathcal{V}_{h}^{*} \overline{F}}, f\rangle & = \overline{\langle \overline{F}, \mathcal{V}_{h} \overline{f} \rangle} = \int_{G\times\ghat} F(\chi) \, \langle \overline{f},\pi(\chi)h \rangle \, d\chi = \int_{G\times\ghat} F(\chi) \, ( \overline{\pi(\chi) h}, \iota(\overline{f}) ) \, d\chi \\
& = (T_{h} F, \iota(\overline{f}) ) = \langle T_{h}F,f\rangle.\end{align*}
Hence $T_{h}F = \overline{\mathcal{V}_{h}^{*}\overline{F}}$. From this equality and the definition of $\widetilde{\mathcal{V}}_{h}$, we conclude that
\[ (F,\widetilde{\mathcal{V}}_{h}\sigma)_{L^{1},(L^{1})'} = \int_{G\times\ghat} F(\chi) \, (\overline{\pi(\chi)h},\sigma)_{\SO,\SOprime} \, d\chi\]
and so
the functional $\widetilde{\mathcal{V}_{h}}\sigma\in (L^{1}(G\times\ghat))'$ is induced by the uniform continuous and bounded function $\chi\mapsto (\overline{\pi(\chi)h},\sigma)$. \\
(ii). By the just established equality $T_{h} F = \overline{\mathcal{V}_{h}^{*} \overline{F}}$ we can easily show that, for all $F\in L^{1}(G\times\ghat)$, $h\in \SO(G)$ and $\sigma\in \SOprime(G)$,
\begin{align*} (\mathcal{V}_{h}^{*} F,\sigma) &= (\overline{T_{h} \overline{F}},\sigma) = \overline{(T_{h} \overline{F},\overline{\sigma})} = \overline{\int_{G\times\ghat} \overline{F(\chi)} \, (\overline{\pi(\chi) h},\overline{\sigma} ) \, d\chi } \\
& = \int_{G\times\ghat} F(\chi) \, (\pi(\chi)h, \sigma) \, d\chi. \end{align*}
(iii). If $f\in \SO(G)$ and $F\in L^{1}(G\times\ghat)$, then 
\begin{align} (f, \widetilde{\mathcal{V}}_{h}^{*} \iota(F) ) & = (\overline{\mathcal{V}_{h}\overline{f}}, \iota(F))_{\SO,\SOprime(G\times\ghat)} = \int_{G\times\ghat} F(\chi) \, \langle \pi(\chi) h, \overline{f}\rangle \, d\chi \nonumber \\
& = \int_{G\times\ghat} F(\chi) \, ( \pi(\chi) h, \iota(f))_{\SO,\SOprime(G)} \, d\chi \label{eq:1307aaa} \end{align}
On the other hand, using the statement in (ii) and the equality $(f, \iota(h))_{\SO,\SOprime(G)} = (h,\iota(f))_{\SO,\SOprime(G)}$ for all $f,h\in \SO(G)$, we establish that
\begin{equation} (f, \iota(\mathcal{V}_{h}^{*} F))_{\SO,\SOprime(G)} = (\mathcal{V}_{h}^{*}F,\iota(f))_{\SO,\SOprime(G)} = \int_{G\times\ghat} F(\chi) \, ( \pi(\chi) h, \iota(f))_{\SO,\SOprime(G)} \, d\chi. \label{eq:1307bbb} \end{equation}
Combining \eqref{eq:1307aaa} and \eqref{eq:1307bbb} we conclude that $\widetilde{\mathcal{V}}_{h}^{*}\iota(F)$ and $\iota(\mathcal{V}_{h}^{*}F)$ define the same elements in $\SOprime(G)$.
The remainder of the statement follows as in (i), where we need to know that the short-time Fourier transform $\mathcal{V}_{h}$, $h\in \SO(G)$ is a linear and bounded operator from  $\SO(G)$ into $\SO(G\times\ghat)$ (Theorem \ref{th:stft-and-s0}). 
Finally we prove (iv). By definition of $\widetilde{\mathcal{V}}_{h}^{*}$, $\widetilde{\mathcal{V}}_{g}$ and \eqref{eq:STFT-dual},
\[ (f,\widetilde{\mathcal{V}}_{h}^{*}\widetilde{\mathcal{V}}_{g}\sigma)_{\SO,\SOprime} = (\overline{ \mathcal{V}_{g}^{*} \mathcal{V}_{h} \overline{f}},\sigma)_{\SO,\SOprime} = (\overline{ \langle g,h\rangle \overline{f}},\sigma)_{\SO,\SOprime} = \langle h,g\rangle (f,\sigma)_{\SO,\SOprime}.\]
Hence $\widetilde{\mathcal{V}}_{h}^{*}\widetilde{\mathcal{V}}_{g}\sigma = \langle h,g\rangle \, \sigma$. Using this identity with the roles $g$ and $h$ interchanged, we immediately have the equalities
\[ \langle g,h\rangle \, (\overline{f}, \sigma)_{\SO,\SOprime} = ( \overline{\mathcal{V}_{g}f} , \widetilde{\mathcal{V}}_{h} \sigma)_{L^{1},(L^{1})'} = \int_{G\times\ghat} \langle \pi(\chi) g, f\rangle \, (\overline{\pi(\chi)h},\sigma)_{\SO,\SOprime} \, d\chi .\]
Applying a complex conjugation yields the equality
\[ \langle h,g\rangle (f,\overline{\sigma})_{\SO,\SOprime(G)} = ( \mathcal{V}_{g} f,\overline{\widetilde{\mathcal{V}}_{h}\sigma})_{L^{1},(L^{1})'} = \int_{G\times\ghat} \langle f, \pi(\chi) g\rangle \, ( \pi(\chi)h,\overline{\sigma})_{\SO,\SOprime}\, d\chi.\]
Finally, replacing $\sigma$ by $\overline{\sigma}$ and using that $\overline{\overline{\sigma}}=\sigma$ yields the desired equality.
\end{proof}

%\begin{example}
%Let $\delta_{\chi_{0}}$, $\chi_{0}\in G\times\ghat$ be the Dirac delta distribution at the point $\chi_{0}$. It is an easy calculation to verify that $\widetilde{\mathcal{V}}_{h}^{*} \delta_{\chi_{0}} = \iota( \pi(\chi_{0}) h)$. 
%\end{example}

The following result is an important statement concerning $\SOprime(G)$, which will be essential for the proof of Theorem \ref{th:SO-expansion}.

%which itself is the key for Theorem \ref{th:SO-minimality-bochner} which, in turn, allows us to prove the remaining results in Section \ref{sec:minimal} and all the results in Sections \ref{sec:restriction}, \ref{sec:series-expansions} and \ref{sec:kernel}.

\begin{proposition} \label{pr:SOprime-equiv-norms} Fix any $g\in \SO(G)\backslash\{0\}$, then
\[ \Vert \cdot \Vert_{M^{\infty},g} : \SOprime(G) \to \R^{+}_{0}, \ \Vert \sigma \Vert_{M^{\infty},g} =  \sup_{\chi\in G\times\ghat} \vert (\pi(\chi)g , \sigma)_{\SO,\SOprime(G)} \vert \]
defines a norm on $\SOprime(G)$ which is equivalent to $\Vert \cdot \Vert_{\SOprime,g}$. In particular, for all $\sigma\in\SOprime(G)$,
\[  \Vert g \Vert_{\SO,g}^{-1} \, \Vert \sigma \Vert_{M^{\infty},g} \le \Vert \sigma \Vert_{\SOprime,g} \le  \Vert g \Vert_{2}^{-2} \, \Vert \sigma \Vert_{M^{\infty},g}.\]
\end{proposition}
\begin{proof} 
By Lemma \ref{le:STFT-on-SOprime}(iv) we have that for all $f\in \SO(G)$ and $\sigma\in\SOprime(G)$,
\[ (f,\sigma) = \Vert g \Vert_{2}^{-2}\int_{G\times\ghat} \langle f, \pi(\chi) g \rangle \, (\pi(\chi) g, \sigma) \, d\chi\]
With these relations the upper inequality follows easily:
\begin{align*}
\Vert \sigma \Vert_{\SOprime,g} & = \sup_{\substack{f\in \SO(G) \\ \Vert f \Vert_{\SO,g}=1}} \vert ( f, \sigma) \vert = \Vert g \Vert_{2}^{-2} \, \sup_{\substack{f\in \SO(G) \\ \Vert f \Vert_{\SO,g}=1}} \Big\vert \int_{G\times\ghat} \langle f, \pi(\chi) g \rangle \, (\pi(\chi) g, \sigma) \, d\chi \, \Big\vert \\
& \le \Vert g \Vert_{2}^{-2} \sup_{\substack{f\in \SO(G) \\ \Vert f \Vert_{\SO,g}=1}}  \Vert f \Vert_{\SO,g}   \, \sup_{\chi\in G\times\ghat} \vert (\pi(\chi) g, \sigma) \vert = \Vert g \Vert_{2}^{-2} \, \Vert \sigma\Vert_{M^{\infty},g}. \end{align*}
The lower equality follows from the straight forward estimate
\[ \Vert \sigma\Vert_{M^{\infty},g} = \sup_{\chi\in G\times\ghat} \vert (\pi(\chi)g,\sigma) \vert \le \Vert g\Vert_{\SO,g} \, \Vert \sigma\Vert_{\SOprime,g}.\]
We omit the details concerning the norm-axioms.
\end{proof}

The ``$M^{\infty},g$'' in the index of the norm in Proposition \ref{pr:SOprime-equiv-norms} is related to the modulation space of order $\infty$ and will be explained shortly.

\subsection{Characterization of $\SO(G)$ as elements of $\SOprime(G)$}
\label{sec:modulation-spaces}
The results on the short-time Fourier transform in the previous section allow us to show further characterizations of $\SO(G)$ and they allow us to introduce the \emph{modulation spaces}.

It is useful to let $\widetilde{\mathcal{V}}_{g}\sigma$, $g\in \SO(G)$, $\sigma\in\SOprime(G)$, \emph{not} denote a functional in $(L^{1}(G\times\ghat))'$ as defined in Lemma \ref{le:STFT-on-SOprime}(i) but rather the uniform continuous and bounded function that induces this functional, i.e.,
\[ \widetilde{\mathcal{V}}_{g}\sigma (\chi) = (\overline{\pi(\chi)g},\sigma)_{\SO,\SOprime(G)}, \ \sigma\in \SOprime(G), \ \chi\in G\times\ghat.\]

For a fixed function $g\in \SO(G)\backslash\{0\}$ we consider the following two sets from Definition \ref{def:S0-bigdef}:
\begin{enumerate}[]
\item $\mathscr{J} = \{ \sigma \in \SOprime(G) \, : \, \widetilde{\mathcal{V}}_{g} \sigma \in \SO(G\times\ghat)\}$,
\item $\mathscr{K} = \{ \sigma \in \SOprime(G) \, : \, \widetilde{\mathcal{V}}_{g} \sigma \in L^{1}(G\times\ghat)\}$. 
\end{enumerate}
%The condition $\widetilde{\mathcal{V}}_{g} \sigma \in \SO(G\times\ghat)$ and $\widetilde{\mathcal{V}}_{g} \sigma \in L^{1}(G\times\ghat)$ means that the functional $\widetilde{\mathcal{V}}_{g}\sigma \in (L^{1}(G\times\ghat))'$ is induced by a function in $\SO(G\times\ghat)$ and $L^{1}(G\times\ghat)$, respectively.

\begin{theorem} \label{th:SO-definition-via-SOprime} Let $g$ be a function in $\SO(G)\backslash\{0\}$. For any locally compact abelian group $G$ it holds that $\iota(\SO(G)) = \mathscr{J} = \mathscr{K}$. That is, a distribution $\sigma\in \SOprime(G)$ is induced by a function $f\in\SO(G)$ such that $\sigma=\iota(f)$, if and only if $\widetilde{\mathcal{V}}_{g}\sigma$ is a function in $\SO(G\times\ghat)$, or, equivalently, $\widetilde{\mathcal{V}}_{g}\sigma$ is a function in $L^{1}(G\times\ghat)$. Furthermore, the vector space $\mathscr{K}\subseteq \SOprime(G)$ with the norm
\[ \Vert \cdot \Vert_{\mathscr{K},g} : \mathscr{K} \to \R_{0}^{+}, \ \Vert \sigma \Vert_{\mathscr{K},g} = \Vert \widetilde{\mathcal{V}}_{g}\sigma \Vert_{1} = \int_{G\times\ghat} \vert ( \overline{\pi(\chi) g}, \sigma)_{\SO,\SOprime(G)} \vert \, d\chi \]
is isometric isomorphic to $(\SO(G),\Vert \cdot \Vert_{\SO(G),g})$.
\end{theorem}
\begin{proof} If $f\in\SO(G)$, then, for the distribution $\iota(f)$ induced by $f$,
\begin{equation} \label{eq:1206a} \widetilde{\mathcal{V}}_{g}\iota(f)(\chi) = ( \overline{\pi(\chi) g}, \iota(f))_{\SO,\SOprime} = \langle f, \pi(\chi) g \rangle =  \mathcal{V}_{g}f(\chi).\end{equation}
By Lemma \ref{le:cont-embedded-in-Lp} and Theorem \ref{th:stft-and-s0}(ii) we know that $\mathcal{V}_{g}f\in \SO(G\times\ghat)\subseteq L^{1}(G\times\ghat)$. 
This proves that $\iota(\SO(G))\subseteq \mathscr{J}\subseteq \mathscr{K}$. 

Let now $\sigma$ be an element in $\mathscr{K}$. That is, $\widetilde{\mathcal{V}}_{g}\sigma\in L^{1}(G\times\ghat)$. Then, since $\mathcal{V}_{g}^{*}$ maps $L^{1}(G\times\ghat)$ onto $\SO(G)$, we can define a function $h\in \SO(G)$ by $h=\Vert g \Vert_{2}^{-2} \mathcal{V}_{g}^{*} \widetilde{\mathcal{V}}_{g} \sigma$. As it turns out $\sigma = \iota(h)$. Indeed, for all $f\in \SO(G)$,
\begin{align}
(f,\iota(h))_{\SO,\SOprime} & = (h,\iota(f))_{\SO,\SOprime} = \Vert g \Vert_{2}^{-2} ( \mathcal{V}_{g}^{*} \widetilde{\mathcal{V}}_{g} \sigma, \iota(f))_{\SO,\SOprime} \nonumber \\
{\small\{\textnormal{Lemma \ref{le:STFT-on-SOprime}(ii)}\}} \ \ \ & = \Vert g \Vert_{2}^{-2} \int_{G\times\ghat} (\overline{\pi(\chi) g}, \sigma) \, \langle \pi(\chi) g, \overline{f}\rangle \, d\chi \nonumber \\
& = \Vert g \Vert_{2}^{-2} \overline{\int_{G\times\ghat} (\pi(\chi) g, \overline{\sigma}) \, \langle\overline{f}, \pi(\chi) g\rangle \, d\chi} \nonumber \\
{\small{\{\textnormal{Lemma \ref{le:STFT-on-SOprime}(iv)}\}}} \ \ \ & = \overline{(\overline{f},\overline{\sigma})} = (f,\sigma). \label{eq:1206b}
\end{align}
We conclude that $\mathscr{J}\subseteq \iota(\SO(G))$. It remains to show the furthermore part.
To this end define two operators,
\begin{align*}
d & : \SO(G) \to \mathscr{K}, \ d(f) = \iota(f),\\
 e & : \mathscr{K} \to \SO(G), \ e (\sigma) = \Vert g \Vert_{2}^{-2} \, \mathcal{V}_{g}^{*} \widetilde{\mathcal{V}}_{g} \sigma.
\end{align*}
It is clear that these are well-defined and linear. It remains to show that they are bounded and that
$d\circ e (\sigma) = \sigma$ and $e\circ d(f) = f$ for all $f\in \SO(G)$ and $\sigma\in \mathscr{K}$.
By use of $\eqref{eq:1206a}$ we easily establish that, for all $f\in \SO(G)$,
\[ \int_{G\times\ghat} \vert ( \overline{\pi(\chi)g},\iota(f))_{\SO,\SOprime(G)} \vert \, d\chi = \int_{G\times\ghat} \vert \mathcal{V}_{g}f(\chi) \vert \, d\chi  = \Vert f \Vert_{\SO,g}. \]
Hence $\Vert d(f) \Vert_{\mathscr{K},g} = \Vert f \Vert_{\SO,g}$.
By \eqref{eq:1206b} we have the equality $\Vert g\Vert_{2}^{-2} (\mathcal{V}_{g}^{*}\widetilde{\mathcal{V}}_{g}\sigma, \iota(f)) = (f,\sigma)$. Therefore, with $f=\overline{\pi(\chi)g}$, we can conclude that
\begin{align*} \int_{G\times\ghat} \vert \langle e(\sigma), \pi(\chi)g\rangle \vert \, d\chi & = \int_{G\times\ghat} \vert \Vert g \Vert_{2}^{-2} ( \mathcal{V}_{g}^{*}\widetilde{\mathcal{V}}_{g}\sigma, \iota(\overline{\pi(\chi)g}) \vert \, d\chi \\
& = \int_{G\times\ghat} \vert ( \overline{\pi(\chi)g},\sigma)\vert \, d\chi = \Vert \sigma \Vert_{\mathscr{K},g}.\end{align*}
Thus $\Vert e (\sigma) \Vert_{\SO,g} = \Vert \sigma \Vert_{\mathscr{K},g}$. We have thus shown that both $d$ and $e$ are isometries. Using the calculation in \eqref{eq:1206b} one can show that $d\circ e(\sigma) = \sigma$ for all $\sigma\in \mathscr{K}$. By use of Lemma \ref{le:STFT-on-SOprime}(ii)+(iv) one can verify that $e\circ d(f) = f$ for all $f\in \SO(G)$. This completes the proof.
\end{proof}

Following the characterization of $\SO(G)$ via the set $\mathscr{K}$ and Proposition \ref{pr:SOprime-equiv-norms}, it is at this point natural to define the modulation spaces. 
\begin{definition} \label{def:mod-space} Fix a function $g\in \SO(G)\backslash\{0\}$ and let $p,q\in [1,\infty]$. The \emph{modulation space of order} $(p,q)$, $\textstyle{M}^{p,q}(G)$, is the set of distributions
$\sigma\in \SOprime(G)$ for which  
\[ \int_{\ghat} \Big( \int_{G} \vert (E_{\omega}T_{x} g, \sigma)_{\SO,\SOprime(G)} \vert^{p} \, dx\Big)^{q/p} \, d\omega  < \infty,\]
with the obvious modification for the situation where either or both $p$ an $q$ equal $\infty$.
\end{definition}

By use of Lemma \ref{le:STFT-on-SOprime}(iv) one can, as with $\SO(G)$ in Proposition \ref{pr:equivalent-norm-on-s0}, show that the definition of $M^{p,q}$ does not depend on the function $g$. In particular, different $g\in \SO(G)\backslash\{0\}$ induce equivalent norms on $M^{p,q}(G)$ via
\[ \Vert \sigma \Vert_{M^{p,q},g} = \bigg( \int_{\ghat} \Big( \int_{G} \vert (E_{\omega}T_{x}g,\sigma)_{\SO,\SOprime} \vert^{p} \, dx\Big)^{q/p} \, d\omega \bigg)^{1/q}.\]
Note that $(E_{\omega}T_{x}g,\sigma)=\widetilde{\mathcal{V}}_{\overline{g}}\sigma(x,-\omega)$. Since $\SO(G)$ is invariant under complex conjugation and the mixed norm spaces $L^{p,q}(G\times\ghat)$ are invariant under automorphisms in the second variable it follows that $\sigma\in \SOprime(G)$ belongs to $M^{p,q}(G)$ if and only if for any (and thus all) $g\in \SO(G)\backslash\{0\}$ the short-time Fourier transform $\widetilde{\mathcal{V}}_{g}\sigma$ belongs to the mixed norm space $L^{p,q}(G\times\ghat)$. If $p=q$, we write $M^{p}(G)$ instead of $M^{p,p}(G)$. Note that $M^{1}(G)\cong\SO(G)$ and $M^{\infty}(G)=\SOprime(G)$ by Theorem \ref{th:SO-definition-via-SOprime} and Proposition \ref{pr:SOprime-equiv-norms}, respectively.
Furthermore, it is possible to show that if $p_{1}\le p_{2}$ and $q_{1}\le q_{2}$, then $M^{p_{1},q_{1}}(G) \subseteq M^{p_{2},q_{2}}(G)$. Note that the inequalities in Corollary~\ref{co:f-in-s0-properties}(x) imply that, for a fixed $g\in \SO(G)\backslash\{0\}$, $\Vert \iota(f) \Vert_{M^{p},g} \le \Vert f \Vert_{\SO,g}$ for all $f\in \SO(G)$, $p\in [1,\infty]$. Moreover, one can show that $M^{\infty,1} \cdot \SO \subseteq \SO$ and $M^{1,\infty} * \SO \subseteq \SO$. For more on the modulation spaces see \cite{fegr92-1,fe03-1,fe06} and the relavant chapters in the books \cite{MR1843717} and \cite{go11}. The modulation spaces are a incredibly fruitful setting for time-frequency analysis and for the theory of pseudo-differential operators, cf.\ the references provided in Section~\ref{sec:history}. 

Incidentally, the name \emph{modulation space} originates from the original definition, which, similar to the characterization of $\SO(G)$ via the sets $\mathscr{A,D}$ and $\mathscr{G}$, makes use of the convolution and the \emph{modulation} operator,
\[ M^{p,q}(G) = \{ \sigma\in\SOprime(G) \, : \, \int_{\ghat} \Vert E_{\omega} g * \sigma \Vert_{p}^{q} \, d\omega < \infty\}. \] 
Lastly, we mention that the modulation spaces can equally be described as the Fourier transform of the Wiener amalgam spaces $W(\mathcal{F}L^{p},L^{q})$ \cite{fe83-4,fe03-1}.

As an application of the characterization of $\SO(G)$ in Theorem~\ref{th:SO-definition-via-SOprime}, we prove Lemma~\ref{le:SOprime-time-SO-conv-S0-is-S0} below. 

\begin{lemma} \label{le:SOprime-time-SO-conv-S0-is-S0} If $\sigma\in \SOprime(G)$ and $f, h \in \SO(G)$, then $(\sigma * f)\cdot h$ and $(\sigma \cdot h) * f$ are distributions that are induced by functions in $\SO(G)$. I.e., $(\sigma * f)\cdot h, \, (\sigma \cdot h) * f\in \mathscr{K}$.
\end{lemma}
\begin{proof} By the definitions in \eqref{eq:multiplication-distribution} and \eqref{eq:convolution-distribution} it is clear that $(\sigma*f)\cdot h$ and $(\sigma \cdot f)* h$ belong to $\SOprime(G)$. According to Theorem~\ref{th:SO-definition-via-SOprime} we need to check that their short-time Fourier transform with respect to a non-zero function $g\in \SO(G)$ is integrable. Straightforward calculations show that this is the case. 
\begin{align*}
& \int_{G\times\ghat} \vert \widetilde{\mathcal{V}}_{g} \big( (\sigma*f)\cdot h \big) (\chi) \vert \, d\chi = \int_{G\times\ghat} \vert ( \overline{\pi(\chi)g}, (\sigma*f)\cdot h) \vert \, d\chi \\
& = \int_{G\times\ghat} \vert ( (\overline{\pi(\chi)g} \cdot h )* f^{r}, \sigma ) \vert \, d\chi \le \Vert \sigma \Vert_{\SOprime,g} \int_{G\times\ghat} \Vert (\overline{\pi(\chi)g} \cdot h )* f^{r} \Vert_{\SO,g} \, d\chi \\
& = \Vert \sigma\Vert_{\SOprime,g} \int_{G\times\ghat} \int_{G\times\ghat} \vert \langle (\overline{\pi(\chi)g} \cdot h)*f^{r}, \pi(\widetilde{\chi})g\rangle \vert \, d\widetilde{\chi} \, d\chi\\
& = \Vert \sigma\Vert_{\SOprime,g} \int_{G\times\ghat} \int_{G\times\ghat} \big\vert \int_{G}\int_{G} h(s) f(s-t) \overline{\pi(\widetilde{\chi})g(t)} \overline{\pi(\chi)g(s)} \, dt \, ds \big\vert \, d\widetilde{\chi} \, d\chi \\
& = \Vert \sigma \Vert_{\SOprime(G),g} \, \Vert \tau_{a} ( h \otimes f) \Vert_{\SO(G\times G),g\otimes g} <\infty. 
\end{align*}
A similar calculation yields that 
\[\int_{G\times\ghat} \vert \widetilde{\mathcal{V}}_{g} \big( (\sigma\cdot h)* f \big) (\chi) \vert \, d\chi\le \Vert \sigma \Vert_{\SOprime(G),g} \, \Vert \tau_{a} ( h \otimes f^{r}) \Vert_{\SO(G\times G),g\otimes g} <\infty. 
\]
\end{proof}

As remarked in \cite{fezi98}, one can use Lemma \ref{le:sufficient-cond-to-be-in-S0}(vii)+(viii) and Lemma \ref{le:SOprime-time-SO-conv-S0-is-S0} to construct nets of functions in $\SO(G)$ that, under the embedding $\iota$ of $\SO$ in $\SOprime$,  converge to any given $\sigma\in \SOprime(G)$ in the weak$^{*}$ sense. 

In Section \ref{sec:banach-space-adjoin} we considered linear and bounded operators $T$ from $L^{2}(G_{1})$ into $L^{2}(G_{2})$ and characterized when $\iota\circ T$ can be extended to an operator from $\SOprime(G_{1})$ into $\SOprime(G_{2})$. Let us answer the same question for operators that are only defined on $\SO$.

\begin{lemma} \label{le:SOop-to-SOprime} Let $T$ be a linear and bounded operator from $\SO(G_{1})$ into $\SO(G_{2})$. The operator $T^{\times}$ defines a linear and bounded operator from $M^{1}(G_{2})$ into $M^{1}(G_{1})$, if, and only if, the operator $\iota\circ T:\SO(G_{1})\to \SOprime(G_{2})$ has a unique weak$^{*}$-weak$^{*}$ continuous extension $\widetilde{T}:\SOprime(G_{1})\to \SOprime(G_{2})$. In that case, for any $\sigma_{1}\in\SOprime(G_{1})$ and $f_{2}\in \SO(G_{2})$, $\widetilde{T}$ is defined by the equality
\begin{equation} \label{eq:1807d} (f_{2}, \widetilde{T} \sigma_{1})_{\SO,\SOprime(G_{2})} =         \Vert g \Vert_{2}^{-2} \,  \int_{G_{1}\times\ghat_{1}} ( T\,  \overline{\pi(\chi)g}, \iota(f_{2})) _{\SO,\SOprime(G_{2})} \, (\pi(\chi)g, \sigma_{1})_{\SO,\SOprime(G_{1})} \, d\chi,  \end{equation}
where $g$ is any non-zero function $g\in \SO(G_{1})$.
\end{lemma}
\begin{proof}
Assume that the operator $T^{\times}$ defines a linear and bounded operator from $M^{1}(G_{2})$ into $M^{1}(G_{1})$. By Theorem  \ref{th:SO-definition-via-SOprime} and Definition \ref{def:mod-space} we have that that $M^{1}(G) = \mathscr{K} \cong \SO(G)$. It follows from the operators $d$ and $e$ in the proof of Theorem \ref{th:SO-definition-via-SOprime} that, for any $g\in \SO(G_{1})\backslash\{0\}$, the operator $S = \Vert g \Vert_{2}^{-2} \, \mathcal{V}_{g}^{*} \circ  \widetilde{\mathcal{V}}_{g} \circ T^{\times} \circ \iota$ is well-defined, linear and bounded from $\SO(G_{2})$ into $\SO(G_{1})$ (here $\iota$ is the embedding of $\SO(G_{2})$ into $M^{1}(G_{2})\subseteq\SOprime(G_{2})$). Let us show that the weak$^{*}$-weak$^{*}$ continuous operator $S^{\times}:\SOprime(G_{1})\to \SOprime(G_{2})$ is an extension of $\iota \circ T$. To that end, let $f_{i}\in \SO(G_{i})$, $i=1,2$. By use of Lemma \ref{le:STFT-on-SOprime} we establish that
\begin{align*} & \quad \ (f_{2}, S^{\times} \iota(f_{1}))_{\SO,\SOprime(G_{2})} \\ & = \Vert g\Vert_{2}^{-2} \, (\mathcal{V}_{g}^{*}\widetilde{\mathcal{V}}_{g}  T^{\times} \iota(f_{2}), \iota(f_{1}))_{\SO,\SOprime(G_{1})} \\
& = \Vert g \Vert_{2}^{-2} \, \int_{G_{1}\times\ghat_{1}} (\overline{\pi(\chi) g} , T^{\times}\iota(f_{2}))_{\SO,\SOprime(G_{1})} \, ( \pi(\chi) g, \iota(f_{1}))_{\SO,\SOprime(G_{1})} \, d\chi \\
& = \Vert g \Vert_{2}^{-2} \, \overline{\int_{G_{1}\times\ghat_{1}} (\pi(\chi) g , \overline{T^{\times}\iota(f_{2})})_{\SO,\SOprime(G_{1})} \, ( \overline{\pi(\chi) g}, \overline{\iota(f_{1})})_{\SO,\SOprime(G_{1})} \, d\chi} \\
& = \Vert g \Vert_{2}^{-2} \, \overline{\int_{G_{1}\times\ghat_{1}} \langle \overline{f_{1}} , \pi(\chi)g\rangle \,  (\pi(\chi) g , \overline{T^{\times}\iota(f_{2})})_{\SO,\SOprime(G_{1})} \,  \, d\chi} \\
& = \overline{(\overline{f_{1}}, \overline{T^{\times}\iota(f_{2})})}_{\SO,\SOprime(G_{1})} = (f_{1}, T^{\times}\iota(f_{2}))_{\SO,\SOprime(G_{1})} = (Tf_{1},\iota(f_{2}))_{\SO,\SOprime(G_{2})} \\
& = (f_{2}, \iota( Tf_{1}))_{\SO,\SOprime(G_{2})}.
\end{align*} 
Hence $S^{\times} \iota (f_{1}) = \iota( T f_{1})$ and we therefore conclude that $S^{\times}$ is a weak$^{*}$-weak$^{*}$ continuous extension of the operator $\iota\circ T: \SO(G_{1})\to \SOprime(G_{2})$ to an operator from $\SOprime(G_{1})$ to $\SOprime(G_{2})$. Since $\SO$ is weak$^{*}$ dense in $\SOprime$ it follows that this extension of $\iota\circ T$ is unique. We therefore define $\widetilde{T} = S^{\times}$. The first two steps in the calculation above applied to $(f_{2}, \widetilde{T} \sigma_{1})_{\SO,\SOprime(G_{2})}$, $f_{2}\in \SO(G_{2})$, $\sigma_{1}\in \SOprime(G_{1})$ yield \eqref{eq:1807d}.

Conversely, assume now that $\iota\circ T$ has a weak$^{*}$-weak$^{*}$ continuous extension $\widetilde{T}$ from $\SOprime(G_{1})$ into $\SOprime(G_{2})$. Then $(\widetilde{T})^{\times}$ is a linear and bounded operator from $\SO(G_{2})$ into $\SO(G_{1})$. For any $g\in \SO(G_{2})\backslash\{0\}$ the operator $S = \Vert g \Vert_{2}^{-2} \iota \circ (\widetilde{T})^{\times} \mathcal{V}_{g}^{*} \widetilde{\mathcal{V}}_{g}$ defines a linear and bounded operator from $M^{1}(G_{2})$ into $M^{1}(G_{2})$. Let now $f_{1}\in \SO(G_{1})$ and $\sigma_{2}\in M^{1}(G_{2})$. Then, by similar steps as above, we establish that
\begin{align*}
& \quad \ (f_{1}, S\sigma_{2})_{\SO,\SOprime(G_{1})} \\
& = \Vert g \Vert_{2}^{-2} \int_{G_{2}\times\ghat_{2}} (\overline{\pi(\chi) g} , \sigma_{2})_{\SO,\SOprime(G_{2})} \, ( \pi(\chi) g, \iota(T f_{1}))_{\SO,\SOprime(G_{2})} \, d\chi \\
& = (f_{1}, T^{\times} \sigma_{2})_{\SO,\SOprime(G_{1})}. 
\end{align*}
This shows that $T^{\times}\vert_{M^{1}(G_{2})} = S$ and we conclude that $T^{\times}$ is a linear and bounded operator from $M^{1}(G_{2})$ into $M^{1}(G_{1})$.
\end{proof}

\section{The minimality of the Feichtinger algebra} \label{sec:minimal}

In this section we prove two major results concerning the Feichtinger algebra. Namely, the characterization of $\SO$ by the set $\mathscr{L}$:
given a function $g\in \SO(G)\backslash\{0\}$ then $\SO(G)$ coincides with the set
\begin{enumerate}[]
\item $\mathscr{L} = \{ f \in L^{1}(G) \, : \, f = \sum_{n\in\N} c_{n} E_{\omega_{n}}T_{x_{n}} g \ \text{with} \ (c_{n})_{n\in\N} \in \ell^{1}(\N) \ \text{and} \  (x_{n},\omega_n)\subseteq G\times\ghat\}$.  
\end{enumerate}
Then, using this characterization of $\SO$ we show that the Feichtinger algebra is the smallest time-frequency shift invariant Banach space. 
We apply the minimality of the Feichtinger algebra in Section \ref{sec:apply-minimal} to establish the tensor factorization property of $\SO$, Theorem \ref{th:tensor-property-s0}, a new characterization of $\SO$ among all Banach spaces in $L^{1}$, Theorem \ref{th:S0-Banach-Space-Characterization}, and a characterization of $\SO$ by open subgroups, Theorem \ref{th:reiter-so}.

We need the following result by Bonsall \cite{bo86-1,bo91-2}.

\begin{lemma} \label{le:bonsall} Let $E$ be a subset of a Banach space $B$. The following statements are equivalent:
\begin{enumerate}
\item[(i)] For some $c,C>0$ one has that $ c \, \Vert x' \Vert_{B'}\le \displaystyle{\sup_{x\in E}} \vert x'(x) \vert \le C \, \Vert x' \Vert_{B'}$ for all $x'\in B'$.
\item[(ii)] For all $x \in B$ there is a sequence $(c_{n})\in \ell^{1}(\N)$ and a sequence $(x_{n})_{n\in \N} \subseteq E$ such that $x = \sum_{n\in\N} c_{n} x_{n}$. Furthermore, 
\[ \Vert x \Vert_{\widetilde{B}} = \inf \big\{ \Vert (c_{n})\Vert_{1} \, : \, x = \sum_{n\in\N} c_{n} x_{n}, \ (c_{n})\in \ell^{1}(\N), \ (x_{n})_{n\in \N}\subseteq E\,\big\} \]
defines an equivalent norm on $B$. In particular, for all $x\in B$, $C^{-1} \, \Vert x \Vert_{B} \le \Vert x \Vert_{\widetilde{B}}  \le c^{-1} \, \Vert x \Vert_{B}$.
\end{enumerate}
\end{lemma}

With Lemma \ref{le:bonsall} in hand we can show the above mentioned characterization of $\SO(G)$ via the set $\mathscr{L}$.

\begin{theorem} \label{th:SO-expansion} For any locally compact abelian group $G$ it holds that $\SO(G) = \mathscr{L}$. Moreover
\[ \Vert f \Vert_{\mathscr{L},g} = \inf \Big\{ \Vert (c_{n}) \Vert_{1} : \, f = \textstyle\sum_{n\in\N} c_{n} E_{\omega_{n}}T_{x_{n}} g \ \text{with} \ (c_{n})\in \ell^{1}(\N) \ \text{and} \  (x_{n},\omega_n)\subseteq G\times\ghat\Big\}\]
defines an equivalent norm on $\SO(G)$.
Specifically,
\begin{equation} \label{eq:3004b1} \Vert g \Vert_{2}^{2} \, \Vert f \Vert_{\mathscr{L},g} \le \Vert f \Vert_{\SO,g} \le \Vert g \Vert_{\SO,g} \, \Vert f \Vert_{\mathscr{L},g} \ \ \text{for all} \ f\in \SO(G).\end{equation}
\end{theorem}
\begin{proof}
Use Lemma \ref{le:bonsall} with $B=\SO(G)$ and 
$E=\{\pi(\chi)g\}_{\chi\in G\times\ghat}$. Proposition \ref{pr:SOprime-equiv-norms} states the inequalities in Lemma \ref{le:bonsall}(i). Hence Lemma \ref{le:bonsall}(ii) holds, which is the theorem.\end{proof}

The result of Theorem~\ref{th:SO-expansion} first appeared in \cite{fe87-1}. A generalization for modulation spaces is available in \cite{fe89-1,fegr92-1}. The proof based on the result by Bonsall first appeared in \cite{MR1843717}.

With the characterization of $\SO(G)$ by the set $\mathscr{L}$ it is rather easy to show the minimality of $\SO(G)$.

%The result in Theorem~\ref{th:SO-expansion} is very powerful. From it we will, in Section \ref{sec:series-expansions}, deduce a variety of other series representations of elements in $\SO(G)$. Let us use the result in Theorem \ref{th:SO-expansion} another very important result concerning the Feichtinger algebra, its minimality.

\begin{theorem} \label{th:SO-minimality-bochner}
Let $B$ be a Banach space and assume that there is a non-zero function $g\in \SO(G)$ such that $\{\pi(\chi)g \, : \, \chi\in G\times\ghat\}\subseteq B$. If there is a constant $c>0$ such that $\Vert \pi(\chi) g \Vert_{B} \le c \, \Vert g \Vert_{B}$ for all $\chi\in G\times \ghat$, then $\SO(G)\subseteq B$ and
\[ \Vert f \Vert_{B} \le c \, \Vert g \Vert_{B} \, \Vert g \Vert_{2}^{-2} \, \Vert f \Vert_{\SO,g} \ \ \text{for all} \ \ f\in \SO(G).\]
\end{theorem}
\begin{proof} By Theorem~\ref{th:SO-expansion} any $f\in \SO(G)$ can be written as $f = \sum_{n\in\N} c_{n} \pi(\chi_{n}) g $
for some suitable $(c_{n})_{n\in\N}\in \ell^{1}(\N)$ and $(\chi_{n})_{n\in\N}\subseteq G\times\ghat$. This implies that any $f\in \SO(G)$ is the limit of a Cauchy sequence in $B$ and thus $\SO(G)\subseteq B$. Moreover, for all $f\in \SO(G)$,
\[ \Vert f \Vert_{B} \le \sum_{n\in\N} \vert c_{n} \vert \, \Vert \pi(\chi_{n})g \Vert_{B} \le c \, \Vert g \Vert_{B}  \, \Vert (c_{n}) \Vert_{\ell^{1}(\N)}.\]
If we take the infimum over all admissible representations $f = \sum_{n\in\N} c_{n} \pi(\chi_{n}) g $, then we establish the inequality
\[ \Vert f \Vert_{B} \le c \, \Vert g \Vert_{B} \Vert f \Vert_{\mathscr{L},g}.\]
The norm equivalence stated in Theorem~\ref{th:SO-expansion} yields the desired inequality.\end{proof}

The minimality of $\SO(G)$ expressed in Theorem~\ref{th:SO-minimality-bochner} is an extension of the early result concerning the Feichtinger algebra which established that $\SO(G)$ is the smallest among all time-frequency shift invariant Segal algebras \cite{Feichtinger1981}, cf.\ Proposition \ref{pr:min-segal}. An alternative proof of Theorem \ref{th:SO-minimality-bochner} can be found in \cite{fezi98}.

\subsection{Applying the minimality of the Feichtinger algebra} \label{sec:apply-minimal}

As a first application of the minimality of the Feichtinger algebra, let us finish the proof of Theorem \ref{th:periodization-restricion-map-in-SO} and show that the periodization and restriction operator with respect to a closed subgroup $H$ of $G$  map \emph{onto} $\SO(G/H)$ and $\SO(H)$, respectively.

\begin{proof}[Proof of Theorem \ref{th:periodization-restricion-map-in-SO}, part 2]
We have already shown that the periodization operator 
\[ P_{H}: \SO(G) \to \SO(G/H), \ P_{H}f(\dot{x}) = \int_{H} f(x+h) \, d\mu_{H}(h), \ \dot{x}=x+H, \ x\in G,\] 
is a linear and bounded operator. In order to show that $P_{H}$ is surjective we argue as follows. Consider $P_H(\SO(G))$ with the quotient norm. That is, for a function $\varphi\in P_H(\SO(G))$, with $\varphi = P_{H}f$ for some $f\in \SO(G)$, we measure its norm by
\[ \Vert \varphi \Vert_{P_H(\SO(G))} = \inf_{f_{0}\in \text{ker}\,P_H} \Vert f +f_{0}\Vert_{\SO}. \]
One can verify that the norm turns $P_{H}(\SO(G))$ into a Banach space and that the norm is invariant under time-frequency shifts $E_{\gamma}T_{\lambda}$ with $\gamma\in H^{\perp}$ and $\lambda\in G/H$. Furthermore we already showed that $P_{H}(\SO(G))\subseteq \SO(G/H)$. Theorem~\ref{th:SO-minimality-bochner} now implies that the time-frequency shift invariant Banach space $P_{H}(\SO(G))$ contains $\SO(G/H)$. It follows that $P_{H} (\SO(G)) = \SO(G/H)$ and $P_{H}$ must be surjective. That the restriction operator is surjective follows as remarked in the (first part of the) proof of Theorem \ref{th:periodization-restricion-map-in-SO}. 
\end{proof}

We will now show the tensor factorization property of the Feichtinger algebra.
In Theorem~\ref{th:stft-and-s0} we established that the tensor product $\otimes$ maps $\SO(G_{1})\times \SO(G_{2})$ into $\SO(G_{1}\times G_{2})$. 
The tensor product can be used to construct the \emph{projective tensor product} of the spaces $\SO(G_{1})$ and $\SO(G_{2})$. For fixed $g_{i}\in \SO(G_{i})$, $i=1,2$, we define
\begin{align*}
  \SO(G_{1})& \hat{\otimes} \, \SO(G_{2}) \\
  = \ & \{ F\in \SO(G_{1}\times G_{2}) \, : \, F = \textstyle\sum_{j\in \N} f_{1,j} \otimes f_{2,j} \ \text{with} \ (f_{i,j})_{j\in\N} \subseteq \SO(G_{i}), i=1,2 \\
 & \hspace{5cm} \text{ for which } \ \textstyle\sum_{j\in \N} \Vert f_{1,j}  \Vert_{\SO(G_{1}),g_{1}} \Vert f_{2,j} \Vert_{\SO(G_{2}),g_{2}} < \infty\}.
\end{align*}
The projective tensor product $\SO(G_{1})\hat{\otimes}\, \SO(G_{2})$ becomes a Banach space under the norm given by
\[ \Vert F \Vert_{\SO(G_{1})\hat{\otimes} \SO(G_{2}), g_{1}\otimes g_{2}} = \inf \ \sum_{j\in \N} \Vert f_{1,j}  \Vert_{\SO(G_{1}),g_{1}} \Vert f_{2,j} \Vert_{\SO(G_{2}),g_{2}}, \]
where the infimum is taken over all representations $F = \sum_{j\in \N} f_{1,j} \otimes f_{2,j}$. For more on tensor products of Banach spaces, see, \eg \cite{MR1888309}.

From the definition of the projective tensor product it is clear that 
\[\SO(G_{1}) \hat{\otimes}\, \SO(G_{2})\subseteq \SO(G_{1}\times G_{2}).\]
We will now show that these two sets coincide. The result goes back to \cite{Feichtinger1981}.

\begin{theorem} \label{th:tensor-property-s0} Let $G_{1}$ and $G_{2}$ be two locally compact abelian groups. Then 
\[ \SO(G_{1}\times G_{2}) = \SO(G_{1}) \hat{\otimes}\, \SO(G_{2}).\]
 Furthermore, with $g_{i}\in \SO(G_{i})\backslash\{0\}$, $i=1,2$, and for all $F\in \SO(G_{1}\times G_{2})$
\begin{equation} \label{eq:2402a} c \, \Vert F \Vert_{\SO(G_{1})\hat{\otimes}\, \SO(G_{2}),g_{1}\otimes g_{2}} \le \Vert F \Vert_{\SO(G_{1}\times G_{2}),g_{1}\otimes g_{2}} \le \Vert F \Vert_{\SO(G_{1})\widehat{\otimes} \SO(G_{2}),g_{1}\otimes g_{2}}, \end{equation}
with $c=\Vert g_{1} \Vert_{\SO,g_{1}}^{-1} \Vert g_{2} \Vert_{\SO,g_{2}}^{-1} \Vert g_{1} \Vert_{2}^{2} \Vert g_{2} \Vert_{2}^{2}$.
\end{theorem}
\begin{proof}
Consider a function $F\in \SO(G_{1})\hat{\otimes}\, \SO(G_{2})$. Then
\[ \Vert F \Vert_{\SO(G_{1}\times G_{2}), g_{1}\otimes g_{2}} \le \sum_{j\in \N} \Vert f_{1,j} \otimes f_{2,j} \Vert_{\SO(G_{1}\times G_{2}),g_{1}\otimes g_{2}}  = \sum_{j\in \N} \Vert f_{1,j} \Vert_{\SO(G_{1}),g_{1}} \Vert f_{2,j} \Vert_{\SO(G_{2}),g_{2}}. \]
Taking the infimum over all representations of $F\in \SO(G_{1})\hat{\otimes}\, \SO(G_{2})$ yields the upper inequality in \eqref{eq:2402a} and shows that $\SO(G_{1})\hat{\otimes}\, \SO(G_{2}) \subseteq \SO(G_{1}\times G_{2})$. For the other inclusion we note that $\SO(G_{1})\hat{\otimes}\,\SO(G_{2})$ is invariant under time-frequency shifts. Indeed, for all $F\in \SO(G_{1})\hat{\otimes}\,\SO(G_{2})$, we have that
\[ \Vert \pi(\chi) F \Vert_{\SO(G_{1})\hat{\otimes}\, \SO(G_{2}),g_{1}\otimes g_{2}} = \Vert F \Vert_{\SO(G_{1})\hat{\otimes}\, \SO(G_{2}),g_{1}\otimes g_{2}} \ \ \text{for all} \ \ \chi\in G_{1}\times G_{2} \times \widehat{G}_{1} \times \widehat{G}_{2}.\]
The minimality of $\SO$ expressed in Theorem~\ref{th:SO-minimality-bochner} implies that $\SO(G_{1}\times G_{2})\subseteq \SO(G_{1})\hat{\otimes}\, \SO(G_{2})$ as well as the lower inequality.
\end{proof}

%The result of Theorem \ref{th:tensor-property-s0} is the key to the kernel theorem for the Feichtinger algebra, see Section \ref{sec:kernel}.
%\begin{remark} \label{rem:L1AG} The following was pointed out by the referee: for non-discrete, non-compact, locally compact abelian groups $G$ \cite[Corollary 3.3]{fospwo07} states that the Segal algebra $L^{1}(G)\cap A(G)$ does not have the tensor-factorization property. Hence, for such groups, $\SO(G)\subsetneq L^{1}(G)\cap A(G)$. In contrast, if $G$ is discrete or compact, then equality holds, see Lemma \ref{le:s0-discrete-compact}.
%\end{remark} 

Next to the tensor factorization property in Theorem \ref{th:tensor-property-s0}, we have shown that the Fourier transform, topological group automorphisms and the translation operator are linear and bounded operators on $\SO$. Furthermore, $\SO(G)$ is continuously embedded into $L^{1}(G)$. As it turns out, these properties characterize the Feichtinger algebra among all Banach spaces contained in $L^{1}(G)$.

\begin{theorem}\label{th:S0-Banach-Space-Characterization} Let $\{ B(G) \}$ be a family of non-trivial Banach spaces defined for all locally compact abelian groups $G$ such that $B(G)\subseteq L^{1}(G)$. Consider the following properties for $\{B(G)\}$.

\begin{enumerate}
\item[(i)] There is a constant $c>0$ such that $\Vert f \Vert_{L^{1}(G)} \le c\, \Vert f \Vert_{B(G)}$ for all $f\in B(G)$.
\item[(ii)] For all $x\in G$ the translation operator $T_x:B(G)\to B(G), \, T_x f(t) = f(t-x)$ is a linear and bounded operator with uniformly bounded operator norm over all $x\in G$.
\item[(iii)] For all topological group automorphisms $\alpha$ of $G$ the operator $f\mapsto f\circ \alpha$ is linear and bounded on $B(G)$.
\item[(iv)] $B(G_{1})\widehat{\otimes} B(G_{2}) = B(G_{1}\times G_{2})$ for all locally compact abelian groups $G_{1}$ and $G_{2}$.
\item[(v)] The Fourier transform is a linear and bounded operator from $B(G)$ into $B(\ghat)$.
\end{enumerate}
If (i)-(v) hold, then $B(G)=\SO(G)$ for all $G$ and the norms on $B(G)$ and $\SO(G)$ are equivalent.
\end{theorem}
\begin{proof} Let $f,g\in B(G)$. By property (iii) also $g^{r}\in B(G)$. By property (iv) it follows that the function $f\otimes g^{r}$ belongs to $B(G\times G)$. Since the asymmetric coordinate transform $\tau_{a}$ is an automorphism of $G\times G$ statement (iii) implies that $\tau_{a}(f\otimes g^{r})\in B(G\times G)$. By the tensor factorization property (iv), we have that
\[ \tau_{a}(f\otimes g^{r}) = \sum_{j\in\N} h_{1,j} \otimes h_{2,j},\]
for some sequences $\{h_{i,j}\}_{j\in\N}\subseteq B(G)$, $i=1,2$. Hence, if we apply the partial Fourier transform in the second coordinate, $\mathcal{F}_{2}$, we can establish that
\[ \mathcal{F}_{2} \tau_{a}(f\otimes g^{r}) = \sum_{j\in\N} h_{1,j} \otimes \mathcal
F h_{2,j}.\]
Since $\mathcal{F}$ is assumed to be a linear and bounded operator from $B(G)$ into $B(\ghat)$, it is clear that $\mathcal{F}_{2} \tau_{a}(f\otimes g) \in B(G\times\ghat)$.
Lastly, the property in (i) implies that $\mathcal{F}_{2} \tau_{a}(f\otimes g^{r})$ is a function in  $L^{1}(G\times\ghat)$, i.e., 
\[ \int_{G\times\ghat} \big\vert \int_{G} \omega(t) f(t)g(x-t) \, dt \big\vert \, d(x,\omega) = \int_{\ghat} \Vert E_{\omega} f*g\Vert_{1} \, d\omega < \infty.\] The characterization of $\SO(G)$ by the set $\mathscr{A}$ implies that $f\in\SO(G)$. Thus $B(G)\subseteq \SO(G)$. Since all the used operators are linear and bounded there is some constant $c>0$ such that
\[ \Vert f \Vert_{\SO,g} \le c \, \Vert f\Vert_{B} \ \ \text{for all} \ \ f\in B(G).\] 

In order to show the converse inclusion, we note that
$E_{\omega} = \mathcal{F}^{-1} T_{\omega} \mathcal{F}$.
Since $\mathcal{F}^{-1}h(t) = \mathcal{F}_{\ghat}h(-t)$ for all $h\in L^{1}(\ghat)$ and $t\mapsto -t$ is an automorphism of $G$ it follows from (ii), (iii) and (v) that $E_{\omega}$ is a linear and bounded operator from $B(G)$ into itself with uniformly bounded operator norm over all $\omega\in \ghat$. Hence $B(G)$ is a Banach space that satisfies the assumptions of Theorem~\ref{th:SO-minimality-bochner}. Therefore $\SO(G)$ is continuously embedded into $B(G)$. 
\end{proof}

Theorem~\ref{th:S0-Banach-Space-Characterization} is an improvement of the characterization of $\SO(G)$ among all Segal algebras given by Losert in \cite{lo80}. Note that Segal algebras automatically satisfy assumptions (i) and (ii) in Theorem \ref{th:S0-Banach-Space-Characterization}. Losert also assumes (iii), a weaker version of (iv) and a stronger version of (v). Additionally, what is not done here, in \cite{lo80} certain assumptions on the restriction operator are also made.

Note that $L^{1}(G)$ itself satisfies assumption (i)-(iv) in Theorem \ref{th:S0-Banach-Space-Characterization}. It is thus the fifth assertion of Theorem \ref{th:S0-Banach-Space-Characterization} which, clearly, distinguishes $\SO(G)$ from $L^{1}(G)$.

If one weakens assumption (v) so that it only requires the Fourier transform to be a linear and bounded operator from $B(G)$ into $L^{1}(\ghat)$, but adds an assumption for the modulation operator as in (ii) for the translation operator, then the statement of Theorem \ref{th:S0-Banach-Space-Characterization} still holds. 

As a final application of the minimality of the Feichtinger algebra we show the following characterization of $\SO(G)$ from \cite[\S 2.9]{re89}.

\begin{theorem} \label{th:reiter-so} Let $G$ be a locally compact abelian group and let $H$ be an open subgroup of $G$. Let $\Gamma$ be a set of coset representatives of the discrete quotient group $G/H$. For a function $f\in L^{1}(G)$ let $f_{\gamma}$, $\gamma\in \Gamma$ be the function $f_{\gamma}(x) = f(\gamma+x)$, $x\in H$. 
A function $f\in L^{1}(G)$ belongs to $\SO(G)$ if, and only if, the functions $\{f_{\gamma}\}_{\gamma\in \Gamma}$ belong to $\SO(H)$ and $\sum_{\gamma\in \Gamma} \Vert f_{\gamma} \Vert_{\SO(H)} < \infty$.
Moreover, $\Vert f \Vert_{\SO,H} = \sum_{\gamma\in \Gamma} \Vert f_{\gamma} \Vert_{\SO(H)}$ defines an equivalent norm on $\SO(G)$.
\end{theorem}
\begin{proof}
Let $Q_{H}$ be the zero-extension operator as in Proposition \ref{pr:zero-ext}. Let $f\in L^{1}(G)$ be such that $\{f_{\gamma}\}_{\gamma\in \Gamma} \subseteq \SO(H)$ and $\sum_{\gamma\in\Gamma} \Vert f_{\gamma} \Vert_{\SO(H)}<\infty$. 
Note that $f=\sum_{\gamma}T_{\gamma} Q_{H} f_{\gamma}$. It follows from Proposition \ref{pr:zero-ext} that $f\in \SO(G)$ and that $\Vert f \Vert_{\SO,g} \le c \, \sum_{\gamma\in\Gamma} \Vert f_{\gamma} \Vert_{\SO(H)}$ for some $c>0$. For the converse inclusion we observe that 
\[ \{ f \in L^{1}(G) \, : \, \{f_{\gamma}\}_{\gamma\in \Gamma}\subseteq\SO(H), \ \sum_{\gamma\in \Gamma} \Vert f_{\gamma} \Vert_{\SO(H)} < \infty\}\]
with the proposed norm is a time-frequency shift invariant Banach space. The minimality of $\SO(G)$, Theorem \ref{th:SO-minimality-bochner}, now implies that these two spaces coincide and have equivalent norms.
\end{proof}

\begin{remark} By the structure theory for locally compact abelian groups \emph{all} locally compact abelian groups, except the Euclidean space, contain a \emph{compact} open subgroup $H$. Hence, for those groups, we have that $f\in L^{1}(G)$ belongs to $\SO(G)$ if, and only if,  $f_{\gamma}\in A(H)$ for all $\gamma\in \Gamma$ and $\sum_{\gamma\in \Gamma} \Vert f_{\gamma} \Vert_{A(H)}< \infty$. I.e., $\SO(G)$ consists of all functions that, when restricted to each coset of the compact open subgroup $H$, are functions with absolutely convergent Fourier series and the Fourier coefficients of all the restrictions over all cosets are absolutely summable. This characterization of $\SO(G)$ is related to the characterization of $\SO(G)$ as the Wiener amalgam space with local component in the Fourier algebra and global component in $L^{1}$, i.e., the characterization by the set $\mathscr{T}$ in Proposition \ref{pr:bupu-time}.
\end{remark}
%The characterization of the Feichtinger algebra via Theorem \ref{th:reiter-so} is particular useful for the understanding of this space on the $p$-adic numbers.

\section{Series expansions of functions in $\SO(G)$ and BUPUs}
\label{sec:series-expansions}

In this section we add more characterizations of $\SO(G)$ by series representations and we characterize $\SO(G)$ by use of bounded uniform partitions of unity. 

For the following result we ask the reader to recall the sets $\mathscr{M}$ to $\mathscr{R}$ from Definition~\ref{def:S0-bigdef}. 
The set $\mathscr{M}$ is the original definition of $\SO(G)$ used by Feichtinger in \cite{Feichtinger1981}, whereas the set $\mathscr{N}$ is closely related to the definition of $\SO(G)$ used in \cite{re89},\cite{MR1802924}. It is no restriction to assume that the sequences of function $(g_{n})_{n\in\N}$ and $(f_{n})_{n\in\N}$ in the sets $\mathscr{M,N,O,P,Q}$ and $\mathscr{R}$ belong to $C_{c}(G)\cap A(G)$ or that they are functions in $L^{1}(G)$ with compactly supported Fourier transform. Most of these series repesentations can be found in \cite{Feichtinger1981}.
In the following $K$ and $\tilde{K}$ are compact sets with non-void interior in $G$ and $\ghat$, respectively.

\begin{proposition} \label{pr:SO-expansion} For any locally compact abelian group $G$ it holds that $\SO(G)=\mathscr{M}=\mathscr{N}=\mathscr{O}=\mathscr{P}=\mathscr{Q}=\mathscr{R}$. Let $g$ be a fixed, non-zero function in $\SO(G)$. Then the following $\SO$-norms are all equivalent to $\Vert \cdot \Vert_{\SO(G),g}$.
\begin{align*}
\textstyle\Vert f \Vert_{\mathscr{M},K} & = \inf\big\{ \sum\limits_{n\in\N} \Vert g_{n} \Vert_{A(G)} \, : \,  f = \sum\limits_{n\in \N} T_{x_{n}} g_{n} , \, (g_{n})_{n\in\N}\subseteq A(G), \, \supp\,g_{n} \subseteq K \, \forall \, n\in\N \\
& \hspace{4.2cm}\text{ with } (x_{n})_{n\in \N} \subseteq G \text{ and } \textstyle\sum\limits_{n\in \N} \Vert g_{n} \Vert_{A(G)} < \infty\big\} \, ; \\
\Vert f \Vert_{\mathscr{N},\tilde{K}} & = \inf\big\{ \textstyle\sum\limits_{n\in\N} \Vert g_{n} \Vert_{1} \, : \, f = \sum\limits_{n\in \N} E_{\omega_{n}} g_{n} , \, (g_{n})_{n\in\N}\subseteq L^{1}(G), \, \supp\,\hat{g}_{n} \subseteq \tilde{K} \, \forall\, n\in\N \\
& \hspace{4.2cm}\text{ with } (\omega_{n})_{n\in \N} \subseteq \ghat \text{ and } \textstyle\sum\limits_{n\in \N} \Vert g_{n} \Vert_{1} < \infty\big\} \, ; \\
 \Vert f \Vert_{\mathscr{O},g}  & = \inf\big\{ \textstyle\sum\limits_{n\in\N} \Vert f_{n} \Vert_{1} \, : \, f = \sum\limits_{n\in \N} f_{n} * E_{\omega_{n}} g, \ (f_{n})_{n\in\N}\subseteq L^{1}(G) \\ 
& \hspace{4.2cm}\text{with} \ \ (\omega_{n})_{n\in \N} \subseteq \ghat, \sum\limits_{n\in \N} \Vert f_{n} \Vert_{1} < \infty\big\} \, ;  \\
 \Vert f \Vert_{\mathscr{P},g} & = \inf\big\{ \sum\limits_{n\in\N} \Vert f_{n} \Vert_{A(G)} \, : \, f = \sum\limits_{n\in \N} f_{n} \cdot T_{x_{n}} g, \ (f_{n})_{n\in\N}\subseteq A(G) \\
& \hspace{4.2cm}\text{with} \ \ (x_{n})_{n\in \N} \subseteq G, \textstyle\sum\limits_{n\in \N} \Vert f_{n} \Vert_{A(G)} < \infty\big\} \, ; \\
 \Vert f \Vert_{\mathscr{Q},g} & = \inf \big\{ \sum_{n\in\N} \Vert f_{n} \Vert_{\SO,g} \, : \, f = R_{\{0\}\times G} \Big( \textstyle\sum\limits_{n\in\N} T_{(\omega_{n},0)} \mathcal{V}_{\hat{g}}\hat{f}_{n}\Big), \ (f_{n})_{n\in\N} \subseteq \SO(G)\\
& \hspace{4.2cm}\text{with} \  (\omega_{n})_{n\in\N} \subseteq \ghat, \ \textstyle\sum\limits_{n\in\N} \Vert f_{n} \Vert_{\SO,g} < \infty\big\} \, ; \\
\Vert f \Vert_{\mathscr{R},g} & = \inf \big\{ \sum_{n\in\N} \Vert f_{n} \Vert_{\SO,g} \, \Vert g_{n} \Vert_{\SO,g} \, : \, f = \sum_{n\in\N} f_{n} *  g_{n} , \ (f_{n})_{n\in\N},(g_{n})_{n\in\N} \subseteq \SO(G) \\
& \hspace{4.2cm}\text{with} \ \ \sum\limits_{n\in\N} \Vert f_{n} \Vert_{\SO,g} \, \Vert g_{n} \Vert_{\SO,g} < \infty \big\}.
\end{align*} 
\end{proposition}
\begin{proof}
Let $f\in \SO(G)$, let the set $K$ be given and take a function $g\in C_{c}(G)\cap A(G)$ with $\supp\, g\subseteq K$. By Lemma~\ref{le:sufficient-cond-to-be-in-S0} we have that $g\in \SO(G)$. Theorem~\ref{th:SO-expansion} implies that $f$ can be written in the form
\begin{equation} \label{eq:3004a1} f = \sum_{n\in\N} c_{n} E_{\omega_{n}}T_{x_{n}} g \end{equation}
for some sequence $(c_{n})_{n\in\N}\in \ell^{1}(\N)$ and $(x_{n},\omega_{n})_{n\in\N}\subseteq G\times\ghat$. Hence $f$ can be written as $f = \sum_{n\in\N} T_{x_{n}} g_{n}$ with $g_{n}=  \omega_{n}(x_{n}) \, c_{n} \, E_{\omega_{n}} g$. This shows that $f\in \mathscr{M}$. Since $f$ was arbitrary this implies that $\SO(G)\subseteq \mathscr{M}$. Furthermore, we find that
\[ \Vert f \Vert_{\mathscr{M},K} = \inf \Big\lbrace \sum_{n} \Vert g_{n} \Vert_{A(G)} \, : \, f = \sum_{n} \ldots \Big\rbrace \le \sum_{n} \Vert \omega_{n}(x_{n}) \, c_{n} \, E_{\omega_{n}} g \Vert_{A(G)} = \Vert g \Vert_{A(G)} \sum_{n} \vert c_{n} \vert.\]
This inequality holds for any representation \eqref{eq:3004a1}. Hence
\[ \Vert f \Vert_{\mathscr{M},K} \le \Vert g \Vert_{A(G)} \, \Vert f \Vert_{\mathscr{L},g}
%\stackrel{\eqref{eq:3004b1}}{\le} \Vert g \Vert_{A(G)} \Vert g \Vert_{2}^{-2} \, \Vert f \Vert_{\SO,g}
.\]
We now show the converse inclusion $\mathscr{M} \subseteq \SO(G)$. Let therefore $g$ be any function in $\SO(G)\backslash\{0\}$ and take a function $h\in \SO(G)$ such that $h= 1$ on $K$, \eg as in Lemma~\ref{le:S0-contains-piecewise-constant-functions}. We then find that for any $f\in \mathscr{M}$
\[ \Vert f \Vert_{\SO,g} \le \sum_{n\in\N} \Vert T_{x_{n}} g_{n} \Vert_{\SO,g} = \sum_{n\in\N} \Vert g_{n} \Vert_{\SO,g} = \sum_{n\in\N} \Vert g_{n} \cdot h \Vert_{\SO,g} \le \sum_{n\in\N} \Vert g_{n} \Vert_{A(G)} \, \Vert h \Vert_{\SO,g} < \infty.\]
This shows that $\mathscr{M} \subseteq \SO(G)$. Moreover, the above estimate holds for any representation of $f\in\mathscr{M}$. Hence $\Vert f \Vert_{\SO,g} \le \Vert h \Vert_{\SO,g} \Vert f \Vert_{\mathscr{M},K}$. The proof for the set $\mathscr{N}$ is similar and thus omitted.

Let $f\in \SO(G)$, fix $h\in L^{1}(G)\backslash\{0\}$ and $g\in \SO(G)\backslash\{0\}$. By Proposition~\ref{pr:s0-l1-ag-ideal} $h*g\in \SO(G)$. Furthermore,
Theorem~\ref{th:SO-expansion} states that $f$ can be written in the form
\[ f = \sum_{n\in\N} c_{n} E_{\omega_{n}}T_{x_{n}} (h*g) = \sum_{n\in\N} (c_{n} E_{\omega_{n}}T_{x_{n}} h)* (E_{\omega_{n}}g)\]
for some sequence $(c_{n})_{n\in\N}\in \ell^{1}(\N)$ and $(x_{n},\omega_{n})_{n\in\N}\subseteq G\times\ghat$. Hence $f$ can be written in the form $f= \sum_{n\in\N} f_{n} * E_{\omega_{n}}g$ with $f_{n}=c_{n} E_{\omega_{n}}T_{x_{n}} h$. This shows that $f\in \mathscr{O}$ and thus $\SO(G)\subseteq \mathscr{O}$. Furthermore, we find that for all $f\in \SO(G)$
\[ \Vert f \Vert_{\mathscr{O},g} = \inf\Big\{ \sum_{n\in\N} \Vert f_{n} \Vert_{1} \, : \, f = \sum_{n} \ldots \Big\} \le \sum_{n\in\N} \Vert c_{n} E_{\omega_{n}} T_{x_{n}}h \Vert_{1} = \Vert h \Vert_{1} \, \sum_{n\in\N} \vert c_{n} \vert.\]
Theorem~\ref{th:SO-expansion} implies that
\[ \Vert f \Vert_{\mathscr{O},g} \le \Vert h \Vert_{1} \Vert f \Vert_{\mathscr{L},h*g}.\]
On the other hand, for any $f\in \mathscr{O}$ we have that
\[ \Vert f \Vert_{\SO,g} \le \sum_{n\in\N} \Vert f_{n} * E_{\omega_{n}}g \Vert_{\SO,g} \le \sum_{n\in\N} \Vert f_{n} \Vert_{1} \, \Vert g \Vert_{\SO,g}.\]
This shows that $\mathscr{O}\subseteq \SO(G)$. Combining these results with Proposition~\ref{pr:equivalent-norm-on-s0} yields the desired conclusion.  The result for $\mathscr{P}$ is shown in a similar way. 

Let now $g,h\in \SO(G)\backslash\{0\}$ be given. Then $h*g^{\dagger}\in\SO(G)$ and Theorem~\ref{th:SO-expansion} implies that any $f\in \SO(G)$, in particular, the reflection of $f$, $f^{r}(t) = f(-t)$, can be written in the form
\begin{equation} \label{eq:Qeq} f^r(t) = \sum_{n\in\N} c_{n} E_{-\omega_{n}}T_{x_{n}} (h*g^{\dagger})(t) = \sum_{n\in\N} T_{(\omega_{n},0)} \mathcal{V}_{\hat{g}} \hat{f}_{n}(0,-t),\end{equation}
where $(c_{n})_{n}\in \ell^{1}(\N)$, $(x_{n},\omega_{n})_{n} \subseteq G\times\ghat$ and $f_{n}=c_{n} E_{-\omega_{n}} T_{x_{n}}h$. This implies that every function $f\in \SO(G)$ can be written as $f = R_{\{0\}\times G} ( \sum_{n\in\N} T_{(\omega_{n},0}) \mathcal{V}_{\hat{g}}\hat{f}_{n} )$, where $f_{n}$ is as above. 
Furthermore, the following inequality holds:
\[ \Vert f \Vert_{\mathscr{Q},g} = \inf\Big\{ \sum_{n} \Vert f_{n} \Vert_{\SO,g} \, : f = \sum_{n} \ldots \Big\} \le \sum_{n} \Vert c_{n} E_{-\omega_{n}} T_{x_{n}}h \Vert_{\SO,g} = \sum_{n} \vert c_{n} \vert \, \Vert h \Vert_{\SO,g}.\]
This estimate is independent of the representation in \eqref{eq:Qeq}, thus
\[ \Vert f \Vert_{\mathscr{Q},g} \le \Vert f \Vert_{\mathscr{L},h*g^{\dagger}} \, \Vert h \Vert_{\SO,g}.\]
On the other hand, for all $f\in \mathscr{Q}$ have that
\begin{align*}
 \Vert f \Vert_{\SO,g} & = \Vert R_{\{0\}\times G} \sum_{n} T_{\omega_{n}} \mathcal{V}_{\hat{g}}\hat{f}_{n} \Vert_{\SO,g} \le C \, \sum_{n} \Vert \mathcal{V}_{\hat{g}}\hat{f}_{n} \Vert_{\SO(\ghat\times G),\mathcal{V}_{\hat{g}}\hat{g}} \\
 & = C \sum_{n} \Vert \hat{g} \Vert_{\SO(\ghat),\hat{g}} \, \Vert \hat{f}_{n} \Vert_{\SO(\ghat),\hat{g}} = C \Vert g \Vert_{\SO,g} \, \sum_{n} \Vert f_{n} \Vert_{\SO,g},
\end{align*}
where $C$ is the operator norm of the restriction operator $R_{\{0\}\times G}$ from $\SO(G\times\ghat)$ with the norm $\Vert \cdot \Vert_{\SO,\mathcal{V}_{\hat{g}}\hat{g}}$ onto $\SO(G)$ with the norm $\Vert \cdot \Vert_{\SO,g}$.
Since the above inequality holds for any representation of $f\in \mathscr{Q}$ we conclude that
\[ \Vert f \Vert_{\SO,g} \le C \Vert g \Vert_{\SO,g} \, \Vert f \Vert_{\mathscr{Q},g}.\]
The equality $\SO(G)=\mathscr{R}$ and the equivalence of their norms is shown in a similar fashion as the previous results.
\end{proof}

\begin{corollary}
For all functions $f\in \SO(G)$ the following inequalities hold.
\begin{enumerate}
\item[(i)] $ \ \ c \, \Vert f \Vert_{\SO,g} \le \Vert f \Vert_{\mathscr{M},K} \le C \, \Vert f \Vert_{\mathscr{L},g} \, $, with $c = \Vert h \Vert_{\SO,g}^{-1} $, $C=\Vert g \Vert_{A(G)}$ and where $h$ is a function in $\SO(G)$ such that $h= 1$ on the set $K\subseteq G$.
\item[(ii)] $ \ \  c \, \Vert f \Vert_{\SO,g} \le \Vert f \Vert_{\mathscr{N},\tilde{K}} \le C \, \Vert f \Vert_{\mathscr{L},g} \,$,
with $c = \Vert h \Vert_{\SO,g}^{-1} $, $C=\Vert g \Vert_{1}$ and where $h$ is a function in $\SO(G)$ such that $\hat{h}=1$ on the set $\tilde{K}\subseteq \ghat$.
\item[(iii)] $ \ \  c \, \Vert f \Vert_{\SO,g} \le \Vert f \Vert_{\mathscr{O},g} \le C \, \Vert f \Vert_{\mathscr{L},h*g} \, $, with $c = \Vert g \Vert_{\SO,g}^{-1}$, $C = \Vert h \Vert_{1}$ and where $h$ is any function in $L^{1}(G)\backslash\{0\}$.
\item[(iv)] $ \ \ c \, \Vert f \Vert_{\SO,g} \le \Vert f \Vert_{\mathscr{P},g} \le C \, \Vert f \Vert_{\mathscr{L},h\cdot g} \, $,
with $c = \Vert g \Vert_{\SO,g}^{-1}$, $C = \Vert h \Vert_{A(G)}$ and where $h$ is any function in $A(G)\backslash\{0\}$.
\item[(v)] $ \ \ c \, \Vert f \Vert_{\SO,g} \le \Vert f \Vert_{\mathscr{Q},g} \le C \, \Vert f \Vert_{\mathscr{L},h*g^{\dagger}} \, $,
where $c =  \Vert R_{\{0\}\times G } \Vert^{-1}_{\textnormal{op},\SO,\mathcal{V}_{\hat{g}}\hat{g}\to \SO,g} \Vert g \Vert_{\SO,g}^{-1}$, $C = \Vert h \Vert_{\SO,g}$ and $h$ is any function in $\SO(G)\backslash\{0\}$.
\item[(vi)] $ \ \ c \, \Vert f \Vert_{\SO,g} \le \Vert f \Vert_{\mathscr{R},g} \le C \, \Vert f \Vert_{\mathscr{L},h_{1}*h_{2}} \, $,
where $c= \Vert g \Vert_{\infty}^{1}$, $C = \Vert h_{1} \Vert_{\SO,g} \,\Vert h_{2} \Vert_{\SO,g}$ and $h_{1},h_{2}\in \SO(G)\backslash\{0\}$.
\end{enumerate}
\end{corollary}

Let us turn to the characterization of the Feichtinger algebra given by the set $\mathscr{T}$. Recall that a family of functions $(\psi_{i})_{i\in I}\subseteq A(G)$ is a \emph{bounded uniform partition of unity of $G$} if there exists a compact set $W\subseteq G$ and a discrete subset $(x_{i})\subseteq G$ such that
\begin{enumerate}[]
\item [(a.i)] $\sum_{i\in I} \psi_{i} (x) = 1 $ for all $x\in G$,
\item [(a.ii)]$\sup_{i\in I} \Vert \psi_{i} \Vert_{A(G)} < \infty$,
\item [(a.iii)]$\supp \psi_{i} \subseteq x_{i} + W$ for all $i\in I$,
\item [(a.iv)]$\sup_{x\in G} \# \{ i\in I \, : \, (x +K) \cap (x_{i}+ W) \ne \emptyset\} < \infty$ for any compact set $K\subseteq G$.
\end{enumerate} 
BUPUs are easily constructed for $G=\R$ by use of triangular functions and in a similar way for in $\R^n$. For general locally compact abelian groups constructions of such BUPUs are also possible, see \cite{st79,fe81}. BUPUs are an essential part in the theory of the Wiener amalgam spaces \cite{fe81-1,MR751019}.  

If $(\psi_{i})_{i\in I}\subset A(G)$ is a bounded uniform partition of unity of $G$, then, as in Definition \ref{def:S0-bigdef}, we define
\begin{enumerate}[]
\item $\mathscr{T} = \{ f \in A(G) \, : \, \sum_{i\in I} \Vert f \psi_{i} \Vert_{A(G)} < \infty\}$.
\end{enumerate}

The space $\mathscr{T}$ is $W(A(G),L^{1})$, the Wiener amalgam space with local component in the Fourier algebra and global component in $L^{1}$.

\begin{proposition} \label{pr:bupu-time} Let $(\psi_{i})_{i\in I}\subseteq A(G)$ be a bounded uniform partition of unity of $G$ as described in (a.i)-(a.iv) above. For any locally compact abelian group $G$ it holds that $\SO(G) = \mathscr{T}$. Moreover, the norm given by
\[ \Vert f \Vert_{\mathscr{T},\psi_{i}} = \sum_{i\in I} \Vert f \psi_{i} \Vert_{A(G)}\]
 is an equivalent norm on $\SO(G)$. Specifically,
\[ \Vert f \Vert_{\mathscr{M},W} \le \Vert f \Vert_{\mathscr{T},\psi_{i}} \le c_{1} \, c_{2} \Vert f \Vert_{\mathscr{M},K}   \]
where
\[ c_{1} = \sup_{x\in G}\# \{ i \in I \, : \, (x+K) \cap (x_{i}+W)\ne \emptyset\} \ \ \text{and} \ \ c_{2} = \sup_{i\in I} \Vert \psi_{i} \Vert_{A(G)}. \]
\end{proposition}
\begin{proof} Let $f\in \mathscr{T}$. Then 
$f= \sum_{i\in I} f \psi_{i} = \sum_{i\in I} T_{x_{i}} (T_{-x_{i}} f \cdot T_{-x_{i}}\psi_{i})$.
Note that $\supp T_{-x_{i}} \psi_{i} = \supp \psi_{i} - x_{i} \subseteq W$. Hence the support of the functions $g_{i} = T_{-x_{i}}f \cdot T_{-x_{i}} \psi_{i}$ is a subset of $W$ for all $i\in I$. Moreover, we have that 
\begin{equation} \label{eq:bupu1} \sum_{i\in I} \Vert g_{i} \Vert_{A(G)} = \sum_{i\in I} \Vert f \psi_{i} \Vert_{A(G)} < \infty.\end{equation} 
By the characterization of $\SO(G)$ via the set $\mathscr{M}$, we conclude that $f \in \SO(G)$. Hence, $\mathscr{T} \subseteq \SO(G)$. Moreover, the calculation in \eqref{eq:bupu1} implies that $\Vert f \Vert_{\mathscr{M},W} \le \Vert f \Vert_{\mathscr{T},\psi_{i}}$. 

Conversely assume now that $f\in \SO(G)$. By the characterization of $\SO(G)$ by the set $\mathscr{M}$, we know that $f= \sum_{n\in \N} T_{x_{n}} g_{n}$ with $(x_{n})_{n\in\N}\subseteq G$, $(g_{n})_{n\in\N}\subseteq A(G)$, $\supp \, g_{n}$ in a compact set $K$ and $\sum_{n\in\N} \Vert g_{n} \Vert_{A(G)} < \infty$.
We can now make the following estimates:
\begin{align*}
 & \sum_{i\in I} \Vert f \psi_{i} \Vert_{A(G)} \le \sum_{n\in\N} \sum_{i\in I} \Vert (T_{x_{n}}g_{n}) \psi_{i} \Vert_{A(G)} \\
 & = \sum_{n\in\N} \sum_{\substack{i\in I\\ (x_{n}+K) \cap (x_{i}+W)\ne \emptyset}} \Vert (T_{x_{n}}g_{n}) \psi_{i} \Vert_{A(G)} \le \sum_{n\in\N} \sum_{\substack{i\in I\\ (x_{n}+K) \cap (x_{i}+W)\ne \emptyset}} \Vert g_{n} \Vert_{A(G)} \, \Vert \psi_{i} \Vert_{A(G)}.
\end{align*}
With the help of the constant $c_{1}$ and $c_{2}$ we find that
\[ \sum_{i\in I} \Vert f \psi_{i} \Vert_{A(G)} \le c_{1} c_{2} \sum_{n\in \N} \Vert g_{n} \Vert_{A(G)} < \infty.\]
It follows that
\[ \Vert f \Vert_{\mathscr{T},\psi_{i}} \le c_{1} \, c_{2} \Vert f \Vert_{\mathscr{M},K} \ \ \text{for all} \ \ f\in \SO(G).\]
Since $f\in \SO(G)$ was arbitrary, we have shown that the norm $\Vert f \Vert_{\mathscr{T},\psi_{i}}$ is equivalent to $\Vert \cdot \Vert_{\mathscr{M},K}$. 
\end{proof}

Similarly, let $(\varphi_{i})_{i\in I}\subseteq L^{1}(G)$ be a \emph{bounded uniform partition of unity of $\ghat$} as described in 
Definition~\ref{def:S0-bigdef}, and consider the set
\begin{enumerate}[]
\item $\mathscr{U} = \{ f \in L^{1}(G) \, : \, \sum_{i\in I} \Vert f *\varphi_{i} \Vert_{1} < \infty\}$.
\end{enumerate}

\begin{proposition} For any locally compact abelian group $G$ it holds that $\SO(G) = \mathscr{U}$. Moreover, the norm given by
\[ \Vert f \Vert_{\mathscr{U},\varphi_{i}} = \sum_{i\in I} \Vert f * \varphi_{i} \Vert_{1}\]
 is an equivalent norm on $\SO(G)$. Specifically,
\[ \Vert f \Vert_{\mathscr{N},V} \le \Vert f \Vert_{\mathscr{U},\varphi_{i}} \le c_{1} \, c_{2} \Vert f \Vert_{\mathscr{N},\tilde{K}} \]
with
\[ c_{1} = \sup_{\omega\in \ghat}\# \{ i \in I \, : \, (\omega+\tilde{K}) \cap (\omega_{i}+V)\ne \emptyset\} \ \ \text{and} \ \ c_{2} = \sup_{i\in I} \Vert \varphi_{i} \Vert_{1}. \]
\end{proposition}
\begin{proof} The proof is very similar to the one for 
Proposition~\ref{pr:bupu-time} and is therefore omitted.
\end{proof}

\section{The Kernel Theorem for the Feichtinger algebra}
\label{sec:kernel}
In this section we show that the Banach space $\SO$ allows for the formulation analogue to the classical kernel theorem due to Schwartz for the space of test functions $C_{c}^{\infty}(\R^n)$ \cite{MR0045307}. 
In fact, the already established tensor factorization property of $\SO$, 
\[ \SO(G_{1})\widehat{\otimes}\SO(G_{2}) = \SO(G_{1}\times G_{2})\]
for all locally compact abelian groups $G_{1}$ and $G_{2}$ (Theorem \ref{th:tensor-property-s0}) together with results on tensor products of Banach spaces \cite{MR1888309} imply the kernel theorem, which is what we describe below.

The kernel theorem is a central result in functional analysis and in the theory of generalized functions with applications in the theory of pseudo-differential operators.

We begin by establishing a one-to-one correspondence between all bilinear and bounded operators from $\SO(G_{1})\times \SO(G_{2})$ into a normed vector space $V$, and all linear and bounded operators from $\SO(G_{1}\times G_{2})$ into $V$.

\begin{lemma} \label{le:SO-tensor-operator} Let $G_{1}$ and $G_{2}$ be two locally compact abelian groups and $V$ a normed vector space. For every bilinear and bounded operator $A: \SO(G_{1})\times \SO(G_{2})\to V$ there exists a unique linear and bounded operator $T: \SO(G_{1}\times G_{2})\to V$ satisfying 
\[ A(f_{1},f_{2})=T(f_{1}\otimes f_{2}) \ \ \text{for all} \ \ f_{i}\in \SO(G_{i}), \ i=1,2.\] The correspondence $A \longleftrightarrow T$ is an isomorphism between the normed vector spaces $\textnormal{Bil}(\SO(G_{1})\times\SO(G_{2}),V)$ and $\textnormal{Lin}(\SO(G_{1}\times G_{2}),V)$.
\end{lemma}
\begin{proof}

Consider the operators
\begin{align*} & e:\textnormal{Lin}(\SO(G_{1}\times G_{2}),V) \to \textnormal{Bil}(\SO(G_{1})\times\SO(G_{2}),V), \ e(T)= (f_{1}, f_{2}) \mapsto T(f_{1}\otimes f_{2}), \\
& e^{-1}: \textnormal{Bil}(\SO(G_{1})\times\SO(G_{2}),V)\to \textnormal{Lin}(\SO(G_{1}\times G_{2}),V), \ e^{-1}(A) = \Big( F \mapsto \sum_{j\in\N} A(f_{1,j},f_{2,j}) \Big),
\end{align*}
where $F = \sum_{j\in\N} f_{1,j}\otimes f_{2,j}$ for some $\{f_{i,j}\}_{j\in\N}\subset \SO(G_{i})$, $i=1,2$.
It is a straight forward to show that $e$ and $e^{-1}$ are linear and bounded operators.
Note that the value of $e^{-1}(A)(F)$, $A\in\textnormal{Bil}(\SO(G_{1})\times\SO(G_{2}),V)$, $F\in  \SO(G_{1}\times G_{2})$ may depend on the used representation of $F$. However, one can easily verify that $e$ and $e^{-1}$ are inverses of one another. Since the inverse is unique the operator $e^{-1}(A)$ is a unique and consequently $e^{-1}(A)(F)$ does not depend on the particular representation of the function $F$ by its tensor-factorization.\end{proof}

In Lemma~\ref{le:SO-tensor-operator}, if the Banach space $\SO(G_{1}\times G_{2})$ is equipped with the projective tensor norm $\Vert \cdot \Vert_{\SO(G_{1})\hat{\otimes}\,\SO(G_{2})}$ instead of the usual norm $\Vert \cdot \Vert_{\SO(G_{1}\times G_{2})}$, then the mapping $e$ in the proof of Lemma~\ref{le:SO-tensor-operator} establishes an \emph{isometric} isomorphism between the normed vector spaces $\textnormal{Bil}(\SO(G_{1})\times\SO(G_{2}),V)$ and $\textnormal{Lin}(\SO(G_{1} \times G_{2}),V)$.

Lemma~\ref{le:SO-tensor-operator} can be generalized to show that for two Banach spaces $X$ and $Y$ and a normed vector space $V$, the
normed vector spaces $\textnormal{Bil}(X\times Y,V)$ and $\textnormal{Lin}(X\hat{\otimes} Y,V)$ are isometrically isomorphic. In fact, tensor products are defined by this property; multi-linear operators on products of spaces are in one-to-one correspondence with linear operators on the projective tensor product of the spaces. For more on this see, \eg \cite{MR1888309}.

Let us state an immediate consequence of Lemma~\ref{le:SO-tensor-operator}.
\begin{corollary} \label{cor:tensor-of-dist} Let $\sigma_{1} \in \SOprime(G_{1})$ and $\sigma_{2}\in \SOprime(G_{2})$ be given. There exists a unique element $\sigma_{1}\otimes \sigma_{2}\in \SOprime(G_{1}\times G_{2})$, which we call \emph{the tensor product} of $\sigma_{1}$ and $\sigma_{2}$, such that
\begin{equation} \label{eq:2506g} ( f_{1}\otimes f_{2}, \sigma_{1}\otimes \sigma_{2})_{\SO,\SOprime(G_{1}\times G_{2})} = (f_{1}, \sigma_{1})_{\SO,\SOprime(G_{1})} \, ( f_{2}, \sigma_{2})_{\SO,\SOprime(G_{2})} \end{equation}
for all $f_{1}\in \SO(G_{1})$ and $f_{2}\in \SO(G_{2})$. Moreover, for any $g_{i}\in\SO(G_{i})\backslash\{0\}$, $i=1,2$,
\[ \Vert \sigma_{1}\otimes \sigma_{2}\Vert_{M^{\infty},g_{1}\otimes g_{2}} = \Vert \sigma_{1}\Vert_{M^{\infty}(G_{1}),g_{1}} \, \Vert \sigma_{2} \Vert_{M^{\infty}(G_{2}),g_{2}}.\]
\end{corollary}
\begin{proof} Define 
\[ A: \SO(G_{1})\times \SO(G_{2}) \to \C, \ A(f_{1},f_{2})= (f_{1}, \sigma_{1})_{\SO,\SOprime(G_{1})} \, ( f_{2}, \sigma_{2})_{\SO,\SOprime(G_{2})}.\]
It is straightforward to show that $A\in \text{Bil}(\SO(G_{1})\times \SO(G_{2}), \C)$. By Lemma~\ref{le:SO-tensor-operator} there exists a unique element in $\SOprime(G_{1}\times G_{2})$, which we denote by $\sigma_{1}\otimes \sigma_{2}$, such that \eqref{eq:2506g} holds. Concerning the moreover part, we observe that,
\begin{align*}
\Vert \sigma_{1}\otimes \sigma_{2}\Vert_{M^{\infty},g_{1}\otimes g_{2}} & = \sup_{\substack{\chi_{1}\in G_{1}\times\ghat_{1} \\ \chi_{2}\in G_{2}\times\ghat_{2} } } \vert ( \pi(\chi_{1}) g_{1} \otimes \pi(\chi_{2})g_{2},\sigma_{1}\otimes \sigma_{2} )_{\SO,\SOprime(G_{1}\times G_{2})} \vert \\
& = \sup_{\substack{\chi_{1}\in G_{1}\times\ghat_{1} \\ \chi_{2}\in G_{2}\times\ghat_{2} } } \vert ( \pi(\chi_{1}) g_{1} ,\sigma_{1})_{\SO,\SOprime(G_{1})} \, (  \pi(\chi_{2})g_{2},\sigma_{2} )_{\SO,\SOprime(G_{2})}\vert \\
& = \Vert \sigma_{1}\Vert_{M^{\infty}(G_{1}),g_{1}} \, \Vert \sigma_{2} \Vert_{M^{\infty}(G_{2}),g_{2}}.
\end{align*}
\end{proof}

We now turn to the kernel theorem of $\SO$. 
Theorem~\ref{th:SO-kernel-theorem} was stated already in \cite{fe80} and later in \cite{ho89}, albeit without a proof.  
It appears in \cite{ke03} (with $G_{1}=G_{2}$) where the reasoning is the same as here: as soon as one has established that $\SO(G_{1}\times G_{2})= \SO(G_{1})\hat{\otimes}\, \SO(G_{2})$, then Theorem~\ref{th:SO-kernel-theorem} follows by the theory of projective tensor products, i.e., Lemma~\ref{le:SO-tensor-operator}.

\begin{theorem}\label{th:SO-kernel-theorem} 
The linear and bounded (i.e., norm-norm continuous) operators $T$ from $\SO(G_{1})$ into $\SOprime(G_{2})$ are exactly those which satisfy  
\begin{equation} \label{eq:0302a}  ( f_{2}, T f_{1})_{\SO,\SOprime(G_{2})} = ( f_{1} \otimes f_{2}, \sigma )_{\SO,\SOprime(G_{1}\times G_{2})} \ \ \text{for all} \ \ f_{1}\in \SO(G_{1}),\, f_{2}\in\SO(G_{2}), \end{equation}
for some $\sigma\in \SOprime(G_{1}\times G_{2})$.
The correspondence $\sigma \longleftrightarrow T$ is an isomorphism between the Banach spaces $\SOprime(G_{1}\times G_{2})$ and $\textnormal{Lin}(\SO(G_{1}),\SOprime(G_{2}))$.
\end{theorem}
\begin{proof} 
Define the operators
\begin{align*} & d:\textnormal{Lin}(\SO(G_{1}),\SOprime(G_{2}))\to \textnormal{Bil}(\SO(G_{1})\times \SO(G_{2}),\C), \ d(T) = (f_{1},f_{2})\mapsto (f_{2},Tf_{1})_{\SO,\SOprime(G_{2})}, \\
& d^{-1}: \textnormal{Bil}(\SO(G_{1})\times \SO(G_{2}),\C)  \to \textnormal{Lin}(\SO(G_{1}),\SOprime(G_{2})), \ d^{-1}(A) = f_{1}\mapsto \big( f_{2}\mapsto A(f_{1},f_{2})\big).\end{align*}
It is a matter of routine to show that these operators are well-defined, linear, bounded and inverses of one another. This shows that $\textnormal{Bil}(\SO(G_{1})\times \SO(G_{2}),\C)\cong \textnormal{Lin}(\SO(G_{1}),\SOprime(G_{2}))$. By Lemma \ref{le:SO-tensor-operator} we know that $\textnormal{Bil}(\SO(G_{1})\times \SO(G_{2}),\C)\cong \SOprime(G_{1}\times G_{2})$. The result now follows.
\end{proof}

An alternative proof strategy of Theorem~\ref{th:SO-kernel-theorem}, not relying on Lemma~\ref{le:SO-tensor-operator}, is to show directly that $\SOprime(G_{1}\times G_{2})$ and $\textnormal{Lin}(\SO(G_{1}),\SOprime(G_{2}))$ can be identified with one another. In order to do this, one has to consider the operators
\[ d: \SOprime(G_{1}\times G_{2}) \to \textnormal{Lin}(\SO(G_{1}),\SOprime(G_{2})), \ d(\sigma) = \Big( f_{1} \mapsto \big( f_{2} \mapsto ( f_{1} \otimes f_{2}, \sigma )_{\SO,\SOprime(G_{1}\times G_{2})} \big) \Big)\]
and 
\[ d^{-1} : \textnormal{Lin}(\SO(G_{1}),\SOprime(G_{2})) \to \SOprime(G_{1}\times G_{2}), \ d^{-1}(T) = \Big( F \mapsto \sum_{j\in \N} ( f_{2,j} , T f_{1,j} )_{\SO,\SOprime(G_{2})}\Big), \]
where $F = \sum_{j\in \N} f_{1,j} \otimes f_{2,j}$ for some $(f_{i,j})_{j\in\N} \subseteq \SO(G_{i})$, $i=1,2$.
One then needs to show that $d$ and $d^{-1}$ are well-defined, linear and bounded operators and indeed are inverses of one another (which, as in Lemma \ref{le:SO-tensor-operator}, implies that the value of $d^{-1}(T)(F)$, $T\in \textnormal{Lin}(\SO(G_{1}),\SOprime(G_{2}))$, $F\in \SO(G_{1}\times G_{2})$, is not dependent on the particular tensor factorization of $F$).

Note that in Theorem~\ref{th:SO-kernel-theorem} it is possible to interchange the role of $G_{1}$ and $G_{2}$ so that 
\[ \textnormal{Bil}(\SO(G_{1})\times \SO(G_{2}), \C) \cong \SOprime(G_{1}\times G_{2}) \cong \textnormal{Lin}(\SO(G_{1}),\SOprime(G_{2})) \cong \textnormal{Lin}(\SO(G_{2}),\SOprime(G_{1})).\]

In case one has a basis for $\SO(G)$ there are proofs available where one does not need to refer to Lemma~\ref{le:SO-tensor-operator}, see \cite{feko98} (for \emph{elementary} locally compact abelian groups) and \cite{MR1843717} (for $G=\R^d$). Yet a different way of proof can be found in \cite{fe89-2,fegr92-1} (for $G=\R^d$), where a sequence of elements in $\SO$ is constructed such that it converges in the weak$^{*}$-topology towards the correct  $\sigma\in \SOprime$ such that \eqref{eq:0302a} holds (see the paragraph after Lemma \ref{le:SOprime-time-SO-conv-S0-is-S0}).

Theorem~\ref{th:SO-kernel-theorem} is the exact analogue to the famous Schwartz-kernel Theorem for the space of test functions $C_{c}^{\infty}(\R^n)$ by Schwartz \cite[Thrm.\ II]{MR0045307}. Later, more elementary proofs of this result appeared in 
\cite{MR0082637} and \cite{MR0125438}. See also the book by H\"ormander \cite{MR1996773} for more on the Schwartz space and its kernel theorem. 
It is remarkable that $\SO$ allows for a kernel theorem, as such a result is usually associated to, so-called (and hence their name), nuclear spaces, see \cite{gr55,MR0173945,tr67}.

\section*{Acknowledgements}
%Tak til Ole og Hans.

%This project started during a research stay at the Department of Mathematics at the University of Maryland in 2015, was significantly helped in Bonn 2016 and finally finished 2 years later in 2017.
  The author thanks Ole Christensen for continuous support with the writing and the presentation of the material. Furthermore, the author thanks Hans G.\ Feichtinger for many invaluable discussions on his algebra and related topics. Thanks also goes to Franz Luef, Jakob Lemvig and Jordy T.\ van Velthoven for helpful comments. 
%The author thanks the Norbert Wiener Center at the University of Maryland and the Hausdorff Research Institute of Mathematics in Bonn for their hospitality during visits in 2015 and January 2016, respectively, which accelerated .

% The author thanks both institutions for their hospitality.

%Parts of this project was  when the author visited the Norbert Wiener Center at the University of Maryland in 2015 and the Hausdorff Research Institute of Mathematics in Bonn in January 2016. The author thanks both institutions for their hospitality.
%{\tiny
{\footnotesize
% automatic bib
%\bibliographystyle{abbrv} % eller unsrt, abbrv eller alpha
%\bibliography{library,nuhagbib,mitnyebib,manual_library}

\begin{thebibliography}{100}

\bibitem[Ant98]{an98-2}
J.-P. {A}ntoine.
\newblock {Q}uantum {M}echanics beyond {H}ilbert {S}pace.
\newblock In A.~{B}ohm, H.-D. {D}oebner, and P.~{K}ileanowski, editors, {\em
  {I}rreversibility and {C}ausality, {S}emigroups and {R}igged {H}ilbert
  {S}paces., volume 504
  of {\em {L}ecture {N}otes in {P}hysics}, {B}erlin, {S}pringer {V}erlag, 1998}

\bibitem[Bal06]{ba06} 
R.~{B}alan. 
\newblock {T}he noncommutative {W}iener lemma, linear independence, and 
spectral properties of the algebra of time-frequency shift operators. 
\newblock {\em Trans. Amer. Math. Soc.}, 360(7):3921--3941, 2006. 

\bibitem[BaCa$^{+}$06a]{bacahela06}
R.~{B}alan, P.~G. {C}asazza, C.~{H}eil, and Z.~{L}andau.
\newblock {D}ensity, overcompleteness, and localization of frames {I}:
  {T}heory.
\newblock {\em J. Fourier Anal. Appl.}, 12(2):105--143, 2006.

\bibitem[BaCa$^{+}$06b]{bacahela06-1}
R.~{B}alan, P.~G. {C}asazza, C.~{H}eil, and Z.~{L}andau.
\newblock {D}ensity, overcompleteness, and localization of frames. {I}{I}:
  {G}abor systems.
\newblock {\em J. Fourier Anal. Appl.}, 12(3):307--344, 2006.

\bibitem[Ben75]{be75}
J.~J. {B}enedetto.
\newblock {\em {S}pectral {S}ynthesis.}
\newblock {A}cademic~{P}ress, {F}rancisco, 1975.

\bibitem[BeZi97]{bezi97}
J.~J. {B}enedetto and G.~{Z}immermann.
\newblock {S}ampling multipliers and the {P}oisson {S}ummation {F}ormula.
\newblock {\em J. Fourier Anal. Appl.}, 3(5):505--523, 1997.

\bibitem[B{\'e}Gr$^{+}$05]{begrheok05}
{\'A}.~{B}{\'e}nyi, K.~{G}r{\"o}chenig, C.~{H}eil, and K.~A. {O}koudjou.
\newblock {M}odulation spaces and a class of bounded multilinear
  pseudodifferential operators.
\newblock {\em J. Operator Theory}, 54(2):387--399, 2005.

\bibitem[B{\'e}Ok04]{beok04}
{\'A}.~{B}{\'e}nyi and K.~{O}koudjou.
\newblock {B}ilinear pseudodifferential operators on modulation spaces.
\newblock {\em J. Fourier Anal. Appl.}, 10(3):301--313, 2004.

\bibitem[BeDa$^{+}$78]{bedadu78}
J.-P. {B}ertrandias, C.~{D}atry, and C.~{D}upuis.
\newblock {U}nions et intersections d'espaces ${L}^p$ invariantes par
  translation ou convolution.
\newblock {\em Ann. Inst. Fourier (Grenoble)}, 28(2):53--84, 1978.


\bibitem[Ber82]{be82} 
J.-P. {B}ertrandias. 
\newblock {E}spaces $l^p(a)$ et $l^p(q)$. 
\newblock volume~{I} of {\em {G}roupe de travail d'analyse harmonique}, pages 
1--13. {U}niversit{\'e} scientifique et medicale de {G}renoble, laboratoire 
de math{\'e}matique pures associ{\'e} au c.n.r.s., 1982.

\bibitem[Ber84]{be84-1} 
J.-P. {B}ertrandias. 
\newblock {E}spaces $l^p({l^\alpha})$. 
\newblock volume~{I}{I} of {\em {G}roupe de travail d'analyse harmonique}, 
pages {I}{V}.1--{I}{V}.12. {U}niversit{\'e} scientifique et medicale de 
{G}renoble, laboratoire de math{\'e}matique pures associ{\'e} au c.n.r.s., 1984.

\bibitem[Bog04]{bo04-2}
P.~{B}oggiatto.
\newblock {L}ocalization operators with ${L}^p$ symbols on modulation spaces.
\newblock In {\em {A}dvances in {P}seudo-differential {O}perators}, volume 155
  of {\em {O}per. {T}heory {A}dv. {A}ppl.}, pages 149--163. {B}irkh{\"a}user,
  {B}asel, 2004.

\bibitem[Bon86]{bo86-1}
F.~F. {B}onsall.
\newblock {D}ecompositions of functions as sums of elementary functions.
\newblock {\em Quart. J. Math. Oxford Ser. (2)}, 37:129--136, 1986.

\bibitem[Bon91]{bo91-2} 
F.~{B}onsall. 
\newblock {A} general atomic decomposition theorem and {B}anach's closed range 
theorem. 
\newblock {\em Quart. J. Math. Oxford Ser. (2)}, 42(165):9--14, 1991. 

\bibitem[BoKa14]{MR3289046}
A.~Bowers and N.~J. Kalton.
\newblock {\em An introductory course in functional analysis}.
\newblock Universitext. Springer, New York, 2014.
\newblock With a foreword by Gilles Godefroy.

\bibitem[Bru61]{br61}
F.~{B}ruhat.
\newblock {D}istributions sur un groupe localement compact et applications
  {\`a} l etude des repr{\'e}sentations des groupes $p$-adiques.
\newblock {\em Bull. Soc. Math. France}, 89:43--75, 1961.

\bibitem[Car64]{ca64-1}
P.~{C}artier.
\newblock {\"{U}}ber einige {I}ntegralformeln in der {T}heorie der
  quadratischen {F}ormen.
\newblock {\em Math. Z.}, 84:93--100, 1964.

\bibitem[ChMa$^{+}$12]{MR2966135}
J.~G. Christensen, A.~Mayeli, and G.~{\'O}lafsson.
\newblock Coorbit description and atomic decomposition of {B}esov spaces.
\newblock {\em Numer. Funct. Anal. Optim.}, 33(7-9):847--871, 2012.

\bibitem[Chr96]{ch96-2}
O.~{C}hristensen.
\newblock {A}tomic decomposition via projective group representations.
\newblock {\em Rocky Mountain J. Math.}, 26(4):1289--1312, 1996.

\bibitem[Chr16]{ch16newbook}
O.~{C}hristensen.
\newblock {\em {A}n {I}ntroduction to {F}rames and {R}iesz {B}ases (2nd
  edition)}.
\newblock {A}pplied and {N}umerical {H}armonic {A}nalysis. {B}irkh{\"a}user,
  2016.

%\bibitem[CiLo$^{+}$79]{cilomi79}
%J.~{C}igler, V.~{L}osert, and P.~W. {M}ichor.
%\newblock {\em {B}anach {M}odules and {F}unctors on {C}ategories of {B}anach
%  {S}paces}, volume~46 of {\em {L}ect. notes on {P}ure and {A}ppl. {M}ath.}
%\newblock {D}ekker, {B}asel - {N}ew {Y}ork, 1979.

\bibitem[Civ15]{ci15}
G.~{C}ivan.
\newblock {\em {I}dentification of {O}perators on {E}lementary {L}ocally
  {C}ompact {A}belian {G}roups}.
\newblock PhD thesis, 2015.

\bibitem[CoFe$^{+}$08]{cofelu08}
E.~{C}ordero, H.~G. {F}eichtinger, and F.~{L}uef.
\newblock {B}anach {G}elfand triples for {G}abor analysis.
\newblock In {\em {P}seudo-differential {O}perators}, volume 1949 of {\em
  {L}ecture {N}otes in {M}athematics}, pages 1--33. {S}pringer, {B}erlin, 2008.

\bibitem[CoGr03]{cogr03-1}
E.~{C}ordero and K.~{G}r{\"o}chenig.
\newblock {T}ime-frequency analysis of localization operators.
\newblock {\em J. Funct. Anal.}, 205(1):107--131, 2003.

\bibitem[CoNi10]{coni10}
E.~{C}ordero and F.~{N}icola.
\newblock {P}seudodifferential operators on $l^p$, {W}iener amalgam and
  modulation spaces.
\newblock {\em Internat. Math. Res. Notices}, 2010(10):1860--1893, 2010.

\bibitem[CoTa$^{+}$13]{cotawa13}
E.~{C}ordero, A.~{T}abacco, and P.~{W}ahlberg.
\newblock {S}chr{\"o}dinger-type propagators, pseudodifferential operators and
  modulation spaces.
\newblock {\em J. Lond. Math. Soc. (2)}, 88(2):375--395, 2013.

\bibitem[Cza03]{cz03}
W.~{C}zaja.
\newblock {B}oundedness of pseudodifferential operators on modulation spaces.
\newblock {\em J. Math. Anal. Appl.}, 284(1):389--396, 2003.

\bibitem[DaFo$^{+}$08]{daforastte08}
S.~{D}ahlke, M.~{F}ornasier, H.~{R}auhut, G.~{S}teidl, and G.~{T}eschke.
\newblock {G}eneralized coorbit theory, {B}anach frames, and the relation to
  $\alpha$-modulation spaces.
\newblock {\em Proc. London Math. Soc.}, 96(2):464--506, 2008.

\bibitem[DaSt$^{+}$04]{dastte04-1}
S.~{D}ahlke, G.~{S}teidl, and G.~{T}eschke.
\newblock {W}eighted coorbit spaces and {B}anach frames on homogeneous spaces.
\newblock {\em J. Fourier Anal. Appl.}, 10(5):507--539, 2004.

\bibitem[dGo11]{go11}
M.~{d}e {G}osson.
\newblock {\em {S}ymplectic {M}ethods in {H}armonic {A}nalysis and in
  {M}athematical {P}hysics}.
\newblock {B}asel: {B}irkh{\"a}user, 2011.

\bibitem[dlM05]{de05-1}
R.~de~la {M}adrid.
\newblock {T}he role of the rigged {H}ilbert space in quantum mechanics.
\newblock {\em European J. Phys.}, 26(2):277--312, 2005.

\bibitem[D{\"o}Fe$^{+}$06]{dofegr06}
M.~{D}{\"o}rfler, H.~G. {F}eichtinger, and K.~{G}r{\"o}chenig.
\newblock {T}ime-frequency partitions for the {G}elfand triple $({S}_0,
  {L}^2,{{S}_0}')$.
\newblock {\em Math. Scand.}, 98(1):81--96, 2006.


\bibitem[Ehr56]{MR0082637}
L.~Ehrenpreis.
\newblock On the theory of kernels of {S}chwartz.
\newblock {\em Proc. Amer. Math. Soc.}, 7:713--718, 1956.


\bibitem[Fei77]{fe77-3}
H.~G. {F}eichtinger.
\newblock {A} characterization of {W}iener's algebra on locally compact groups.
\newblock {\em Arch. Math. (Basel)}, 29:136--140, 1977.

\bibitem[Fei79a]{fe79-5}
H.~G. {F}eichtinger.
\newblock {E}ine neue {S}egalalgebra mit {A}nwendungen in der {H}armonischen
  {A}nalyse.
\newblock In {\em {W}interschule 1979, {I}nternationale {A}rbeitstagung
  {\"u}ber {T}opologische {G}ruppen und {G}ruppenalgebren}, pages 23--25, 1979.

\bibitem[Fei79b]{fe79}
H.~G. {F}eichtinger.
\newblock {G}ewichtsfunktionen auf lokalkompakten {G}ruppen.
\newblock {\em {S}itzungsber.d.{\"o}sterr. {A}kad.{W}iss.}, 188:451--471, 1979.

\bibitem[Fei80]{fe80}
H.~G. {F}eichtinger.
\newblock {U}n espace de {B}anach de distributions temp{\'e}r{\'e}es sur les
  groupes localement compacts ab{\'e}liens.
\newblock {\em C. R. Acad. Sci. Paris S'er. A-B}, 290(17):791--794, 1980.

\bibitem[Fei81a]{fe81}
H.~G. {F}eichtinger.
\newblock {A} characterization of minimal homogeneous {B}anach spaces.
\newblock {\em Proc. Amer. Math. Soc.}, 81(1):55--61, 1981.

\bibitem[Fei81b]{fe81-1}
H.~G. {F}eichtinger.
\newblock {B}anach spaces of distributions of {W}iener's type and
  interpolation.
\newblock In P.~{B}utzer, S.~{N}agy, and E.~{G}{\"o}rlich, editors, {\em
  {P}roc. {C}onf. {F}unctional {A}nalysis and {A}pproximation, {O}berwolfach
  {A}ugust 1980}, number~69 in {I}nternat. {S}er. {N}umer. {M}ath., pages
  153--165. {B}irkh{\"a}user {B}oston, {B}asel, 1981.

\bibitem[Fei81c]{Feichtinger1981}
H.~G. Feichtinger.
\newblock {On a New Segal Algebra}.
\newblock {\em Monatshefte f{\"{u}}r Mathematik}, 92:269--289, 1981.

\bibitem[Fei83a]{fe83-1}
H.~G. {F}eichtinger.
\newblock {A} new family of functional spaces on the {E}uclidean n-space.
\newblock In {\em {P}roc.{C}onf. on {T}heory of {A}pproximation of
  {F}unctions}, {T}eor. {P}riblizh., 1983.

\bibitem[Fei83b]{MR751019}
H.~G. Feichtinger.
\newblock Banach convolution algebras of {W}iener type.
\newblock In {\em Functions, series, operators, {V}ol. {I}, {II} ({B}udapest,
  1980)}, volume~35 of {\em Colloq. Math. Soc. J\'anos Bolyai}, pages 509--524.
  North-Holland, Amsterdam, 1983.

\bibitem[Fei83c]{fe83-4}
H.~G. {F}eichtinger.
\newblock {M}odulation spaces on locally compact {A}belian groups.
\newblock Technical report, {J}anuary 1983.

\bibitem[Fei87a]{fe87}
H.~G. {F}eichtinger.
\newblock {B}anach spaces of distributions defined by decomposition methods.
  {I}{I}.
\newblock {\em Math. Nachr.}, 132:207--237, 1987.

\bibitem[Fei87b]{fe87-1}
H.~G. {F}eichtinger.
\newblock {M}inimal {B}anach spaces and atomic representations.
\newblock {\em Publ. Math. Debrecen}, 34(3-4):231--240, 1987.

\bibitem[Fei89a]{fe89-2}
H.~G. {F}eichtinger.
\newblock {A}n elementary approach to the generalized {F}ourier transform.
\newblock In T.~{R}assias, editor, {\em {T}opics in {M}athematical {A}nalysis},
  {S}er. {P}ure {M}ath. 11, pages 246--272. {W}orld {S}ci.{P}ub., 1989.

\bibitem[Fei89b]{fe89-1}
H.~G. {F}eichtinger.
\newblock {A}tomic characterizations of modulation spaces through {G}abor-type
  representations.
\newblock In {\em {P}roc. {C}onf. {C}onstructive {F}unction {T}heory},
  volume~19 of {\em {R}ocky {M}ountain {J}. {M}ath.}, pages 113--126, 1989.

\bibitem[Fei92]{fe92-3}
H.~G. {F}eichtinger.
\newblock {W}iener amalgams over {E}uclidean spaces and some of their
  applications.
\newblock In K.~{J}arosz, editor, {\em {F}unction {S}paces, {P}roc {C}onf,
  {E}dwardsville/{I}{L} ({U}{S}{A}) 1990}, volume 136 of {\em {L}ect. {N}otes
  {P}ure {A}ppl. {M}ath.}, pages 123--137. {N}u{H}{A}{G};{C}lassical, {M}arcel
  {D}ekker, 1992.

\bibitem[Fei02]{fe02}
H.~G. {F}eichtinger.
\newblock {S}pline-type spaces in {G}abor analysis.
\newblock In D.~X. {Z}hou, editor, {\em {W}avelet {A}nalysis: {T}wenty {Y}ears
  {D}evelopments {P}roceedings of the {I}nternational {C}onference of
  {C}omputational {H}armonic {A}nalysis, {H}ong {K}ong, {C}hina, {J}une 4--8,
  2001}, volume~1 of {\em {S}er. {A}nal.}, pages 100--122. {W}orld
  {S}ci.{P}ub., {R}iver {E}dge, {N}{J}, 2002.

\bibitem[Fei03]{fe03-1}
H.~G. {F}eichtinger.
\newblock {M}odulation spaces of locally compact {A}belian groups.
\newblock In R.~{R}adha, M.~{K}rishna, and S.~{T}hangavelu, editors, {\em
  {P}roc. {I}nternat. {C}onf. on {W}avelets and {A}pplications}, pages 1--56,
  {C}hennai, {J}anuary 2002, 2003. {N}ew {D}elhi {A}llied {P}ublishers.

\bibitem[Fei06]{fe06}
H.~G. {F}eichtinger.
\newblock {M}odulation {S}paces: {L}ooking {B}ack and {A}head.
\newblock {\em Sampl. Theory Signal Image Process.}, 5(2):109--140, 2006.

\bibitem[Fei09]{fe09}
H.~G. {F}eichtinger.
\newblock {B}anach {G}elfand triples for applications in physics and
  engineering.
\newblock volume 1146 of {\em {A}{I}{P} {C}onf. {P}roc.}, pages 189--228.
  {A}mer. {I}nst. {P}hys., 2009.

\bibitem[FeGr85]{fegr85}
H.~G. {F}eichtinger and P.~{G}r{\"o}bner.
\newblock {B}anach spaces of distributions defined by decomposition methods.
  {I}.
\newblock {\em Math. Nachr.}, 123:97--120, 1985.

\bibitem[FeGr88]{MR942257}
H.~G. Feichtinger and K.~Gr{\"o}chenig.
\newblock A unified approach to atomic decompositions via integrable group
  representations.
\newblock In {\em Function spaces and applications ({L}und, 1986)}, volume 1302
  of {\em Lecture Notes in Math.}, pages 52--73. Springer, Berlin, 1988.

\bibitem[FeGr89a]{MR1021139}
H.~G. Feichtinger and K.~Gr{\"o}chenig.
\newblock Banach spaces related to integrable group representations and their
  atomic decompositions. {I}.
\newblock {\em J. Funct. Anal.}, 86(2):307--340, 1989.

\bibitem[FeGr89b]{MR1026614}
H.~G. Feichtinger and K.~Gr{\"o}chenig.
\newblock Banach spaces related to integrable group representations and their
  atomic decompositions. {II}.
\newblock {\em Monatsh. Math.}, 108(2-3):129--148, 1989.

\bibitem[FeGr92]{fegr92-1}
H.~G. {F}eichtinger and K.~{G}r{\"o}chenig.
\newblock {G}abor wavelets and the {H}eisenberg group: {G}abor expansions and
  short time {F}ourier transform from the group theoretical point of view.
\newblock In C.~K. {C}hui, editor, {\em {W}avelets: a tutorial in theory and
  applications}, volume~2 of {\em {W}avelet {A}nal. {A}ppl.}, pages 359--397.
  {A}cademic {P}ress, {B}oston, 1992.

\bibitem[FeGr97]{fegr97}
H.~G. {F}eichtinger and K.~{G}r{\"o}chenig.
\newblock {G}abor frames and time-frequency analysis of distributions.
\newblock {\em J. Funct. Anal.}, 146(2):464--495, 1997.

\bibitem[FeGr$^{+}$92]{fegrwa92}
H.~G. {F}eichtinger, K.~{G}r{\"o}chenig, and D.~F. {W}alnut.
\newblock {W}ilson bases and modulation spaces.
\newblock {\em Math. Nachr.}, 155:7--17, 1992.

\bibitem[FeHe$^{+}$06]{fehelaleto06}
H.~G. {F}eichtinger, B.~{H}elffer, M.~P. {L}amoureux, N.~{L}erner, and
  J.~{T}oft.
\newblock {\em {P}seudo-{D}ifferential {O}perators}, volume 1949 of {\em
  {L}ecture {N}otes in {M}athematics}.
\newblock {S}pringer, {B}erlin, 2006.

\bibitem[FeH{\"o}90]{feho90}
H.~G. {F}eichtinger and W.~{H}{\"o}rmann.
\newblock {H}armonic analysis of generalized stochastic processes on locally
  compact abelian groups, 1990.

\bibitem[FeH{\"o}14]{feho14}
H.~G. {F}eichtinger and W.~{H}{\"o}rmann.
\newblock {A} distributional approach to generalized stochastic processes on
  locally compact abelian groups.
\newblock In {\em {N}ew perspectives on approximation and sampling theory.
  {F}estschrift in honor of {P}aul {B}utzer's 85th birthday}, pages 423--446.
  {C}ham: {B}irkh{\"a}user/{S}pringer, 2014.

\bibitem[FeKa04]{feka04}
H.~G. {F}eichtinger and N.~{K}aiblinger.
\newblock {V}arying the time-frequency lattice of {G}abor frames.
\newblock {\em Trans. Amer. Math. Soc.}, 356(5):2001--2023, 2004.

\bibitem[FeKa07]{feka07}
H.~G. {F}eichtinger and N.~{K}aiblinger.
\newblock {Q}uasi-interpolation in the {F}ourier algebra.
\newblock {\em J. Approx. Theory}, 144(1):103--118, 2007.

\bibitem[FeKo98]{feko98}
H.~G. {F}eichtinger and W.~{K}ozek.
\newblock {Q}uantization of {T}{F} lattice-invariant operators on elementary
  {L}{C}{A} groups.
\newblock In H.~G. {F}eichtinger and T.~{S}trohmer, editors, {\em {G}abor
  analysis and algorithms}, {A}ppl. {N}umer. {H}armon. {A}nal., pages 233--266.
  {B}irkh{\"a}user {B}oston, {B}oston, {M}{A}, 1998.

\bibitem[FeLu06]{felu06}
H.~G. {F}eichtinger and F.~{L}uef.
\newblock {W}iener amalgam spaces for the {F}undamental {I}dentity of {G}abor
  {A}nalysis.
\newblock {\em Collect. Math.}, 57({E}xtra {V}olume (2006)):233--253, 2006.

\bibitem[FeLu$^{+}$07]{feluwe07}
H.~G. {F}eichtinger, F.~{L}uef, and T.~{W}erther.
\newblock {A} guided tour from linear algebra to the foundations of {G}abor
  analysis.
\newblock In {\em {G}abor and {W}avelet {F}rames}, volume~10 of {\em {L}ect.
  {N}otes {S}er. {I}nst. {M}ath. {S}ci. {N}atl. {U}niv. {S}ingap.}, pages
  1--49. {W}orld {S}ci. {P}ubl., {H}ackensack, 2007.

\bibitem[FeWe06a]{fewe06}
H.~G. {F}eichtinger and F.~{W}eisz.
\newblock {T}he {S}egal algebra ${S}_0({R}^d)$ and norm summability of
  {F}ourier series and {F}ourier transforms.
\newblock {\em Monatsh. Math.}, 148:333--349, 2006.

\bibitem[FeWe06b]{fewe06-1}
H.~G. {F}eichtinger and F.~{W}eisz.
\newblock {W}iener amalgams and pointwise summability of {F}ourier transforms
  and {F}ourier series.
\newblock {\em Math. Proc. Cambridge Philos. Soc.}, 140(3):509--536, 2006.

\bibitem[FeZi98]{fezi98}
H.~G. {F}eichtinger and G.~{Z}immermann.
\newblock {A} {B}anach space of test functions for {G}abor analysis.
\newblock In H.~G. {F}eichtinger and T.~{S}trohmer, editors, {\em {G}abor
  {A}nalysis and {A}lgorithms: {T}heory and {A}pplications}, {A}pplied and
  {N}umerical {H}armonic {A}nalysis, pages 123--170, {B}oston, {M}{A}, 1998.
  {B}irkh{\"a}user {B}oston.

\bibitem[Fol89]{fo89}
G.~B. {F}olland.
\newblock {\em {H}armonic {A}nalysis in {P}hase {S}pace}.
\newblock {P}rinceton {U}niversity {P}ress, {P}rinceton, {N}.{J}., 1989.

\bibitem[Fol95]{fo95}
G.~B. {F}olland.
\newblock {\em {A} {C}ourse in {A}bstract {H}armonic {A}nalysis.}
\newblock {S}tudies in {A}dvanced {M}athematics. {B}oca {R}aton, {F}{L}:
  {C}{R}{C} {P}ress. viii, {B}oca {R}aton, {F}{L}, 1995.

\bibitem[Fol06]{fo06-1}
G.~B. {F}olland.
\newblock {T}he abstruse meets the applicable: some aspects of time-frequency
  analysis.
\newblock {\em Proc. Indian Acad. Sci. Math. Sci.}, 116:121--136, 2006.

\bibitem[FoSp$^{+}$07]{fospwo07} 
B.~E. {F}orrest, N.~{S}pronk, and P.~{W}ood. 
\newblock {O}perator {S}egal algebras in {F}ourier algebras. 
\newblock {\em Studia Math.}, 179(3):277--295, 2007. 

\bibitem[Fri98]{MR1721032}
F.~G. Friedlander.
\newblock {\em Introduction to the theory of distributions}.
\newblock Cambridge University Press, Cambridge, second edition, 1998.
\newblock With additional material by M. Joshi.

\bibitem[F{\"u}h15]{fu15}
H.~{F}{\"u}hr.
\newblock {C}oorbit spaces and wavelet coefficient decay over general dilation
  groups.
\newblock {\em Trans. Amer. Math. Soc.}, 367(10):7373--7401, 2015.

\bibitem[F{\"u}Vo15]{fuvo15}
H.~{F}{\"u}hr and F.~{V}oigtlaender.
\newblock {W}avelet coorbit spaces viewed as decomposition spaces.
\newblock {\em J. Funct. Anal.}, (1):80--154, 2015.

\bibitem[Gas60]{MR0125438}
H.~Gask.
\newblock A proof of {S}chwartz's kernel theorem.
\newblock {\em Math. Scand.}, 8:327--332, 1960.


\bibitem[GeVi64]{MR0173945}
I.~M. Gel'fand and N.~Y. Vilenkin.
\newblock {\em Generalized functions. {V}ol. 4: {A}pplications of harmonic
  analysis}.
\newblock Translated by Amiel Feinstein. Academic Press, New York - London,
  1964, 1964.

\bibitem[Gro55]{gr55}
A.~{G}rothendieck.
\newblock {P}roduits tensoriels topologiques et espaces nucl{\'e}aires.
\newblock {\em Mem. Amer. Math. Soc.}, 1955(16):140, 1955.

\bibitem[Gr\"o91]{gr91}
K.~{G}r{\"o}chenig.
\newblock {D}escribing functions: atomic decompositions versus frames.
\newblock {\em Monatsh. Math.}, 112(3):1--41, 1991.

\bibitem[Gr\"o96]{gr96}
K.~{G}r{\"o}chenig.
\newblock {A}n uncertainty principle related to the {P}oisson summation
  formula.
\newblock {\em Studia Math.}, 121(1):87--104, 1996.

\bibitem[Gr\"o98]{gr98}
K.~{G}r{\"o}chenig.
\newblock {A}spects of {G}abor analysis on locally compact abelian groups.
\newblock In H.~G. {F}eichtinger and T.~{S}trohmer, editors, {\em {G}abor
  {A}nalysis and {A}lgorithms: {T}heory and {A}pplications}, pages 211--231.
  {B}irkh{\"a}user {B}oston, {B}oston, {M}{A}, 1998.

\bibitem[Gr\"o01]{MR1843717}
K.~Gr{\"o}chenig.
\newblock {\em Foundations of time-frequency analysis}.
\newblock Applied and Numerical Harmonic Analysis. Birkh\"auser Boston, Inc.,
  Boston, MA, 2001.

\bibitem[Gr\"o03]{gr03-2}
K.~{G}r{\"o}chenig.
\newblock {U}ncertainty principles for time-frequency representations.
\newblock In H.~G. {F}eichtinger and T.~{S}trohmer, editors, {\em {A}dvances in
  {G}abor {A}nalysis}, {A}ppl. {N}umer. {H}armon. {A}nal., pages 11--30.
  {B}irkh{\"a}user {B}oston, {B}oston, 2003.

\bibitem[Gr\"o06a]{gr06-6}
K.~{G}r{\"o}chenig.
\newblock {A} pedestrian's approach to pseudodifferential operators.
\newblock In C.~{H}eil, editor, {\em {H}armonic {A}nalysis and {A}pplications},
  volume in {H}onor of {J}ohn {J}. {B}enedetto's 65th {B}irthday of {\em
  {A}ppl. {N}umer. {H}armon. {A}nal.}, pages 139--169. {B}irkh{\"a}user
  {B}oston, {B}oston, {M}{A}, 2006.

\bibitem[Gr\"o06b]{gr06-3}
K.~{G}r{\"o}chenig.
\newblock {C}omposition and spectral invariance of pseudodifferential operators
  on modulation spaces.
\newblock {\em J. Anal. Math.}, 98:65--82, 2006.

\bibitem[Gr\"o07a]{gr07-2}
K.~{G}r{\"o}chenig.
\newblock {G}abor frames without inequalities.
\newblock {\em {I}nt. {M}ath. {R}es. {N}ot. {I}{M}{R}{N}}, (23):{A}rt. {I}{D}
  rnm111, 21, 2007.

\bibitem[Gr\"o07b]{gr07-3}
K.~{G}r{\"o}chenig.
\newblock {W}iener's {L}emma: {T}heme and variations.
\newblock {\em {S}hort course at summer school on '{H}armonic {A}nalysis,
  {W}avelets, and {I}mage {P}rocessing'}, {S}eptember 2007.

\bibitem[Gr\"o07c]{gr07} 
K.~{G}r{\"o}chenig. 
\newblock {W}eight functions in time-frequency analysis. 
\newblock In L.~{R}odino and et~al., editors, {\em {P}seudodifferential 
{O}perators: {P}artial {D}ifferential {E}quations and {T}ime-{F}requency 
{A}nalysis}, volume~52 of {\em {F}ields {I}nst. {C}ommun.}, pages 343--366. 
{A}mer. {M}ath. {S}oc., {P}rovidence, {R}{I}, 2007. 

\bibitem[Gr\"o14]{gr14}
K.~{G}r{\"o}chenig.
\newblock {T}he mystery of {G}abor frames.
\newblock {\em J. Fourier Anal. Appl.}, 20(4):865--895, 2014.

\bibitem[GrHe99]{grhe99}
K.~{G}r{\"o}chenig and C.~{H}eil.
\newblock {M}odulation spaces and pseudodifferential operators.
\newblock {\em {I}ntegr. {E}qu. {O}per. {T}heory}, 34(4):439--457, 1999.

\bibitem[GrHe03]{grhe03}
K.~{G}r{\"o}chenig and C.~{H}eil.
\newblock {M}odulation spaces as symbol classes for pseudodifferential
  operators.
\newblock In R.~R. {M}.~{K}rishna, editor, {\em {P}roceedings of
  {I}nternational conference on wavelets and applications 2002}, pages
  151--170, {C}hennai,{I}ndia, 2003. {A}llied {P}ublishers, {C}hennai.

\bibitem[GrLe04]{grle04}
K.~{G}r{\"o}chenig and M.~{L}einert.
\newblock {W}iener's lemma for twisted convolution and {G}abor frames.
\newblock {\em J. Amer. Math. Soc.}, 17:1--18, 2004.

\bibitem[GrOr$^{+}$15]{grorro15}
K.~{G}r{\"o}chenig, J.~{O}rtega {C}erd{\`a}, and J.~L. {R}omero.
\newblock {D}eformation of {G}abor systems.
\newblock {\em Adv. Math.}, 277(4):388--425, 2015.

\bibitem[GrSt07]{grst07}
K.~{G}r{\"o}chenig and T.~{S}trohmer.
\newblock {P}seudodifferential operators on locally compact abelian groups and
  {S}j{\"o}strand's symbol class.
\newblock {\em J. Reine Angew. Math.}, 613:121--146, 2007.

\bibitem[HaNi98]{hani98}
V.\,~P.~{H}avin and N.K.~Nikol'skij.
\newblock {C}ommutative {H}armonic {A}nalysis {I}{I}. {G}roup {M}ethods in {C}ommutative {H}armonic {A}nalysis.
\newblock Springer, 1998.

\bibitem[Hei03]{he03}
C.~{H}eil.
\newblock {A}n introduction to weighted {W}iener amalgams.
\newblock In M.~{K}rishna, R.~{R}adha, and S.~{T}hangavelu, editors, {\em
  {W}avelets and their {A}pplications ({C}hennai, {J}anuary 2002)}, pages
  183--216. {A}llied {P}ublishers, {N}ew {D}elhi, 2003.

\bibitem[Hei07]{he07}
C.~{H}eil.
\newblock {H}istory and evolution of the density theorem for {G}abor frames.
\newblock {\em J. Fourier Anal. Appl.}, 13(2):113--166, 2007.

\bibitem[Hei11]{he11}
C.~{H}eil.
\newblock {\em {A} {B}asis {T}heory {P}rimer. {E}xpanded ed.}
\newblock {A}pplied and {N}umerical {H}armonic {A}nalysis. {B}asel:
  {B}irkh{\"a}user, 2011.

\bibitem[HeRo63]{hero63}
E.~{H}ewitt and K.~A. {R}oss.
\newblock {\em {A}bstract {H}armonic {A}nalysis {I}}.
\newblock Number 115 in {G}rundlehren {M}ath. {W}iss. {S}pringer, {B}erlin,
  1963.

\bibitem[HeRo70]{hero70}
E.~{H}ewitt and K.~A. {R}oss.
\newblock {\em {A}bstract {H}armonic {A}nalysis. {V}ol. {I}{I}: {S}tructure and
  {A}nalysis for {C}ompact {G}roups. {A}nalysis on {L}ocally {C}ompact
  {A}belian {G}roups}.
\newblock {S}pringer, {B}erlin, {H}eidelberg, {N}ew {Y}ork, 1970.

\bibitem[Hey14]{hey14}
H.~{H}eyer.
\newblock {R}andom fields and hypergroups.
\newblock In M.~M. {R}ao, editor, {\em {R}eal and {S}tochastic {A}nalysis,
  {C}urrent {T}rends}, pages 85--182. {W}orld {S}ci. {P}ubl., {H}ackensack,
  2014.

\bibitem[HoLa01]{hola01}
J.~A. {H}ogan and J.~D. {L}akey.
\newblock {E}mbeddings and {U}ncertainty {P}rinciples for {G}eneralized
  {M}odulation {S}paces.
\newblock In {\em {M}odern {S}ampling {T}heory}, {A}ppl. {N}umer. {H}armon.
  {A}nal., pages 73--105. {B}irkh{\"a}user {B}oston, {B}oston, {M}{A}, 2001.

\bibitem[HoLa05]{hola05}
J.~A. {H}ogan and J.~D. {L}akey.
\newblock {\em {T}ime-{F}requency and {T}ime-{S}cale {M}ethods. {A}daptive
  {D}ecompositions, {U}ncertainty {P}rinciples, and {S}ampling.}
\newblock {B}irkh{\"a}user, {B}oston, 2005.

\bibitem[H\"or90]{MR1996773}
L.~H{\"o}rmander.
\newblock {\em The analysis of linear partial differential operators. {I}}.
\newblock Classics in Mathematics. Springer-Verlag, Berlin, 2003.
\newblock Distribution theory and Fourier analysis, Reprint of the second
  (1990) edition [Springer, Berlin; MR1065993 (91m:35001a)].

\bibitem[H\"or89]{ho89}
W.~{H}{\"o}rmann.
\newblock {\em {G}eneralized {S}tochastic {P}rocesses and {W}igner
  {D}istribution}.
\newblock PhD thesis, {U}niversity of {V}ienna, ({A}{U}{S}{T}{R}{I}{A}), 1989.

\bibitem[Ito54]{it54}
K.~{I}to.
\newblock {S}tationary random distributions.
\newblock {\em {M}em. {C}oll. {S}ci. {U}niv. {K}yoto, {S}er. {A}}, 28:209--223,
  1954.

\bibitem[Jan95]{ja95} 
A.~J. E.~M. {J}anssen. 
\newblock {D}uality and biorthogonality for {W}eyl-{H}eisenberg frames. 
\newblock {\em J. Fourier Anal. Appl.}, 1(4):403--436, 1995. 

\bibitem[Jan05]{ja05}
A.~J. E.~M. {J}anssen.
\newblock {H}ermite function description of {F}eichtinger's space ${S}_0$.
\newblock {\em J. Fourier Anal. Appl.}, 11(5):577--588, 2005.

\bibitem[Jan06]{ja06-2}
A.~{J}anssen.
\newblock {Z}ak transform characterization of $s_0$.
\newblock {\em Sampl. Theory Signal Image Process.}, 5(2):141--162, 2006.

\bibitem[KaLe94]{kale94}
J.-P. {K}ahane and P.-G. {L}emarie {R}ieusset.
\newblock {R}emarks on the {P}oisson summation formula. ({R}emarques sur la
  formule sommatoire de {P}oisson.).
\newblock {\em Studia Math.}, 109:303--316, 1994.

\bibitem[Kai05]{ka05}
N.~{K}aiblinger.
\newblock {A}pproximation of the {F}ourier transform and the dual {G}abor
  window.
\newblock {\em J. Fourier Anal. Appl.}, 11(1):25--42, 2005.

\bibitem[Kat04]{MR2039503}
Y.~Katznelson.
\newblock {\em An introduction to harmonic analysis}.
\newblock Cambridge Mathematical Library. Cambridge University Press,
  Cambridge, third edition, 2004.

\bibitem[Kev03]{ke03}
B.~{K}eville.
\newblock {\em {M}ultidimensional {S}econd {O}rder {G}eneralised {S}tochastic
  {P}rocesses on {L}ocally {C}ompact {A}belian {G}roups}.
\newblock PhD thesis, {T}rinity {C}ollege {D}ublin, 2003.

\bibitem[La01]{la01-1}
D.~{L}abate.
\newblock {P}seudodifferential operators on modulation spaces.
\newblock {\em J. Math. Anal. Appl.}, 262(1):242--255, 2001.

\bibitem[Lie90]{li90-1}
E.~H. {L}ieb.
\newblock {I}ntegral bounds for radar ambiguity functions and {W}igner
  distributions.
\newblock {\em J. Math. Phys.}, 31(3):594--599, 1990.

\bibitem[Los80]{lo80}
V.~{L}osert.
\newblock {A} characterization of the minimal strongly character invariant
  {S}egal algebra.
\newblock {\em Ann. Inst. Fourier (Grenoble)}, 30:129--139, 1980.

\bibitem[Lue07]{lu07-2}
F.~{L}uef.
\newblock {G}abor analysis, noncommutative tori and {F}eichtinger's algebra.
\newblock In {\em {G}abor and wavelet frames}, volume~10 of {\em {L}ect.
  {N}otes {S}er. {I}nst. {M}ath. {S}ci. {N}atl. {U}niv. {S}ingap.}, pages
  77--106. {W}orld {S}ci. {P}ubl., {H}ackensack, 2007.

\bibitem[Lue09]{lu09}
F.~{L}uef.
\newblock {P}rojective modules over non-commutative tori are multi-window
  {G}abor frames for modulation spaces.
\newblock {\em J. Funct. Anal.}, 257(6):1921--1946, 2009.

\bibitem[May87]{ma87} 
M.~{M}ayer. 
\newblock {E}ine {E}inf{\"u}hrung in die verallgemeinerte 
{F}ouriertransformation. 
\newblock Diplomarbeit, {U}niversity of {V}ienna, 1987. 

\bibitem[Meg98]{me98} 
R.~{M}egginson. 
\newblock {A}n introduction to {B}anach space theory. 
\newblock Springer-Verlag, New York, Graduate Texts in Mathematics, Vol.183 (1998) p.xx+596. 


\bibitem[Oko04]{ok04}
K.~A. {O}koudjou.
\newblock {E}mbedding of some classical {B}anach spaces into modulation spaces.
\newblock {\em Proc. Amer. Math. Soc.}, 132(6):1639--1647 (electronic), 2004.

\bibitem[Osb75]{os75}
M.~S. {O}sborne.
\newblock {O}n the {S}chwartz-{B}ruhat space and the {P}aley-{W}iener theorem
  for locally compact {A}belian groups.
\newblock {\em J. Funct. Anal.}, 19:40--49, 1975.

\bibitem[PaSh15]{kupa15}
K.~{P}arthasarathy and N.~{S}hravan {K}umar.
\newblock {F}eichtinger's {S}egal algebra on homogeneous spaces.
\newblock {\em {I}nt. {J}. {M}ath.}, 26(8):9, 2015.

\bibitem[Pfa13]{pf13-1}
G.~E. {P}fander.
\newblock {S}ampling of operators.
\newblock {\em J. Fourier Anal. Appl.}, 19(3):612--650, 2013.

\bibitem[PfWa06]{pfwa06}
G.~E. {P}fander and D.~F. {W}alnut.
\newblock {M}easurement of time-variant channels.
\newblock {\em IEEE Trans. Inform. Theory}, 52(11):4808--4820, {N}ovember 2006.

\bibitem[Pog80]{po80-1}
D.~{P}oguntke.
\newblock {G}ewisse {S}egalsche {A}lgebren auf lokalkompakten {G}ruppen.
\newblock {\em Arch. Math. (Basel)}, 33:454--460, 1980.

\bibitem[Que89]{qu89}
F.~{Q}uehenberger.
\newblock {S}pektralsynthese und die {S}egalalgebra {$S_0(\mathbb{R}^m)$}.
\newblock Master's thesis, {U}niversity of {V}ienna, 1989.

\bibitem[Rei71]{re71}
H.~{R}eiter.
\newblock {\em $L^1$-algebras and {S}egal {A}lgebras}.
\newblock {S}pringer, {B}erlin, {H}eidelberg, {N}ew {Y}ork, 1971.

\bibitem[Rei78]{re78}
H.~{R}eiter.
\newblock {\"{U}}ber den {S}atz von {W}eil--{C}artier.
\newblock {\em Monatsh. Math.}, 86:13--62, 1978.

\bibitem[Rei89]{re89}
H.~{R}eiter.
\newblock {\em {M}etaplectic {G}roups and {S}egal {A}lgebras}.
\newblock {L}ect. {N}otes in {M}athematics. {S}pringer, {B}erlin, 1989.

\bibitem[Rei93]{re93-3} 
H.~{R}eiter. 
\newblock {O}n the {S}iegel-{W}eil formula. 
\newblock {\em Monatsh. Math.}, 116:299--330, 1993. 

\bibitem[ReSt00]{MR1802924}
H.~Reiter and J.~D. Stegeman.
\newblock {\em Classical harmonic analysis and locally compact groups},
  volume~22 of {\em London Mathematical Society Monographs. New Series}.
\newblock The Clarendon Press, Oxford University Press, New York, second
  edition, 2000.

\bibitem[Rie81]{ri81}
M.~A. {R}ieffel.
\newblock {C}*-algebras associated with irrational rotations.
\newblock {\em Pacific J. Math.}, 93:415--429, 1981.

\bibitem[Rie88]{ri88}
M.~A. {R}ieffel.
\newblock {P}rojective modules over higher-dimensional noncommutative tori.
\newblock {\em Canad. J. Math.}, 40(2):257--338, 1988.

\bibitem[Rud62]{ru62}
W.~{R}udin.
\newblock {\em {F}ourier {A}nalysis on {G}roups}.
\newblock {I}nterscience {P}ublishers, {N}ew {Y}ork, {L}ondon, 1962.

\bibitem[Rud91]{ru91}
W.~{R}udin.
\newblock {\em {F}unctional {A}nalysis 2nd ed}.
\newblock {I}nternational {S}eries in {P}ure and {A}pplied {M}athematics.
  {M}c{G}raw-{H}ill, {N}ew {Y}ork, 1991.

\bibitem[Rya02]{MR1888309}
R.~A. Ryan.
\newblock {\em Introduction to tensor products of {B}anach spaces}.
\newblock Springer Monographs in Mathematics. Springer-Verlag London, Ltd.,
  London, 2002.

\bibitem[Sch52]{MR0045307}
L.~Schwartz.
\newblock Th\'eorie des noyaux.
\newblock In {\em Proceedings of the {I}nternational {C}ongress of
  {M}athematicians, {C}ambridge, {M}ass., 1950, vol.1}, pages 220--230. Amer.
  Math. Soc., Providence, R. I., 1952.

\bibitem[Sch66]{MR0209834}
L.~Schwartz.
\newblock {\em Th\'eorie des distributions}.
\newblock Publications de l'Institut de Math\'ematique de l'Universit\'e de
  Strasbourg, No. IX-X. Nouvelle \'edition, enti\'erement corrig\'ee, refondue
  et augment\'ee. Hermann, Paris, 1966.

\bibitem[Spr07]{sp07}
{N}ico {S}pronk.
\newblock {O}perator space structure on {F}eichtinger's {S}egal algebra.
\newblock {\em J. Funct. Anal.}, 248:152--174, 2007.

\bibitem[Ste79]{st79}
J.~{S}tewart.
\newblock {F}ourier transforms of unbounded measures.
\newblock {\em Canad. J. Math.}, 31:1281--1292, 1979.

\bibitem[Tac94]{ta94}
K.~{T}achizawa.
\newblock {T}he boundedness of pseudodifferential operators on modulation
  spaces.
\newblock {\em Math. Nachr.}, 168:263--277, 1994.

\bibitem[Tof04a]{to04-3}
J.~{T}oft.
\newblock {C}ontinuity properties for modulation spaces, with applications to
  pseudo-differential calculus. {I}.
\newblock {\em J. Funct. Anal.}, 207(2):399--429, 2004.

\bibitem[Tof04b]{to04-2}
J.~{T}oft.
\newblock {C}ontinuity properties for modulation spaces, with applications to
  pseudo-differential calculus. {I}{I}.
\newblock {\em Ann. Global Anal. Geom.}, 26(1):73--106, 2004.

\bibitem[Tof10]{to10}
J.~{T}oft.
\newblock {P}seudo-differential operators with symbols in modulation spaces,
  2010.

\bibitem[ToWo$^{+}$07]{towozh07}
J.~{T}oft, M.~{W}ong, and H.~{Z}hu, editors.
\newblock {\em {M}odern {T}rends in {P}seudo-{D}ifferential {O}perators},
  volume 172.
\newblock 2007.

\bibitem[Tr{\`e}67]{tr67}
F.~{T}r{\`e}ves.
\newblock {\em {T}opological {V}ector {S}paces, {D}istributions and {K}ernels}.
\newblock Number~25 in {P}ure {A}ppl. {M}ath. {A}cademic {P}ress, {N}ew {Y}ork,
  1967.


\bibitem[UlRa11]{raul11}
T.~{U}llrich and H.~{R}auhut.
\newblock {G}eneralized coorbit space theory and inhomogeneous function spaces
  of {B}esov-{L}izorkin-{T}riebel type.
\newblock {\em J. Funct. Anal.}, 11:3299--3362, 2011.

\bibitem[Voi15]{vo15}
F.~{V}oigtlaender.
\newblock {\em {E}mbedding {T}heorems for {D}ecomposition {S}paces with
  {A}pplications to {W}avelet {C}oorbit {S}paces}.
\newblock PhD thesis, 2015.

\bibitem[Wah05]{wa05}
P.~{W}ahlberg.
\newblock {T}he random {W}igner distribution of {G}aussian stochastic processes with covariance in ${S}_0 ({R}^{2d})$.
\newblock {\em {J}. {F}unct. {S}paces {A}ppl.}, 3(2):163--181, 2005.

\bibitem[WaPf$^{+}$15]{kapfwa15}
D.~{W}alnut, G.~E. {P}fander, and T.~{K}ailath.
\newblock {C}ornerstones of sampling of operator theory.
\newblock {\em {A}r{X}iv e-prints}, feb 2015.

\bibitem[Wei64]{we64}
A.~{W}eil.
\newblock {S}ur certains groupes d'op{\'e}rateurs unitaires.
\newblock {\em Acta Math.}, 111:143--211, 1964.

\end{thebibliography}

}

\end{document}